\documentclass[reqno]{amsart}
\usepackage{amsmath, amssymb, amsthm, graphicx}

\font\dixmath=cmsy10
\theoremstyle{plain}
\newtheorem{prop}{Proposition}[section]
\newtheorem{coro}[prop]{Corollary}
\newtheorem{lemm}[prop]{Lemma}
\newtheorem{thrm}[prop]{Theorem}
\newtheorem*{thrm*}{Theorem}
\newtheorem*{thrm1}{Theorem 1}
\newtheorem*{thrm2}{Theorem 2}
\theoremstyle{definition}
\newtheorem{defi}[prop]{Definition}

\newtheorem{exam}[prop]{Example}
\newtheorem{rema}[prop]{Remark}

\newenvironment{conditions}{\setlength{\jot}{-1pt}\vspace{-5pt}\subequations\renewcommand{\theequation}{\theprop.\roman{equation}}\flalign}{\endflalign\endsubequations}

\numberwithin{equation}{section}

\renewcommand{\aa}[2]{a_{#1, #2}}

\newcommand{\aab}[2]{\underline{a}_{#1, #2}}

\newcommand{\BB}[1]{B_{#1}}
\newcommand{\bidual}{^{\hspace{-0.05em}\raise0.6pt\hbox{$\scriptscriptstyle+$}\hspace{-0.05em}*}}
\newcommand{\BKL}[1]{B_{#1}\bidual}
\newcommand{\BP}[1]{B_{#1}\plusexp}
\newcommand{\br}{\beta}
\newcommand{\brbr}{\gamma}
\newcommand{\brd}[1]{\mathrm{br}_{#1}}
\newcommand{\brr}{\beta'}
\newcommand{\cond}[1]{&\quad\text{-- #1}&}
\newcommand{\dang}[1]{D(#1)}
\newcommand{\ddd}[1]{\delta_{#1}}

\newcommand{\dddd}[2]{d_{#1,#2}}
\newcommand{\ddddb}[2]{\underline{d}_{#1,#2}}

\newcommand{\dpt}[1]{\mathrm{dp}_{#1}}
\newcommand{\equal}\equiv
\newcommand{\ff}[1]{\phi_{\hspace{-0.5pt}\raise-1pt\hbox{$\scriptstyle #1$}}}

\newcommand{\find}[1]{e(#1)}
\newcommand{\findi}[1]{f(#1)}

\newcommand{\flip}[1]{\Phi_{#1}}
\newcommand{\floor}[1]{F(#1)}
\renewcommand{\ge}{\geqslant}

\newcommand{\ie}{\emph{i.e.}}
\newcommand{\inv}{^{\minus\hspace{-0.1em}1}}
\newcommand{\last}[1]{{#1}^{\scriptscriptstyle\mathtt{\#}}}
\newcommand{\Ldots}{...\,} 
\renewcommand{\le}{\leqslant}
\newcommand{\len}[1]{\vert#1\vert}

\newcommand{\lenr}[2]{\vert#2\vert_{#1}}
\newcommand{\lens}[1]{{\parallel}#1{\parallel}_{\sigma}}
\newcommand{\minus}{\mathchoice{-}{-}{\raise0.7pt\hbox{$\scriptscriptstyle-$}\scriptstyle}{-}}
\newcommand{\newintegeri}[2]{
	\expandafter\def\csname #1\endcsname{#2}
	\expandafter\def\csname #1o\endcsname{{#2\minus1}}
	\expandafter\def\csname #1t\endcsname{{#2\minus2}}
	\expandafter\def\csname #1p\endcsname{{#2\plus1}}
	\expandafter\def\csname #1pp\endcsname{{#2\plus2}}}
\newcommand{\newintegerii}[1]{\newintegeri{#1#1}{#1}}
\newcommand{\nf}[1]{\mathrm{NF}_{#1}}
\newcommand{\NF}[1]{\underline{\mathrm{NF}}_{#1}}

\newcommand{\plus}{\mathchoice{+}{+}{\raise0.7pt\hbox{$\scriptscriptstyle+$}\scriptstyle}{+}}
\newcommand{\plusminus}{\mathchoice{\pm}{\pm}{\raise0.7pt\hbox{$\scriptscriptstyle\pm$}\scriptstyle}{\pm}}
\newcommand{\plusexp}{^{\raise0.8pt\hbox{$\hspace{-0.05em}\scriptscriptstyle+$}}}

\newcommand{\rev}[1]{{\curvearrowright}^{\scriptscriptstyle(#1)}}
\newcommand{\revv}{\curvearrowright}
\newcommand{\revi}[1]{R_{#1}}
\newcommand{\revii}[1]{R'_{#1}}
\newcommand{\reviii}[1]{R''_{#1}}
\newcommand{\sig}[1]{\sigma_{\!#1}} 
\newcommand{\siginv}[1]{\sigma_{\!#1}^{\hspace{-0.05em}\raise0.8pt\hbox{$\scriptscriptstyle-$}\hspace{-0.1em}1}}
\newcommand{\sigpm}[1]{\sigma_{\!#1}^{\hspace{-0.05em}\raise0.8pt\hbox{$\scriptscriptstyle\pm$}\hspace{-0.1em}1}}
\newcommand{\sigg}{\sigma}
\newcommand{\tail}{\mathrm{tail}}
\newcommand{\uu}{u}
\newcommand{\uuu}{\uu'}
\newcommand{\vv}{v}

\newcommand{\vvv}{v'}
\newcommand{\WBKL}[1]{\underline{B}_{#1}\bidual}
\newcommand{\ww}{w}
\newcommand{\wwb}{\underline{w}}
\newcommand{\wwl}[1]{\ww_{(#1)}}
\newcommand{\wwt}{\overline{w}}
\newcommand{\www}{\ww'}
\newcommand{\wwwb}{\underline{\ww}'}
\newcommand{\wwwt}{\overline{w'}}
\newcommand{\wwww}{\ww''}
\newcommand{\wwwwt}{\overline{w''}}
\newcommand{\wwh}{{\widehat\ww}}

\newcommand{\xx}{x}

\newcommand{\XX}{X}
\newcommand{\yy}{y}
\newintegeri{indi}{p}
\newintegeri{indii}{q}
\newintegeri{indiii}{r}
\newintegeri{indiv}{s}
\newintegeri{indv}{t}
\newintegeri{indvi}{t'}
\newintegeri{brdi}{b}
\newintegeri{brdii}{c}
\newintegeri{ll}{\ell}
\newintegerii{d}
\newintegerii{h}
\newintegerii{i}
\newintegerii{j}
\newintegerii{k}
\newintegerii{m}
\newintegerii{n}
\newintegerii{p}
\newintegerii{q}
\newintegerii{r}
\newintegerii{s}
\newintegerii{t}

\begin{document} 

\title{Every braid admits a short sigma-definite representative}
\author{Jean Fromentin} 
\address{Laboratoire de Math\'ematiques Nicolas Oresme,
  UMR 6139 CNRS, Universit\'e de Caen BP 5186, 14032 Caen, France}
\email{jean.fromentin@math.unicaen.fr}
\maketitle

\begin{abstract}
A result by Dehornoy (1992) says that every nontrivial braid 
admits a $\sigg$-definite word representative, defined as a braid word in which the generator
$\sig\ii$ with maximal index $\ii$ appears with
exponents that are all positive, or all negative. This is the
ground result for ordering braids. In this paper, we enhance
this result and prove that every braid admits  a
$\sigg$-definite word  representative that, in addition, is
quasi-geodesic. This establishes a longstanding
conjecture.  Our proof uses the dual braid monoid and a new
normal form called the rotating normal form.
\end{abstract}

%
%

It is known since~\cite{Dehornoy:LD} that Artin's braid
groups are orderable, by an ordering  that enjoys many remarkable
properties~\cite{DehornoyDynnikovRolfsenWiest}.  The key point in the existence of this ordering is
the property that every nontrivial braid admits a
$\sigg$-definite representative, defined to be a braid
word~$\ww$ in the standard Artin generators~$\sig\ii$ in
which the generator~$\sig\ii$ with highest index~$\ii$
occurs only positively (no $\siginv\ii$), in which case
$\ww$ is called \emph{$\sigg$-positive}, or only negatively
(no $\sig\ii$), in which case $\ww$ is called
\emph{$\sigg$-negative}. 
For $\br$ a braid, let $\lens\br$ denote the length of the
shortest expression of~$\br$ in terms of the Artin generators
$\sigpm1$. Our main goal in this paper is to prove the
following result.

\begin{thrm1} 
Each $\nn$-strand braid~$\br$ admits a $\sigg$-definite expression of length at most $6\,(\nno)^2\,\lens\br$.
\end{thrm1}

Theorem 1 answers a puzzling open question in the theory of
braids. Indeed,
the problem of finding a short $\sigma$-definite
representative word for every braid has an already long
history. In the  past two decades, at least five or six different
proofs of the existence of such $\sigg$-definite
representatives have been given. The first one by Dehornoy
in 1992 relies on self-distributive
algebra~\cite{Dehornoy:LD}. The next one, by
Larue~\cite{Larue}, uses the Artin representation of braids
as automorphisms of a free groups, an argument that was
independently rediscovered by
Fenn--Greene--Rolfsen--Rourke--Wiest \cite{FennGreenRolfsenRourkeWiest} in a
topological language of so-called curve diagrams. A
completely different proof based on the geometry of the
Cayley graph of~$B_\nn$ and on Garside's theory appears
in \cite{Dehornoy:FA}. Further methods have been proposed
in connection with relaxation algorithms, which are
strategies for inductively simplifying some geometric object
associated with the considered braid, typically a family of
closed curves drawn in a punctured disk. Both the methods
of Dynnikov--Wiest in
\cite{DynnikovWiest} and of Bressaud in~\cite{Bressaud} lead to $\sigg$-definite representatives.
However, a frustrating feature of all the above methods is that, when one starts with a braid
word~$\ww$ of length~$\ll$, one obtains in the best case the existence of a
$\sigg$-definite word~$\ww'$ equivalent to~$\ww$ whose length is bounded above by an
exponential in~$\ll$---in the cases of \cite{Larue, FennGreenRolfsenRourkeWiest,
Dehornoy:FA, DynnikovWiest, Bressaud}, the original method of~\cite{Dehornoy:LD} is much
worse. By contrast, experiments, specially those based on the algorithms derived
from~\cite{Dehornoy:FA} and~\cite{DynnikovWiest}, strongly suggested the existence of short
$\sigg$-definite representatives, making it natural to conjecture that every braid word of
length~$\ll$ is equivalent to a $\sigg$-definite word of length~$O(\ll)$. This is what Theorem~1
establishes. It is fair to mention that the method of~\cite{DynnikovWiest} proves the existence of
``relatively short $\sigg$-definite representatives''. Indeed, it provides for every length~$\ell$
braid word a $\sigg$-definite equivalent word whose length with respect to some
conveniently extended alphabet lies in~$O(\ll)$. However, when the output word is translated
back to the alphabet of Artin's generators~$\sig\ii$, the only upper bound Dynnikov and Wiest
could deduce so far is exponential in~$\ll$.

The statement of Theorem~1 is essentially optimal. 
Indeed, it is observed in~\cite[Chapter XVI]{DehornoyDynnikovRolfsenWiest} that the
length~$4(\nn-2)$ braid word
\[
\sig\nno \sig\nnt^{-2} \, ... \, \sig2^{-2e} \sig1^{2e}
\sig2^{2e}  \,...\,  \sig\nnt^2\siginv\nno,
\]
with $e = \pm1$ according to the parity of~$\nn$, is
equivalent to no $\sigg$-definite word of length smaller than
$\nn^2 - \nn -2$. Thus, in any case, the factor~$(\nn-1)^2$ of Theorem~1 could not be
possibly replaced with a factor less than~$O(\nn)$.

Our proof of Theorem~1 is effective, and it directly leads to an algorithm that returns, for every $\nn$-strand
braid~$\br$, a distinguished $\sigg$-definite
word~$\NF\nn(\br)$  that represents~$\br$. Analyzing the
complexity of this algorithm leads to

\begin{thrm2}
There exists an effective algorithm which, for each $\nn$-strand braid 
specified by a word of length $\ll$, computes the
$\sigg$-definite word~$\NF\nn(\br)$ in $O(\ll^2)$ steps.
\end{thrm2}

We prove Theorems~1 and~2 using the dual braid monoid~$\BKL\nn$ associated with the Birman--Ko--Lee generators 
and introducing a new normal form on~$\BKL\nn$,
called the rotating normal form, which is analogous to the
alternating normal form of~\cite{Burckel:WO}
and~\cite{Dehornoy:AF}. The rotating normal form is based
on the $\ff\nn$-splitting operation, a natural way of
expressing every $\nn$-strand dual braid in terms of a finite
sequence of ($\nno$)-strand dual braids. 

The principle of the argument is as follows.
Given a $\nn$-strand braid~$\br$, we first express it as a
fraction $\ddd\nn^{-\tt}\,\brr$, where $\ddd\nn$ is the
Garside element of the monoid~$\BKL\nn$ and
$\brr$ belongs to~$\BKL\nn$. If the exponent~$\tt$ happens to
be greater than the length of the
above-mentioned
$\ff\nn$-splitting of~$\brr$, then the $\sigg$-negative factor
$\ddd\nn^{-\tt}$ wins over the $\sigg$-positive factor~$\brr$,
and a
$\sigg$-negative word representing~$\br$ can be obtained
by an easy direct computation. Otherwise, we determine the
rotating normal form~$\ww$ of~$\brr$ and try to find
a
$\sigg$-positive representative of~$\br$ by pushing the negative
factor~$\ddd\nn^{-\tt}$ to the right through the positive
part~$\ww$. The process is incremental. The problem is that
certain special $\sigg$-negative words, called dangerous,
appear in the process. The key point is that rotating normal
words satisfy some syntactic conditions that enable them to
neutralize dangerous words. In this way, one finally obtains a word
representative of~$\br$ that contains no~$\sig\nno\inv$, hence
is either $\sigg$-positive, or involves no~$\sig\nno$ at all. An
induction on the braid index~$\nn$ then allows one to conclude.

The basic step of the above process consists in switching one
dangerous factor  and one rotating normal word. This step
increases the length by a multiplicative factor~$3$ at most,
and this is the way the length and time upper bounds of
Theorems~1 and~2 arise.

In this paper, the braid ordering is not used---in contrary, the
existence of the latter can be (re)-deduced from our current
results. However, the braid ordering is present behind our
approach. What actually explains the existence of our normal
form is the connection between the rotating normal form of
Section~\ref{S:NormalForm} and the restriction of the braid ordering to the
dual braid monoid, which is sketched in~\cite{F}.

The paper is organized as follows. In
Section~\ref{S:DualBraidMonoids}, we briefly recall the definition
of the dual braid monoids and the properties of these monoids
that are needed in the sequel, in particular those connected with
the Garside structure. In Section~\ref{S:NormalForm}, we
introduce the rotating normal form, which is our new normal form
on~$\BKL\nn$. In Section~\ref{S:Ladders}, we establish syntactic
constraints about rotating normal words, namely that every normal word is
what we call a ladder. In Section~\ref{S:Reversing}, we introduce the
notion of a dangerous braid word and define the so-called
reversing algorithm, which transforms each word consisting of a
dangerous word followed by a ladder into a particular type of
$\sigg$-definite word called a wall. In Section~\ref{S:Walls} we
compute the complexity of the above reversing algorithm.
Finally, we put all pieces together and establish Theorems~1 and~2
in Section~\ref{S:TheMainResult}.

%
%

\section{Dual braid monoids}
\label{S:DualBraidMonoids} 

Our first ingredient for investigating braids will be the Garside
structure of the so-called dual braid monoid~$\BKL\nn$. Here we
recall the needed definitions and results.

%
%

\subsection{Birman--Ko--Lee generators}

We recall that Artin's braid group~$\BB\nn$ is
defined for
$\nn\ge2$ by the presentation
\begin{equation}
\label{E:BnPresentation}
\left<\sig1,\Ldots,\sig\nno;\begin{array}{cl} \sig\ii\sig\jj\,=\,\sig\jj\sig\ii & \text{for $|\ii\minus\jj|\ge2$}\\ \sig\ii\sig\jj\sig\ii\,=\,\sig\jj\sig\ii\sig\jj & \text{for $|\ii\minus\jj|=1$}  \end{array}\right>.
\end{equation}

The submonoid of~$\BB\nn$ generated by $\{\sig1, \Ldots,
\sig\nno\}$ is denoted by~$\BP\nn$, and its elements are called
\emph{positive braids}. As is well known, the monoid~$\BP\nn$ equipped with
Garside's fundamental braid~$\Delta_\nn$ has the structure of
what is now usually called a Garside monoid~\cite{Garside,
Dehornoy:GG}.

The \emph{dual braid monoid} is another submonoid
of~$\BB\nn$. It is generated by a subset of~$\BB\nn$
that properly includes $\{\sig1, \Ldots,
\sig\nno\}$, and consists of the so-called \emph{Birman--Ko--Lee
generators} introduced in~\cite{BirmanKoLee}.

\begin{defi}\label{D:DualGenerator}
(See Figure~\ref{F:DualGenerator}.)
For $1 \le \indi < \indii$, we put
\begin{equation}
\label{E:DualGenerator}
\aa\indi\indii = \sig\indi ... \sig\indiit \, \sig\indiio \, \siginv\indiit ... \siginv\indi.
\end{equation}
\end{defi}

\begin{figure}[htb]
\begin{picture}(80,19) 
\put(5,0){\includegraphics{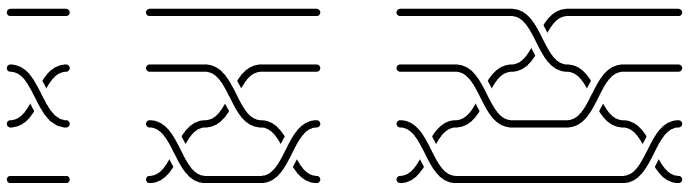}}
{\small
\put(3,0){$1$}
\put(3,6){$2$}
\put(3,12){$3$}
\put(3,17.5){$4$}}
\end{picture}
\caption{\sf \smaller From the left to the right : diagram of the braids   $\aa23
(=\nobreak\sig2)$, $\aa13 (=\sig1\sig2\siginv1)$ and $\aa14
(=\sig1\sig2\sig3\siginv2\siginv1)$. The generator
$\aa\indi\indii$ corresponds to the half-twist where the~$\indii$th strand
crosses over the $\indi$th strand, both remaining under all
intermediate strands.}
\label{F:DualGenerator}
\end{figure}

\begin{rema}
In~\cite{BirmanKoLee}, $\aa\indi\indii$ is defined to be  $\sig\indiio... \sig\indip \, \sig\indi \, \siginv\indip ... \siginv\indiio$, \ie, it
corresponds to the strands at positions~$\indi$ and~$\indii$
passing in front of all intermediate strands, not behind.  Both
options lead to isomorphic monoids, but our choice is the only
one that naturally leads to the suitable embedding
of~$\BKL\nno$ into~$\BKL\nn$.
\end{rema}

The family of all braids~$\aa\indi\indii$ enjoys nice invariance
properties with respect to cyclic permutations of the indices,
which are better visualized when $\aa\indi\indii$ is represented
on a cylinder---see Figure~\ref{F:Cylinder}. Then, it is natural to
associate with~$\aa\indi\indii$ the chord connecting the
vertices~$\indi$ and~$\indii$ in a circle with $\nn$~marked
vertices \cite{BessisDigneMichel}.

\begin{figure}[htb]
\begin{picture}(89,23) 
\put(2,1){\includegraphics{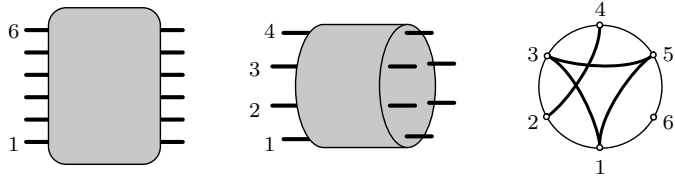}}
{\footnotesize
\put(0,3){$1$}
\put(0,18){$6$}
\put(34,3){$1$}
\put(32,7.5){$2$}
\put(32,13){$3$}
\put(34,18){$4$}
\put(78,0){$1$}
\put(69,6){$2$}
\put(69,15){$3$}
\put(78,21){$4$}
\put(87,15){$5$}
\put(87,6){$6$}}
\end{picture}
\caption{{\sf \smaller Rolling up the usual diagram helps up to visualize the
symmetries of the braids $\aa\indi\indii$. On the resulting cylinder,
$\aa\indi\indii$ naturally corresponds to the chord connecting the vertices
$\indi$ and
$\indii$.}}
\label{F:Cylinder}
\end{figure}

Hereafter, we write $[\indi, \indii]$ for the interval~$\{\indi, \Ldots,
\indii\}$ of~$\mathbb{N}$, and we say that $[\indi, \indii]$ is
\emph{nested} in~$[\indiii,\indiv]$ if we have $\indiii < \indi
<\indii < \indiv$. A nicely symmetric presentation of~$\BB\nn$ in terms of the generators~$\aa\indi\indii$ is as follows.

\begin{lemm}\cite{BirmanKoLee}
\label{L:DualRelations}
In terms of the~$\aa\indi\indii$, the group~$\BB\nn$ is presented by the relations
\begin{gather}
\label{E:DualCommutativeRelation}
\aa\indi\indii\aa\indiii\indiv= \aa\indiii\indiv\aa\indi\indii \text{\quad for $[\indi, \indii]$ and $[\indiii, \indiv]$ disjoint or nested},\\
\label{E:DualNonCommutativeRelation}
\aa\indi\indii\aa\indii\indiii= \aa\indii\indiii\aa\indi\indiii =\aa\indi\indiii\aa\indi\indii \text{\quad for $1 \le \indi<\indii<\indiii\le\nn$}.
\end{gather}
\end{lemm}

In the representation of
Figure~\ref{F:Cylinder}, the relations of
type~\eqref{E:DualCommutativeRelation} mean
that, in each chord triangle, the product of two adjacent edges
taken clockwise does not depend on the edges: for instance, the triangle~$(1, 3, 5)$ gives
$\aa13\aa35 = \aa35\aa15 = \aa15\aa13$.  Relations of
type~\eqref{E:DualNonCommutativeRelation} say that the
generators associated with non-intersecting chords commute: for
instance, on Figure~\ref{F:Cylinder}, we read that $\aa24$
and~$\aa15$ commute---but, for instance, nothing is
claimed about~$\aa24$ and~$\aa13$.

%
%

\subsection{The dual braid monoid~$\BKL\nn$ and 
its Garside structure}
\label{SS:DualBraidMonoid}

By definition, we have $\sig\indi  = \aa\indi{\indip}$  for each~$\indi$: every Artin generator is a Birman--Ko--Lee
generator. On the other hand, the braid~$\aa13$ belongs to no
monoid~$\BP\nn$. Hence, for $\nn \ge 3$, the submonoid
of~$\BB\nn$ generated by the Birman--Ko--Lee
braids~$\aa\indi\indii$ is a proper extension of~$\BP\nn$: this
submonoid is what is called the dual braid monoid.

 \begin{defi}
\label{D:DualBraidMonoid}
For $\nn \ge 2$, the \emph{dual braid monoid}~$\BKL\nn$ is defined to be
the submonoid of~$\BB\nn$ generated by the braids~$\aa\indi\indii$ with~$1
\le \indi < \indii \le\nobreak \nn$.  
\end{defi}

So, every positive $\nn$-strand braid belongs to~$\BKL\nn$, but the converse
is not true for $\nn \ge 3$: the braid~$\aa13$, \ie,
$\sig1\sig2\siginv1$, belongs to~$\BKL3$ but not  to~$\BP3$.

\begin{prop} \cite{BirmanKoLee}
\label{P:Garside}
For each~$\nn$, the relations of Lemma~\ref{L:DualRelations} make a
presentation of~$\BKL\nn$ in terms of the generators~$\aa\pp\qq$, and
$\BKL\nn$ is a Garside monoid with Garside element
\begin{equation}
\label{E:DualGarsideElement}
\ddd\nn=\aa12\,\aa23\,...\,\aa\nno\nn \ (\  = \sig1\,\sig2 \,...\, \sig\nno\ ).
\end{equation}
\end{prop}

Proposition~\ref{P:Garside} implies that the left and right-divisibility relations
in the dual braid monoid~$\BKL\nn$ have lattice properties, \ie, that any two
elements of~$\BKL\nn$ admit (left and right) greatest common divisors and
least common multiples. It also implies that
$\BB\nn$ is a group of fractions for the monoid~$\BKL\nn$, and that every
element of~$\BKL\nn$ admits  a distinguished decomposition similar to the
greedy normal form of~$\BP\nn$~\cite{BirmanKoLee}.  This decomposition involves
the so-called simple elements of~$\BKL\nn$, which are the divisors
of~$\ddd\nn$, and are in one-to-one correspondence with the non-crossing
partitions  of~$\{1, ..., \nn\}$ \cite{BirmanKoLee, Bessis}.

\subsection{The rotating automorphism}
\label{SS:CyclingAutomorphism}

An important role in the sequel will be played by the
so-called \emph{rotating automorphism}~$\ff\nn$
of~$\BKL\nn$. In every Garside monoid, conjugating under the
Garside element defines an automorphism~\cite{Dehornoy:GG}.
In the case of the monoid~$\BP\nn$ and its Garside
element~$\Delta_\nn$, the associated automorphism is the flip
automorphism that exchanges~$\sig\ii$ and~$\sig{\nn-\ii}$ for
each~$\ii$, thus an involution that corresponds to a symmetry in
braid diagrams. In the case of the dual monoid~$\BKL\nn$ and
its Garside element~$\ddd\nn$, the associated automorphism has
order~$\nn$, and it is similar to a rotation.

\begin{lemm}
\label{L:Rotation}
(See Figure~\ref{F:Rotation}.)
For each~$\br$ in~$\BKL\nn$, let $\ff\nn(\br)$ be defined by
\begin{equation}
\ddd\nn \, \br = \ff\nn(\br) \, \ddd\nn.
\end{equation}
Then, for all~$\indi, \indii$ with $1\le\indi <\indii\le\nn$, we have
\begin{equation}
\label{E:Rotation}
\ff\nn(\aa\indi\indii) = 
\begin{cases}
\aa{\indip}{\indiip}&\text{for $\indii \le \nno$},\\
\aa1{\indip}&\text{for $\indii = \nn$}.
\quad\qquad
\end{cases}
\end{equation}
\end{lemm}

The proof is an easy verification from~\eqref{E:DualGenerator}, \eqref{E:DualGarsideElement}
and the relations~\eqref{E:DualCommutativeRelation},
\eqref{E:DualNonCommutativeRelation}.  Note that the relation
$\ff\nn(\aa\indi\indii) = \aa{\indip}{\indiip}$ always holds
provided the indices are taken mod~$\nn$ and possibly switched so that, for
instance,  $\aa{\indip}{\nnp}$ means~$\aa1{\indip}$.

\begin{figure}[htb]
\begin{picture}(53,23)
\put(1.5,2.5){\includegraphics{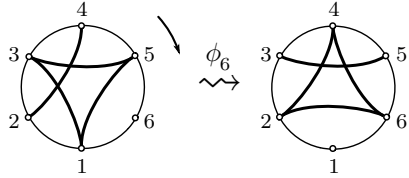}}
{\footnotesize
\put(9,0){$1$}
\put(0,6){$2$}
\put(0,15){$3$}
\put(9,21){$4$}
\put(18,15){$5$}
\put(18,6){$6$}
\put(42.5,0){$1$}
\put(33.5,6){$2$}
\put(33.5,15){$3$}
\put(42.5,21){$4$}
\put(51.5,15){$5$}
\put(51.5,6){$6$}}
\put(25,10.5){\huge $\leadsto$}
\put(25,15){ $\ff6$}
\end{picture}
\caption{{\sf \smaller Representation  of the rotating automorphism~$\ff\nn$ as a clockwise
rotation
 of the marked circle by~$2\pi/\nn$}.}
\label{F:Rotation}
\end{figure}

The formulas of~\eqref{E:Rotation} show that $\BKL\nn$ is globally
invariant under~$\ff\nn$. By contrast, note that $\BKL\nn$ is \emph{not}
invariant under the flip automorphism~$\flip\nn$: for instance,
$\flip3(\aa13)$, which is~$\sig2\sig1\siginv2$, does not belong to~$\BKL3$.

%
%

\section{The rotating normal form}
\label{S:NormalForm}

Besides the Garside structure, the main tool we shall use in
this paper is a new normal form for the elements of the dual
braid monoid~$\BKL\nn$, \ie, a new way of associating with
every element of~$\BKL\nn$ a distinguished word (in the
letters~$\aa\indi\indii$) that represents it.  This normal form is
called the rotating normal form, as it relies on the rotating
automorphism~$\ff\nn$ which we have seen is similar to a
rotation.

The rotating normal form is reminiscent of the alternating normal
form introduced in~\cite{Dehornoy:AF} for the case of the
monoid~$\BP\nn$---which is itself connected with Burckel's
approach of~\cite{Burckel:WO}. It is also closely connected with the normal
forms introduced in~\cite{Ito}, which are other developments, in a
different direction, of the alternating normal form.
As the properties of~$\BKL\nn$ and~$\ff\nn$ are essentially the same as those
of~$\BP\nn$ and~$\flip\nn$, adapting the results of~\cite{Dehornoy:AF} is
easy and, therefore, constructing the rotating normal form is
not very hard---what will be harder is identifying the needed
properties of rotating normal words, as will be done
in subsequent sections.

%
%

\subsection{The $\ff\nn$-splitting}

The basic observation of~\cite{Dehornoy:AF} is that
each braid  in the monoid~$\BP\nn$ admits a unique maximal right-divisor
that lies in the submonoid~$\BP\nno$.  A similar phenomenon occurs in
the dual monoid~$\BKL\nn$.

\begin{lemm}
\label{L:Tail}
For $\nn\ge3$, every braid~$\br$ of~$\BKL\nn$ admits a 
maximal right-divisor lying in~$\BKL\nno$. 
The latter is the unique right-divisor~$\br_1$
of~$\br$ such that $\br\br_1\inv$ has no  nontrivial $($\ie,
$\not=1)$ right-divisor lying in~$\BKL\nno$. 
\end{lemm}

\begin{proof}
The submonoid~$\BKL\nno$ of~$\BKL\nn$ is closed under right-divisor and left-lcm. 
Hence we can apply Lemma 1.12 of~\cite{Dehornoy:AF}.
\end{proof}

\begin{defi}
The braid~$\br_1$ of Lemma~\ref{L:Tail} is called the 
\emph{$\BKL\nno$-tail} of~$\br$ and it is denoted by $\tail_\nno(\br)$.
\end{defi}

\begin{exam}
\label{X:Tail}
Let us compute the~$\BKL2$-tail of~$\ddd3^2$. 
As $\BKL2$ is generated by~$\aa12$, this $\BKL2$-tail is the maximal power
of~$\aa12$  that right-divides~$\ddd3^2$.  By definition, we have $\ddd3^2 =
\aa12\aa23\aa12\aa23$.  By
applying~\eqref{E:DualNonCommutativeRelation}  twice, we obtain  \[\ddd3^2
=
\aa12\aa23\aa13\aa12 = \aa12\aa13\aa12^2.\] As the word~$\aa12\aa13$ is alone
in its equivalence class, the braid it represents cannot be right-divisible
by~$\aa12$.  Therefore, the~$\BKL2$-tail of~$\ddd3^2$ is $\aa12^2$.
\end{exam}

In the context of the monoid~$\BP\nn$, one obtains a distinguished
decomposition for every braid in~$\BP\nn$ by considering the~$\BP\nno$-tail
and the~$\flip\nn\big(\BP\nno)$-tail alternatively, which is possible because
$\BP\nn$ is generated by $\BP\nno$ and $\flip\nn(\BP\nno)$. In our context
of~$\BKL\nn$, we shall  use the~$\BKL\nno$-tail,
the~$\ff\nn\big(\BKL\nno)$-tail,
$\Ldots$,  the~$\ff\nn^\nno\big(\BKL\nno)$-tail cyclically to obtain a
distinguished decomposition for every braid of~$\BKL\nn$.

In order to show that every braid in~$\BKL\nn$ admits such a
decomposition, we must check that the images of~$\BKL\nno$
under the powers of~$\ff\nn$ cover~$\BKL\nn$. Actually, iterating twice is
enough.

\begin{lemm}
\label{L:Covering}
For $\nn\ge3$, every generator~$\aa\indi\indii$ of~$\BKL\nn$ belongs
to~$\BKL\nno\cup\ff\nn\big(\BKL\nno\big)\cup\ff\nn^2\big(\BKL\nno\big)$.
\end{lemm}

\begin{proof}
For  $\indii\le\nno$, the braid~$\aa\indi\indii$ belongs to~$\BKL\nno$. 
Next, for $\indii = \nn$ and $\indi\ge 2$, we have $\aa\indi\nn = \ff\nn(\aa\indio\nno)$, which belongs to~$\ff\nn\big(\BKL\nno\big)$. 
Finally, for $\indi = 1$ and $\indii = \nn$, we find $\aa\indi\indii = \ff\nn(\aa\nno\nn) = \ff\nn^2(\aa\nnt\nno)$, which belongs to~$\ff\nn^2\big(\BKL\nno)$.
\end{proof}

By iterating the tail construction, we then associate with every braid
of~$\BKL\nn$ a finite sequence of braids of~$\BKL\nno$ that specifies it completely.

\begin{prop}
\label{P:Splitting}
Assume $\nn\ge3$. Then, for each nontrivial braid~$\br$ of~$\BKL\nn$,
there exists a unique sequence $(\br_\brdi,\Ldots,\br_1)$ in $\BKL\nno$
satisfying
$\br_\brdi \not= 1$ and
\begin{gather}
\label{E:SplittingDecomposition}
\br=\ff\nn^\brdio(\br_\brdi)\cdot...\cdot\ff\nn(\br_2)\cdot\br_1,\\
\label{E:SplittingConditionV1}
\text{for each~$\kk\ge1$, the braid~$\br_\kk$ is the~$\BKL\nno$-tail of~$\ff\nn^{\brdi\minus\kk}(\br_\brdi)\cdot...\cdot\br_\kk$}.
\end{gather}
\end{prop}

\begin{proof}
Starting from $\br^{(0)}=\br$, we define two sequences, denoted $\br^{(\kk)}$ and $\br_\kk$, by
\begin{equation}
\label{E:P:Splitting:1}
\br^{(\kk)}=\ff\nn\inv\big(\br^{(\kko)}\,\br_\kk\inv\big)\quad\text{and}\quad\br_\kk = \tail_\nno(\br^{(\kko)})\quad\text{for}\quad\kk\ge1.
\end{equation}
Using induction on $\kk\ge1$, we prove the relations
\begin{gather}
\label{E:P:Splitting:2} \br=\ff\nn^\kk(\br^{(\kk)})\cdot\ff\nn^\kko(\br_\kk)\cdot...\cdot\br_1,\\
\label{E:P:Splitting:3} \tail_\nno\big(\ff\nn\big(\br^{(\kk)}\big)\big) = 1.
\end{gather}
Assume $\kk=1$. 
Lemma~\ref{L:Tail} implies that the~$\BKL\nno$-tail of~$\br\,\br_1\inv$ is trivial. 
Then, as $\ff\nn\big(\br^{(1)}\big)$ is equal to~$\br\,\br_1\inv$, the~$\BKL\nno$-tail of~$\ff\nn\big(\br^{(1)}\big)$ is trivial, and the relation $\br=\ff\nn(\br^{(1)})\cdot\br_1$ holds.
Assume $\kk\ge2$.
By construction of~$\br^{(\kk)}$, we have $\ff\nn(\br^{(\kk)})=\br^{(\kko)}\,\br_\kk\inv$, hence $\br^{(\kko)}=\ff\nn(\br^{(\kk)})\cdot\br_\kk$.  
Then we have the relation
\begin{equation}
\label{E:P:Splitting:4}
\ff\nn^\kko\big(\br^{(\kko)})=\ff\nn^\kk(\br^{(\kk)})\cdot\ff\nn^\kko\big(\br_\kk\big).
\end{equation}
On the other hand, by induction hypothesis, we have
\begin{equation}
\label{E:P:Splitting:5}
\br=\ff\nn^\kko(\br^{(\kko)})\cdot\ff\nn^\kkt(\br_\kko)\cdot...\cdot\br_1.
\end{equation}
 Substituting \eqref{E:P:Splitting:4} in \eqref{E:P:Splitting:5}, we obtain \eqref{E:P:Splitting:2}.
As $\br_\kk$ is the~$\BKL\nno$-tail of~$\br^{(\kko)}$, Lemma~\ref{L:Tail} gives \eqref{E:P:Splitting:3}.

By construction, the sequence of right-divisors of~$\br$,
\[\br_1,  \ \ff\nn(\br_2)\br_1, \ \ff\nn^2(\br_3)\ff\nn(\br_2)\br_1, \ ...\]
is non-decreasing for divisibility, and, therefore, for length reasons, it must be
eventually constant.  Hence, by right cancellativity of~$\BKL\nn$, there exists
$\brdi$ such that for
$\kk\ge\brdi$, we have
$\ff\nn^\kko(\br_\kk)\cdot...\cdot\br_1=\ff\nn^\brdio(\br_\brdi)\cdot...\cdot\br_1$. 
Then \eqref{E:P:Splitting:2} implies 
\[\br=\ff\nn^\brdi(\br^{(\brdi)})\ff\nn^\brdio(\br_\brdi)\cdot...\cdot\br_1,\]
with $\br_\brdi \not= 1$ whenever $\brdi$ is chosen to be minimal.

By definition of~$\brdi$, we have $\br_\kk = 1$, and therefore
$\ff\nn(\br^{(\kk)})=\br^{(\kko)}$ by~\eqref{E:P:Splitting:1}, for
$\kk\ge\brdip$. Then, we have
$\br^{(\brdi)}=\ff\nn(\br^{(\brdip)})$, 
$\ff\nn\inv(\br^{(\brdi)})=\ff\nn(\br^{(\brdi\plus2)})$ and
$\ff\nn^{\minus2}(\br^{(\brdi)})=\ff\nn(\br^{(\brdi\plus3)})$.
By~\eqref{E:P:Splitting:3}, the~$\BKL\nno$-tails of~$\br^{(\brdi)}$, 
$\ff\nn\inv(\br^{(\brdi)})$ and $\ff\nn^{\minus2}(\br^{(\brdi)})$ are trivial.
Hence, for every generator $\xx$ of $\BKL\nno$, the braid $\br^{(\brdi)}$ is
not right-divisible by $\xx$, nor is it either by $\ff\nn(\xx)$ or by
$\ff\nn^2(\xx)$. Then Lemma~\ref{L:Covering} implies that $\br^{(\brdi)}$ is
right-divisible by no $\aa\indi\indii$ with $1\le\indi<\indii \le \nn$, \ie, we
have
$\br^{(\brdi)} = 1$,
whence $\br=\ff\nn^\brdio(\br_\brdi)\cdot...\cdot\br_1$.

We prove now the uniqueness of $(\br_\brdi,\Ldots,\br_1)$.
Let $\ff\nn^{\brdiio}(\brbr_\brdii)\cdot...\cdot\ff\nn(\brbr_2)\cdot\ff\nn(\brbr_1)$ be a decomposition of~$\br$ satisfying $\brbr_\brdii\not=1$ and $\brbr_\kk=\tail_\nno(\ff\nn^{\brdii-\kk}(\brbr_\brdii)\cdot ...\cdot\brbr_\kk)$ for each $\kk\ge1$.
Using an induction on $\kk\ge1$, we prove the relations
\[\brbr_\kk=\br_\kk\quad\text{and}\quad\ff\nn^{\brdii\minus\kk\minus1}(\brbr_\brdii)\cdot ...\cdot\ff\nn(\brbr_\kkpp)\cdot\brbr_\kkp=\br^{(\kk)}.\] 
For $\kk=1$, by hypothesis, we have  $\br=\big(\ff\nn^\brdiio(\brbr_\brdii)\cdot ...\cdot\ff\nn(\brbr_2)\big)\cdot\brbr_1$, where $\brbr_1$ is the $\BKL\nno$-tail of $\br$, hence, by Lemma~\ref{L:Tail}, we have $\br_1=\brbr_1$ and $\br^{(1)}=\ff\nn^{\brdiit}(\brbr_\brdii)\cdot ...\cdot\ff\nn(\brbr_3)\cdot\brbr_2$.
By induction hypothesis, we have 
\[\big(\ff\nn^{\brdii\minus\kk\minus1}(\brbr_\brdii)\cdot ...\cdot\ff\nn(\brbr_\kkpp)\big)\cdot\brbr_\kkp=\br^{(\kk)},\]
 and by hypothesis about $\brbr_\kkp$, the braid $\brbr_\kkp$ is the $\BKL\nno$-tail of $\br^{(\kk)}$.
Then, by Lemma~\ref{L:Tail} again, we have $\brbr_\kkp=\br_\kkp$ and $\ff\nn^{\brdii\minus\kk\minus2}(\brbr_\brdii)\cdot ...\cdot\ff\nn(\brbr_{\kk+3})\cdot\brbr_\kkpp=\br^{(\kkp)}$.
We proved $\brbr_\kk = \br_\kk$ for
$\brdi\ge\kk\ge1$, hence we find
\begin{equation}
\label{E:P:Splitting:6}
\ff\nn^{\brdii\minus\brdi\minus1}(\brbr_\brdii)\cdot ...\cdot\ff\nn(\brbr_\brdipp)\cdot\brbr_\brdip=\br^{(\brdi)}
\end{equation} 
By definition of $\brdi$, we have $\br^{(\brdi)} = 1$, whereas, by hypothesis,
the braid $\brbr_\brdii$ is nontrivial. So \eqref{E:P:Splitting:6} may hold only
for $\brdii=\brdi$. 
\end{proof}

\begin{defi}
\label{D:Splitting}
The sequence $(\br_\brdi,\Ldots, \br_1)$ of  Proposition~\ref{P:Splitting} is
called the \emph{$\ff\nn$-splitting of~$\br$}.  Its length, \ie, the
parameter~$\brdi$, is called the \emph{$\nn$-breadth} of~$\br$.
\end{defi}

The idea of the~$\ff\nn$-splitting is very simple: starting 
with a braid~$\br$ of~$\BKL\nn$, we extract the maximal right-divisor that lies in~$\BKL\nno$, \ie, that leaves the~$\nn$th
strand unbraided, then we extract the maximal right-divisor of the
remainder that leaves the first strand unbraided, and so on
rotating by~$2\pi/\nn$ at each step---see Figure~\ref{F:Splitting}.

\begin{figure}[htb]
\begin{picture}(104,28)
\put(0,0){\includegraphics{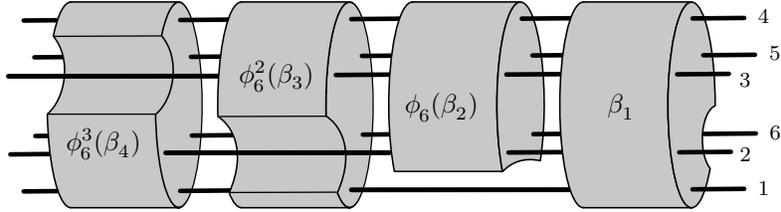}}
\put(80,13){$\br_1$} \put(53,13){$\ff6(\br_2)$}
\put(31,17){$\ff6^2(\br_3)$} \put(8,8){$\ff6^3(\br_4)$}
{\footnotesize
\put(101.55,9.5){$6$} \put(101.5,20){$5$}
\put(100,25){$4$} \put(97.5,16.5){$3$}
\put(97.6,6.5){$2$} \put(100,2){$1$}}
\end{picture}
\caption{{\sf \smaller The $\ff6$-splitting of a braid of~$\BKL6$.  
Starting from the right, we extract the maximal right-divisor that keeps the sixth
strand unbraided, then rotate by $2\pi/6$ and extract the maximal
right-divisor that keeps the first strand unbraided, etc.}}
\label{F:Splitting}
\end{figure}

In practice, we shall use the following criterion for recognizing a
$\ff\nn$-splitting.
\begin{lemm}
Condition~\eqref{E:SplittingConditionV1}
 is equivalent to
\begin{equation}
\label{E:SplittingConditionV2}
\text{for each $\kk\ge1$, the~$\BKL\nno$-tail of~$\ff\nn^{\brdi\minus\kk}(\br_\brdi)\cdot...\cdot\ff\nn(\br_\kkp)$ is trivial.}
\end{equation}
\end{lemm}

\begin{proof}
By Lemma~\ref{L:Tail}, for every $\kk\ge1$, the braid $\br_\kk$ is the $\BKL\nno$-tail of the braid~$\ff\nn^{\brdi-\kk}(\br_\brdi)\cdot...\cdot\ff\nn(\br_\kkp)\cdot\br_\kk$ if and only if the $\BKL\nno$-tail of $\ff\nn^{\brdi-\kk}(\br_\brdi)\cdot...\cdot\ff\nn(\br_\kkp)$ is trivial.
Hence \eqref{E:SplittingConditionV1} and \eqref{E:SplittingConditionV2} are
equivalent.
\end{proof}

As the notion of $\ff\nn$-splitting is both new and fundamental for the sequel,
we mention  several examples.

\begin{exam}
\label{X:GeneratorSplitting}
Let us first determine the~$\ff\nn$-splitting of the Birman--Ko--Lee generators
of~$\BKL\nn$. For $\indii\le\nno$, the braid~$\aa\indi\indii$ belongs
to~$\BKL\nno$, then its $\ff\nn$-splitting is $(\aa\indi\indii)$. As
$\aa\indi\nn$ does not lie in~$\BKL\nno$, the rightmost entry in
its~$\ff\nn$-splitting must be~$1$. Now, we have
$\ff\nn\inv(\aa\indi\nn)=\aa\indio\nno$ for
$\indi\ge2$. Hence, for $\indi \ge 2$, the~$\ff\nn$-splitting of~$\aa\indi\nn$
is
$(\aa\indio\nno,1)$.  Finally, the braids $\aa1\nn$ and
$\ff\nn\inv(\aa1\nn)=\aa\nno\nn$ do not lie in~$\BKL\nno$, but
$\ff\nn^{\minus2}(\aa1\nn)=\aa\nnt\nno$ does. So the~$\ff\nn$-splitting
of~$\aa1\nn$ is $(\aa\nnt\nno,1,1)$.  To summarize, the~$\ff\nn$-splitting
of~$\aa\indi\indii$ is
\begin{equation}
\begin{cases}
(\aa\indi\indii)&\text{for $\indi<\indii\le\nno$},\\
(\aa\indio\nno,1)&\text{for $2\le\indi$ and $\indii=\nn$},\\
(\aa\nnt\nno,1,1)&\text{for $\indi=1$ and $\indii=\nn$}.
\end{cases}
\end{equation}
\end{exam}

\begin{exam}
\label{X:GarsideElementSplitting}
Let us compute the~$\ff3$-splitting of~$\ddd3^2$.
With the notation of the proof of~Proposition~\ref{P:Splitting}, we obtain
\begin{align*}
\br^{(0)}&=\br=(\aa12\,\aa23)^2 \qquad&& \br_1=\tail_2(\br^{(0)})=\aa12^2\\
\br^{(1)}&=\ff3\inv\big(\br^{(0)}\br_1\inv\big)=\ff3\inv(\aa12\aa13)=\aa13\aa23 \qquad&& \br_2=\tail_2(\br^{(1)})=1\\
\br^{(2)}&=\ff3\inv\big(\br^{(1)}\br_2\inv\big)=\ff3\inv(\aa13\aa23)=\aa23\aa12 \qquad&& \br_3=\tail_2(\br^{(2)})=\aa12,\\
\br^{(3)}&=\ff3\inv\big(\br^{(2)}\br_3\inv\big)=\ff3\inv(\aa23)=\aa12 \qquad&& \br_4=\tail_2(\br^{(3)})=\aa12,\\
\br^{(4)}&=\ff3\inv\big(\br^{(3)}\br_4\inv\big)=1 
\end{align*}
and we stop as the remainder~$\br^{(4)}$ is trivial.
Thus the~$\ff3$-splitting of~$\ddd3^2$ is the sequence~$(\aa12,\aa12,1,\aa12^2)$.

\end{exam}

%
%

\subsection{The rotating normal form}
\label{S:RotatingNF}

Using the~$\ff\nn$-splitting, we shall now construct  a unique normal form for the elements of~$\BKL\nn$, \ie, we identify for each braid~$\br$ in~$\BKL\nn$ a distinguished word that represents~$\br$.

The principle is as follows. First, each braid of~$\BKL2$ is represented by a unique word~$\aa12^\kk$.
Then, the~$\ff\nn$-splitting provides a distinguished decomposition for every braid of~$\BKL\nn$ in terms of braids of~$\BKL\nno$.
So, using induction on~$\nn$, we can define a normal form for~$\br$ 
in~$\BKL\nn$ starting with the normal form of the entries in the~$\ff\nn$-splitting
of~$\br$. 

For the rest of this paper, it will be convenient to take the following conventions
for braid words and the braids they represent.

\begin{defi}
A word on the letters~$\sig\ii$ (\emph{resp.} on the letters~$\aa\indi\indii$) is 
called a
\emph{$\sigg$-word} (\emph{resp.} an \emph{$a$-word}). 
The set of all positive $\nn$-strand $a$-words is denoted by~$\WBKL\nn$.
The braid
represented by an $a$-word or a $\sigg$-word $\ww$ is denoted by $\wwt$. For $\ww$ a
$\sigg$-word or an $a$-word and $\www$ a $\sigg$-word or an $a$-word, we say that $\ww$ is
equivalent  to~$\www$, denoted $\ww\equiv\www$, if we have $\wwt =
\wwwt$.
\end{defi}

According to the formulas~\eqref{E:Rotation}, $\ff\nn$ maps
each braid~$\aa\indi\indii$ to another similar braid~$\aa\indiii\indiv$. Using
this observation, we can introduce the alphabetical homomorphism, still
denoted~$\ff\nn$, that maps the letter~$\aa\indi\indii$ to the corresponding
letter~$\aa\indiii\indiv$, and extends to every $a$-word. Note that, in this
way, if the $a$-word~$\ww$ represents the braid~$\br$, then
$\ff\nn(\ww)$ represents~$\ff\nn(\br)$.

\begin{defi}
\label{D:RotatingNormalForm}
$(i)$ For $\br$ in~$\BKL2$, the \emph{$\ff2$-rotating normal form} of~$\br$
is defined to be the unique $a$-word~$\aa12^\kk$ that represents~$\br$.

$(ii)$ For $\br$ in~$\BKL\nn$ with $\nn \ge 3$, the \emph{$\ff\nn$-rotating normal form} of~$\br$ is defined to be the~$a$-word
$\ff\nn^\brdio(\ww_\brdi) \,... \,\ww_1$, where $(\br_\brdi,\Ldots,\br_1)$ is
the~$\ff\nn$-splitting of~$\br$ and $\ww_\kk$ is the
$\ff\nno$-rotating normal form of~$\br_\kk$ for each $\kk$.
\end{defi}

As the~$\ff\nn$-splitting of a braid~$\br$ lying in~$\BKL\nno$ is the length $1$ sequence~$(\br)$, the~$\ff\nn$-normal form and the $\ff\nno$-normal form of~$\br$ coincide.
Therefore, we can drop the subscript~$\nn$, and
speak of the \emph{rotating normal form}, or, simply, of the \emph{normal
form}, of a braid of~$\BKL\nn$. We naturally say that a positive
$a$-word is \emph{normal} if it is the normal form of the braid its represents.

\begin{exam}
\label{X:CyclingNormalForm}
Let us compute the normal form of~$\ddd4^2$.  
First, we check the equality $\ddd4^2=\aa12\,\aa14\,\ddd3^2$. Then,
the~$\ff4$-splitting of~$\ddd4^2$ turns out to be $(\aa23,\aa23,1,\ddd3^2)$.
The $\ff3$-splitting of~$\aa23$ is $(\aa12,1)$, and, therefore, its normal form
is $\ff3(\aa12)$, which is $\aa23$. Next, we saw in
Example~\ref{X:GarsideElementSplitting} that the~$\ff3$-splitting
of~$\ddd3^2$ is $(\aa12,\aa12,1,\aa12^2)$. Therefore, its normal form is
$\ff3^3(\aa12)\cdot\ff3^2(\aa12)\cdot\ff3(1)\cdot\aa12\aa12$, hence
$\aa12\cdot\aa13\cdot\varepsilon\cdot\aa12\aa12$, \ie,
$\aa12\aa13\aa12\aa12$. So, finally, the normal form of~$\ddd4^2$ is
$$\ff4^3(\aa23)\cdot\ff4^2(\aa23)\cdot\ff4(1)\cdot\aa12\aa13\aa12\aa12,$$
hence $\aa12\cdot\aa14\cdot\varepsilon\cdot\aa12\aa13\aa12\aa12$, \ie,
$\aa12\aa14\aa12\aa13\aa12\aa12$.
\end{exam}

As the relations of Lemma~\ref{L:DualRelations} preserve the
length, positive equivalent $a$-words always have the same length. Hence, if
$\www$ is the unique normal word equivalent to some word~$\ww$
of~$\WBKL\nn$, then $\ww$ and $\www$ have the same length.

\begin{prop}
\label{P:AlgoC}
For each length $\ll$ word~$\ww$ of $\WBKL\nn$, the normal form of
$\wwt$ can be computed in at most $O(\ll^2)$ elementary steps.
\end{prop}

\begin{proof}
Computing the $\BKL\nno$-tail
of the braid~$\wwt$ can be done in $O(\ll)$~steps. Hence computing the
$\ff\nn$-splitting can be done in~$O(\ll^2)$ steps. Taking into account the
observation that the lengths of equivalent words are equal, one deduces using
an easy induction on~$\nn$ that computing the rotating normal form of
$\wwt$ can be done in $O(\ll^2)$~steps.
\end{proof}

We considered above the question of going from~$\ww$ to an equivalent
normal word, thus first identifying the $\ff\nn$-splitting of~$\ww$ and then
finding the normal form of the successive entries. Conversely, when we start
with a normal word~$\ww$, it is easy to isolate the successive entries of the
$\ff\nn$-splitting of the braid~$\wwt$, \ie, to group the successive letters in
blocks.

Hereafter, if $\ww$ is a $\nn$-normal word , the (unique) sequence of $\nno$-normal words of
$(\ww_\brdi, \Ldots, \ww_1)$ such that $(\overline{\ww_\brdi}, \Ldots, \overline{\ww_1})$
is the $\ff\nn$-splitting of~$\wwt$ is naturally called the \emph{$\ff\nn$-splitting} of~$\ww$.

\begin{lemm}
\label{L:AlgoS}
Assume $\nn \ge 3$. For each normal word~$\ww$ of $\WBKL\nn$, the
$\ff\nn$-splitting of~$\ww$ can be computed in at most $O(\ll)$
elementary steps.
\end{lemm}

\begin{proof}
 By definition of $\ff\nn$, a generator
$\aa\indi\indii$ lies in $\ff\nn^\kk(\BKL\nno)$ if and only if
we have $\indi\not=\kk \mod \nn$ and
$\indii\not=\kk \mod \nn$. Therefore, given a normal word~$\ww$ in
$\WBKL\nn$, we can directly read the~$\ff\nn$-splitting $(\ww_\brdi,\Ldots,\ww_1)$ of~$\ww$. 
Indeed, reading $\ww$ from the right, $\ww_1$ is the
maximal suffix of~$\ww$ that lies in
$\WBKL\nno$, then $\ff\nn(\ww_2)$ is the maximal suffix of the remaining braid
lying in~$\ff\nn(\BKL\nno)$, etc, until the empty word is left.
\end{proof}

\begin{exam}
\label{X:RotatingToSplitting}
Let us consider the normal word $\ww=\aa12\,\aa14\,\aa23\,\aa12$ and 
compute the $\ff4$-splitting of~$\ww$. Reading
$\ww$ form the right, we find that the maximal suffix of $\ww$ containing no letter $\aa\indi\indii$ with $\indi=0 \mod \nn$ or $\indii=0 \mod \nn$ 
is $\aa23\,\aa12$. The latter is the maximal suffix of $\ww$ lying
in~$\WBKL3$, so we have
$\ww_1=\aa23\,\aa12$.
Repeating this process, one would easily find that the $\ff4$-splitting of~$\ww$ is $(\ff4^{-3}(\aa12),\, \ff4^{-2}(\aa14),\,\ff4\inv(1),\, \aa23\,\aa12)$, hence the sequence~$(\aa23, \aa23, 1, \aa23\,\aa12)$.
 \end{exam}

%
%

\section{Ladders}
\label{S:Ladders}

The $\ff\nn$-splitting operation associates with every braid in~$\BKL\nn$ a
finite sequence of braids in~$\BKL\nno$. Now, in the other direction, every
sequence of braids in~$\BKL\nno$ need not be the~$\ff\nn$-splitting of a
braid in~$\BKL\nn$. The aim of this section is to establish constraints that are
satisfied by  the entries of a $\ff\nn$-splitting. The main constraint is that
a $\ff\nn$-splitting necessarily contains what we call ladders, which are
sequences of (non-adjacent) letters~$\aa\indi\indii$ whose indices~$\indii$
make an increasing sequence (the bars of the ladder).

%
%

\subsection{Last letters}

We begin with some elementary observations about the last letters of the
normal forms of the entries in a $\ff\nn$-splitting.

 \begin{defi}
For each nonempty word~$\ww$, the last letter of $\ww$ is denoted by
$\last{\ww}$. Then, for each nontrivial braid~$\br$  in~$\BKL\nn$, we define
the
\emph{last letter} of~$\br$, denoted~$\last\br$, to be the last letter in the
normal form of~$\br$.
\end{defi}

\begin{lemm}
\label{L:LastLetter}
Assume $\nn\ge3$, and let $(\br_\brdi,\Ldots,\br_1)$ be a
$\ff\nn$-splitting.

$(i)$ For $\kk\ge 2$, the letter $\last{\br}_\kk$ is
$\aa\indi\nno$  for some~$\indi$, unless $\br_\kk = 1$ holds.

$(ii)$ For $\kk\ge3$, we have $\br_\kk \not= 1$.

$(iii)$ For $\kk\ge 2$, if the normal form of~$\br_\kk$ is
$\ww\,\aa\nnt\nno$  with
$\ww$ nonempty, then the letter $\last\ww$ is
$\aa\indi\nno$ for some~$\indi$.
\end{lemm}

\begin{proof}
$(i)$ Assume $\kk \ge 2$. Put  $\aa\indi\indii = \last{\br}_\kk$.
By~\eqref{E:SplittingConditionV2}, the~$\BKL\nno$-tail
of~$\ff\nn^{\brdi\minus\kkp}(\br_\brdi)\cdot...\cdot\ff\nn(\br_\kk)$ is
trivial. In particular, $\ff\nn(\last{\br}_\kk)$ cannot lie in~$\BKL\nno$, so 
$\last{\br}_\kk$ must be a letter of the form $\aa\indi\nno$.

$(ii)$ Assume that we have $\br_\brdii = 1$ with $\brdii \ge 3$ and
$\br_\kk \not= 1$ for $\brdi \ge \kk > \brdii$. By definition of a
$\ff\nn$-splitting, 
$\br_\brdi \not= 1$ holds, hence we must have $\brdii\le\brdio$. By
definition of~$\brdii$, we have $\br_\brdiip \not= 1$, hence, by~$(i)$,
$\last{\br}_\brdiip = \aa\indiii\nno$ for some~$\indiii$.
By~\eqref{E:SplittingConditionV2}, the~$\BKL\nno$-tail
of~$\ff\nn^{\brdi\minus\brdiio}(\br_\brdi)\cdot ...
\cdot\ff\nn^2(\br_\brdiip)\,\ff\nn(\br_\brdii)$ is~$1$. As we have
$\br_\brdii = 1$, we deduce that the $\BKL\nno$-tail
of~$\ff\nn^{\brdi\minus\brdiio}(\br_\brdi)\cdot ...
\cdot\ff\nn^2(\br_\brdiip)$ is~$1$ as well. This implies that the last letter of
$\ff\nn^2(\br_\brdiip)$, which is
$\ff\nn^2(\aa\indiii\nno)$, does not belong
to~$\BKL\nno$. Then \eqref{E:Rotation} implies
$\indiii=\nnt$ and
$\ff\nn^3(\aa\indiii\nno)=\aa12$. As the normal form of $\br_\brdiio$ is a word
of~$\WBKL\nno$, the braid $\ff\nn(\br_\brdiio)$ is represented by a
word that contains no letter $\aa1\indii$. Now the relations 
\begin{equation*}
\aa12\,\aa\indi\indii\equiv
\begin{cases}
\aa\indi\indii\,\aa12 &\text{for $2<\indi$,}\\
\aa1\indii\,\aa12 & \text{for $2=\indi$,}\\
\end{cases}
\end{equation*}
imply that there exists a braid~$\brr$ in $\BKL\nn$ satisfying $\aa12\,\ff\nn(\br_\brdiio)\equiv\brr\,\aa12$.
Therefore $\aa12$ is a right-divisor of~$\ff\nn^3(\br_\brdiip)\cdot\ff\nn^2(\br_\brdii)
\cdot\ff\nn(\br_\brdiio)$. As we have $\brdiio\ge2$ by hypothesis, this
contradicts~\eqref{E:SplittingConditionV2}.

$(iii)$ Assume that the normal form of $\br_\kk$ is $\ww\,\aa\nnt\nno$ with $\ww\not=\varepsilon$.
Let~$\aa\indi\indii$ be the last letter of~$\ww$.
As we have
\begin{equation}
\aa\indi\indii\,\aa\nnt\nno\equiv
\begin{cases}
\aa\nnt\nno\, \aa\indi\indii&\text{for $\indii<\nnt$,}\\
\aa\indi\nno\, \aa\indi\indii&\text{for $\indii=\nnt$,}
\end{cases}
\end{equation}
we must have $\indii=\nno$. Indeed, otherwise, $\aa\indi\indii$ would be  a
right-divisor of~$\br_\kk$,
\ie, the~$\BKL\nno$-tail of~$\ff\nn(\br_\kk)$ would be nontrivial, contradicting~\eqref{E:SplittingConditionV2}.
\end{proof}

%
%

\subsection{Barriers}

If $(\br_\brdi,\Ldots,\br_1)$ is the $\ff\nn$-splitting of a braid
of~$\BKL\nn$, then Lemma~\ref{L:LastLetter} says that,  for $\kk \ge 3$, the
letter $\last{\br}_\kk$ must be some letter~$\aa\indio\nno$. We shall see now
that the braid~$\br_\kko$ cannot be an arbitrary braid of~$\BKL\nno$: its
normal form has to satisfy some constraints involving the integer~$\indi$,
namely to contain a letter called an $\aa\indi\nn$-barrier---a key point in
subsequent results.

\begin{defi}
\label{D:Barrier}
The letter~$\aa\indiii\indiv$ is called an \emph{$\aa\indi\nn$-barrier} if we
have
\begin{equation}
1\le \indiii<\indi<\indiv\le \nno.
\end{equation}
\end{defi}
There exists no $\aa\indi\nn$-barrier with $\nn \le 3$; the only
$\aa\indi4$-barrier is~$\aa13$, which is an $\aa24$-barrier.

By definition, if the letter~$\xx$ is an $\aa\indi\nn$-barrier, then there exists in the 
presentation of~$\BKL\nn$ no relation of the form
$\aa\indi\nn \cdot \xx = \yy \cdot \aa\indi\nn$ allowing one to push the letter $\aa\indi\nn$ to
the right through the letter~$\xx$: so, in some sense, $\xx$ acts as a barrier. We shall prove now that (almost) every non-terminal entry~$\br_\kk$ of a splitting necessarily contains a barrier---a key point for the sequel. The reason is simple: if there were no barrier in~$\br_\kk$, then the relations would enable one to push the last letter of~$\ff\nn^2(\br_\kkp)$ through~$\ff\nn(\br_\kk)$ and incorporate it in~$\br_\kko$, contradicting the definition of a splitting. 

\begin{lemm}
\label{L:ExistBarrier}
Assume $\nn\ge3$, that $\br$ is a braid of~$\BKL\nno$ and that the~$\BKL\nno$-tail
of~$\ff\nn(\aa\indi\nn\br)$ is trivial for $\indi\le\nnt$. Then the normal
form of~$\br$ is not the empty word and it contains an $\aa\indi\nn$-barrier.
\end{lemm}

\begin{proof}
We assume that
the normal form~$\ww$ of~$\br$ contains no
$\aa\indi\nn$-barrier, and  derive a contradiction. Let $\www$ be the
word~$\aa\indi\nn\ww$ and let $\XX$ be the set of all
letters~$\aa\indii\indiii$ with
$\indi<\indiii\le\nno$.   Write $\www=\uu\,\vv$ where $\vv$ is the maximal
suffix of~$\ww$ containing letters from~$\XX$ only.
By hypothesis, the $\BKL\nno$-tail of $\wwwt$ is trivial. Hence the word
$\www$ ends with~$\aa\indii\nno$ for some $\indii$, \ie, $\vv$ is not empty. As
the first letter of~$\www$ is~$\aa\indi\nn$, which is not in $\XX$, the
word~$\uu$ is not empty. Let $\aa\indiv\indv$ be the last letter of~$\uu$. 
By construction of $\uu$, the letter~$\aa\indiv\indv$ is either $\aa\indi\nn$ or it satisfies $\indv\le\indi$.
In both cases, the braid~$\ff\nn(\aa\indiv\indv)$ lies in~$\BKL\nno$. We shall now prove
that $\aa\indiv\indv$ quasi-commutes with~$\vv$, \ie, there exists a word~$\vvv$ satisfying $\aa\indiv\indv\vv\equal\vv'\aa\indiv\indv$. 
Every letter~$\aa\indii\indiii$ occurring in~$\vv$ is not an $\aa\indi\nn$-barrier, \ie, it satisfies $\indi\le
\indii<\indiii\le\nno$. Hence, by the relations
 \[\aa\indiv\indv\aa\indii\indiii\equal
\begin{cases}
\aa\indii\indiii\aa\indiv\indv, & \text{for $\indi<\indii$ or $\indv<\indi$}\qquad\hfill\text{by \eqref{E:DualCommutativeRelation},}\\
\aa\indiv\indiii\aa\indiv\indv, & \text{for $\indii=\indv=\indi$}\qquad\hfill\text{by \eqref{E:DualNonCommutativeRelation},}\\
\aa\indiii\indv\aa\indiv\indv, & \text{for $\indii=\indiv=\indi$}\qquad\hfill\text{by \eqref{E:DualNonCommutativeRelation},}\\
\end{cases}\]
the letter~$\aa\indiv\indv$ quasi-commutes with~$\vv$. 
Then, $\ff\nn(\aa\indiv\indv)$ is a right-divisor of~$\ff\nn(\aa\indi\nn\,\br)$.
This contradicts the hypothesis that
the~$\BKL\nno$-tail of~$\ff\nn(\aa\indi\nn\br)$ is trivial since the braid $\ff\nn(\aa\indiv\indv)$ belongs to $\BKL\nno$.
\end{proof}

We now show how Lemma~\ref{L:ExistBarrier} can be used in the context of
a $\ff\nn$-splitting. 

\begin{lemm}
\label{L:ExistBarrierv2}
Let $(\br_\brdi,\Ldots,\br_1)$ be a $\ff\nn$-splitting of some braid of~$\BKL\nn$ with $\nn\ge3$.
Then, for each $\kk$ in~$\{\brdi-1, \Ldots, 2\}$ such that
$\last{\br}_\kkp$ is not~$\aa\nnt\nno$ (if any), the normal form of
$\br_\kk$ contains an $\ff\nn(\last{\br}_\kkp)$-barrier.
\end{lemm}

\begin{proof}
Condition~\eqref{E:SplittingConditionV2} implies that the~$\BKL\nno$-tail of~$\ff\nn^{\brdi-\kk+1}(\br_\brdi)\cdot ... \cdot\ff\nn^2(\br_\kkp)\,\ff\nn(\br_\kk)$ is trivial.
In particular the~$\BKL\nno$-tail of~$\ff\nn^2(\last{\br}_\kkp)\,\ff\nn(\br_\kk)$ is trivial. 
Then, Lemma~\ref{L:ExistBarrier} implies that the normal form of~$\br_\kk$
contains an $\aa\indi\nn$-barrier.
\end{proof}

\begin{exam}
\label{X:ExistBarrier}
Let us consider the braid $\br$ whose normal form is
\[
\aa24\,\aa13\,\aa45\,\aa24\,\aa24\,\aa35\,\aa45.
\]
The $\ff5$-splitting of $\br$ is $(\br_4, \br_3, \br_2, \br_1)$ with
\[
\br_4=\aa14, \quad\br_3=\aa14, \quad\br_2= \aa34\aa13\aa13\aa24\aa34\quad\text{and}\quad \br_1=1. 
\]
 The letter $\last{\br}_4$ is $\aa14$, hence by Lemma~\ref{L:ExistBarrierv2} the normal form of
$\br_3$ must contain an $\aa25$-barrier: this is true, since $\aa14$
is an $\aa25$-barrier. The letter $\last{\br}_3$ is $\aa14$.
Then, again by Lemma~\ref{L:ExistBarrierv2}, the normal form of $\br_2$ has to contain an
$\aa25$-barrier: this is true, since the normal form of
$\br_2$ is $\aa34\aa13\aa13\aa24\aa34$, which contains the  $\aa25$-barrier~$\aa13$.
\end{exam}

%
%

\subsection{Ladders}

We have seen above in Lemma~\ref{L:ExistBarrierv2} that every normal
word~$\ww$ of~$\WBKL\nno$ such that the~$\BKL\nno$-tail
of~$\ff\nn(\aa\indi\nn\,\wwt)$ is trivial contains at least one
$\aa\indi\nn$-barrier.  We shall see now that, under the same hypotheses,
$\ww$ contains not only one barrier, but even a sequence of overlapping
barriers. Words containing such sequences are what we shall call ladders.

\begin{defi}
\label{D:Ladder}
For $\nn\ge 3$, we say that a normal word~$\ww$ is an
\emph{$\aa\indi\nn$-ladder of height~$\hh$ lent on~$\aa\indiio\nno$}, if
there exists a decomposition
\begin{equation}
\label{E:Ladder}
\ww = \ww_0\,\xx_1\,\ww_1 \,...\, \ww_\hho\,\xx_\hh\,\ww_\hh,
\end{equation}
and a sequence $\indi = \findi0 < \findi1 < ... < \findi\hh = \nno$ such that

$(i)$ for each~$\kk\le\hh$, the letter~$\xx_\kk$ is an $\aa{\findi\kko}\nn$-barrier of the form $\aa{..}{\findi\kk}$,

$(ii)$ for each~$\kk<\hh$, the word~$\ww_\kk$ contains no
$\aa{\findi\kk}\nn$-barrier,

$(iii)$ the last letter of $\ww$ is~$\aa\indiio\nno$.
\end{defi}

By convention, any $a$-word whose last letter is $\aa\indiio\nno$ is an
$\aa\nno\nn$-ladder lent on~$\aa\indiio\nno$ and its height is $0$.
There exist no $\aa\indi\nn$-barrier with $\nn\ge3$, hence there exist only
$\aa12$-ladders in $\WBKL3$.

The concept of a ladder is easily illustrated by representing the
generators~$\aa\indi\indii$ as a vertical line from the~$\indi$th line to
the~$\indii$th line on an $\nn$-line stave.  Then, for every $\kk\ge0$, the
letter~$\xx_\kk$ looks like a bar of a ladder---see Figure~\ref{F:Ladder}.

\begin{figure}[htb]
\begin{picture}(110,23)
\put(2,0){\includegraphics{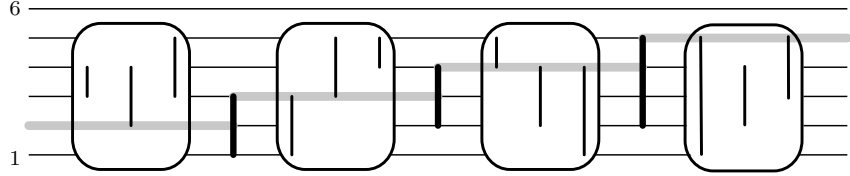}}
\put(0,1){\footnotesize $1$}
\put(0,21){\footnotesize $6$}
\end{picture}
\caption{{\sf \smaller An $\aa25$-ladder lent on~$\aa35$ (the last letter).
The gray line starts at position $2$ and goes up to position $5$ using the bars of the ladder.
The empty spaces between bars in the ladder are represented by a framed box.
In such boxes the vertical line representing the letter~$\aa\ii\jj$ does not cross the gray line.
The bars of the ladder are represented by black thick vertical lines.}}
\label{F:Ladder}
\end{figure}

Our aim is to prove that the normal form of each non-terminal entry in a
$\ff\nn$-splitting is a ladder. In order to do that, we begin with a preparatory lemma showing that barriers
necessarily occur after certain letters of a normal form. Applying this result
repeatedly will eventually provide us with a ladder.

\begin{lemm}
\label{L:For:L:Ladder}
Assume $\nn\ge4$, that $\ww$ is a suffix of a normal word of $\WBKL\nno$,
that $\aa\indi\indii$ belongs to~$\BKL\nnt$, and that  the~$\BKL\nno$-tail
of~$\ff\nn(\aa\indi\indii\wwt)$ is trivial. Then $\ww$ contains an
$\aa\indii\nn$-barrier.
\end{lemm}

\begin{proof}
Let $\XX$ be the set of all letters~$\aa\indiii\indiv$ with $\indiv>\indii$. 
Write $\ww=\uu\,\vv$ where $\vv$ is the maximal suffix containing letters of~$\XX$ only.
As, by hypothesis, the~$\BKL\nno$-tail of~$\ff\nn(\aa\indi\indii\wwt)$ is trivial, the last letter of~$\ww$ exists and
has the form $\aa{..}\nno$, hence $\vv$ is nonempty.

As the letter~$\aa\pp\qq$ does not
lie in~$\XX$, the word~$\uu$ is not empty. Let $\xx=\aa\indv\indvi$ be the
last letter of~$\uu$. 
By definition of $\uu$, we have $\indvi\le\indii$.
We suppose that~$\vv$ contains no
$\aa\indii\nn$-barrier,
\ie, every letter~$\aa\indiii\indiv$ of~$\vv$ satisfies~$\indiii\ge\indii$, and
eventually derive a contradiction. 
By~\eqref{E:DualCommutativeRelation}
and~\eqref{E:DualNonCommutativeRelation}, we have
 \[\xx\aa\indiii\indiv\equiv
\begin{cases}
\aa\indiii\indiv\xx & \text{for $\indiii>\indii$ or $\indvi<\indii$,}\\
\aa\indv\indiv\xx & \text{for $\indii=\indiii=\indvi$,}
\end{cases}\]
which implies that $\xx$ and $\vv$ quasi-commute, \ie, there exists an
$a$-word~$\vvv$ satisfying $\xx\,\vv\equiv\vvv\,\xx$. Then
$\ff\nn(\xx)$ is a right-divisor of the braid represented
by~$\ff\nn(\aa\indi\indii\ww)$.
The hypothesis about~$\aa\indi\indii$ and the relation $\indvi\le\indii$
imply that $\ff\nn(\xx)$ lies in~$\BKL\nno$, which contradicts the
hypothesis that $\ff\nn(\aa\indi\indii\,\wwt)$ is trivial.
\end{proof}

We can now show that every normal word satisfying some mild
additional condition is a ladder.

\begin{prop}
\label{P:NormalFormIsLadder}
Assume $\nn\ge3$, that $\br$ belongs to~$\BKL\nno$ and that the~$\BKL\nno$-tail
of~$\ff\nn(\aa\indi\nn\,\br)$ is trivial for some $\indi \le \nnt$.
Then the normal form of~$\br$
is an $\aa\indi\nn$-ladder lent on~$\last{\br}$.
\end{prop}

\begin{proof}
We put $\findi1=\indi$ and let $\ww$ be the normal form of~$\br$.
Lemma~\ref{L:ExistBarrierv2} implies that $\ww$ admits a decomposition
$\ww_0\xx_1\ww^{(0)}$, where $\ww_0$ is the maximal prefix of~$\ww$ that
contains no $\aa\indi\nn$-barrier and $\xx_1=\aa{..}{\findi1}$ is
an
$\aa\indi\nn$-barrier.  By hypothesis, the~$\BKL\nno$-tail of the braid
$\ff\nn(\aa\indi\nn\,\wwt)$ is trivial, \ie, the~$\BKL\nno$-tail of~$\ff\nn(\xx_1\,\overline{\ww}^{(0)})$ is trivial.   Assume
$\findi1\not=\nno$. Lemma~\ref{L:For:L:Ladder} implies that the word~$\ww^{(0)}$
admits a decomposition~$\ww_1\,\xx_2\,\ww^{(1)}$, where $\ww_1$ is the
maximal prefix of~$\ww^{(0)}$ that contains no $\aa{\findi1}\nn$-barrier and
$\xx_2$ is an $\aa{\findi1}\nn$-barrier. The same argument repeats until we
find a decomposition~$\ww_0\,\xx_1\,\ww_1\Ldots\xx_\hh\ww^{(\hh)}$ with
$\findi\hh=\nno$. Then, putting $\ww_\hh=\ww^{(\hh)}$, we have
obtained for~$\br$ a word representative that satisfies all requirements
of~Definition~\ref{D:Ladder}.
\end{proof}

Applying Proposition~\ref{P:NormalFormIsLadder} to the successive entries of a
$\ff\nn$-splitting allows one to deduce that its entries contain ladders.

\begin{coro}
\label{C:SplittingAndLadders}
Assume $\nn\ge3$ and that $(\br_\brdi,\Ldots, \br_1)$ is a sequence
in~$\BKL\nno$ that is the $\ff\nn$-splitting of some braid of~$\BKL\nn$.
Then, for each~$\kk$ in~$\{\brdio, \Ldots, 2\}$, the normal form
of~$\br_\kk$ is a  $\ff\nn(\last{\br}_\kkp)$-ladder lent on~$\last{\br}_\kk$.
\end{coro}

\begin{proof}
Condition~\eqref{E:SplittingConditionV2} implies that the~$\BKL\nno$-tail
of~$\ff\nn^2(\br_\kkp)\ff\nn(\br_\kk)$ is trivial. In particular,
the~$\BKL\nno$-tail of~$\ff\nn^2(\last{\br}_\kkp)\ff\nn(\br_\kk)$ is trivial. 
By Lemma~\ref{L:LastLetter}, the letter $\last{\br}_\kkp$ has the form
$\aa{..}\nno$. Then Proposition~\ref{P:NormalFormIsLadder} implies that the
normal form of~$\br_\kk$ is a $\ff\nn(\last{\br}_\kkp)$-ladder lent
on~$\last{\br}_\kk$.
\end{proof}

By definition of a ladder, as the letter $\aa\nnt\nno$ is not a barrier, if  a
word $\ww\,\aa\nnt\nno$ is an $\aa\indi\nn$-ladder and $\ww$ is
nonempty, then
$\ww$ is an $\aa\indi\nn$-ladder lent on
$\aa\indiiio\nno$ for some $\indiii$---see Lemma~\ref{L:LastLetter}$(iii)$.

Another consequence of Proposition~\ref{P:NormalFormIsLadder} is:

\begin{coro}
\label{C:SplittingAndLadders2}
Assume $\nn\ge3$ and that $(\br_\brdi,\Ldots, \br_1)$ is a sequence
in~$\BKL\nno$ that is the $\ff\nn$-splitting of some braid of~$\BKL\nn$.
Then, for each~$\brdii$ in~$\{\brdio, \Ldots, 2\}$ such that $\br_\brdii$ is
either~$1$ or~$\aa\nno\nn$, we have $\last{\br}_\brdiip = \aa\nnt\nno$.
\end{coro}

\begin{proof}
Assume $\br_\brdii \in \{1, \aa\nnt\nno\}$. Let $\aa\indio\nno$
be the last letter of $\br_\brdiip$. Condition~\eqref{E:SplittingConditionV2}
implies that the~$\BKL\nno$-tail of~$\ff\nn^2(\br_\brdiip)\ff\nn(\br_\brdii)$ is
trivial. In particular the~$\BKL\nno$-tail of~$\ff\nn(\aa\indi\nn\,\br_\brdii)$ is
trivial.  Then, as the normal form of $\br_\brdii$ contains no barrier,
Proposition~\ref{P:NormalFormIsLadder} implies $\indi=\nno$.
Therefore we have $\last{\br}_\brdiip=\aa\nnt\nno$.
\end{proof}

\begin{exam}
\label{X:ExistLadder}
 Let us consider the braid of Example~\ref{X:ExistBarrier} again.
Its $\ff4$-splitting is is $(\br_4, ..., \br_1)$ with $\br_4=\aa14$,
$\br_3=\aa14$, $\br_2= \aa34\aa13\aa13\aa24\aa34$ and $\br_1=1$. The
normal form of $\br_4$ ends with $\aa14$, hence the normal form of $\br_3$
must be an $\aa25$-ladder lent on $\aa14$. This is true: here the ladder
is $\varepsilon\cdot\aa14\cdot\varepsilon$, and it has height~$1$, corresponding, with the notation of Definition~\ref{D:Ladder}, to $\ww_0=\varepsilon$, $\xx_1=\aa14$ and $\ww_1=\varepsilon$. Similarly, the normal form of
$\br_3$ ends with $\aa14$, hence by Corollary~\ref{C:SplittingAndLadders},
the normal form of $\br_2$ must be an $\aa25$-ladder lent on
$\aa34$. This is true again. Here the ladder has height~$2$, and its
decomposition is
$\aa34\cdot\aa13\cdot\aa13\cdot\aa24\cdot\aa34$, corresponding, with the
notation of Definition~\ref{D:Ladder}, to $\ww_0=\aa34$,
$\xx_1=\aa13$,
$\ww_1=\aa13$,
$\xx_2=\aa24$ and $\ww_2=\aa34$. We observe that $\aa13$ is an
$\aa25$-barrier and that
$\aa24$ is an $\aa35$-barrier.
\end{exam}

%
%

\section{Reversing}
\label{S:Reversing}

In Section~\ref{S:Ladders}, we have established that almost every normal word
is a ladder. We wish to use this result to establish Theorem~1,
\ie, to obtain (short) $\sigg$-definite representatives.  The basic question is
as follows. Starting with a braid word that contains letters~$\sig\ii$
with both positive and  negative exponents, we shall try to obtain an equivalent
word that is $\sigg$-positive---it is known that one cannot obtain both a
$\sigg$-positive and a $\sigg$-negative representative, so our attempt must fail
in some cases. The problem is to get rid of the letters~$\siginv\ii$ with
maximal index~$\ii$. We shall see that, without loss of generality, we can
assume that the initial word consists of an initial fragment---that will be called
dangerous---containing the negative letters (those with a negative exponent),
followed by a normal word, hence by a ladder according to
Proposition~\ref{P:NormalFormIsLadder}. Then, the main technical step
consists in proving that the product of a dangerous word with a ladder  can be
transformed using a simple algorithmic process called reversing into an
equivalent
$\sigg$-positive word: roughly speaking, ladders protect against dangerous
elements.

%
%

\subsection{D-words}

Up to now, we have considered braid words involving letters of two different
alphabets, namely the Artin generators~$\sig\ii$ and the
Birman--Ko--Lee generators~$\aa\indi\indii$. From now on,
we shall also use a third alphabet, corresponding to the following braids.

\begin{defi}
For $1\le\indi<\indii$, we put
\[
\dddd\indi\indii = \aa\indi\indip\,\aa\indip\indipp\,...\,\aa\indiio\indii \ (\  = \sig\indi\sig\indip \,...\, \sig\indiio ). 
\]
\end{defi}

So, in particular, the equalities
\begin{equation}
\label{E:DeltaRelations:DualGenerator}
\aa\indi\indii = \dddd\indi\indii \dddd\indi{\indii-1}\inv = \dddd\indi\indiio\ \sig\indiio\ \dddd\indi\indiio\inv
 \end{equation}
hold for $1\le\indi<\indii$.

Hereafter it is convenient to use $\dddd\indi\indii$ as a single letter. In this
context, a word on the letters~$\dddd\indi\indii^{\plusminus1}$ 
(\emph{resp.}
$\aa\indi\indii^{\plusminus1}$ and
$\dddd\indi\indii^{\plusminus1}$, \emph{resp.} $ \sig\ii^{\pm1}$) will be
called a
\emph{$d$-word} (\emph{resp.} an
\emph{$ad$-word}, \emph{resp.} a $\sigg$-word). 
We adopt the convention that the $d$-word $\dddd\indi\indi$ is the empty
word $\varepsilon$ for all $\indi$.

All words over the above alphabets represent braids, and they can be translated
into $\sigg$-words. It is coherent with the intended braid interpretations to define
words~$\aab\indi\indii$ and $\ddddb\indi\indii$ by

\begin{equation}
\label{E:Translation}
\aab\indi\indii=\sig\indi...\sig\indiit\sig\indiio\siginv\indiit...\siginv\indi, \quad
\ddddb\indi\indii=\sig\indi...\sig\indiio.
\end{equation}
In this way, for each $ad$-word~$\ww$, the braid represented by~$\ww$
coincides with the braid represented by the~$\sigg$-word~$\wwb$ obtained
from~$\ww$ by replacing every letter~$\aa\indi\indii$ by $\aab\indi\indii$ and every letter~$\dddd\indi\indii$ by $\ddddb\indi\indii$, and no ambiguity can
result from using different alphabets. Of course, if
$\ww$ and
$\www$ are two $ad$-words, we declare that $\ww\equiv\www$ is true if
the~$\sigg$-words $\wwb$ and $\wwwb$ are equivalent under the
braid relations~\eqref{E:BnPresentation}. Note in particular that the
braid represented by the~$d$-word~$\dddd1\nn$ is the Garside
braid~$\ddd\nn$.

The following equivalences of~$ad$-words easily result from the definitions.

\begin{lemm}
\label{L:DeltaRelations}
The following relations are satisfied:
\begin{subequations}
\renewcommand{\theequation}{\theprop.\roman{equation}}
\begin{align}
\label{E:DeltaRelations:Decomposition}
\dddd\indi\indiii&\equiv\dddd\indi\indii\ \dddd\indii\indiii&&\qquad\textrm{for $\indi<\indii<\indiii$},\\
\label{E:DeltaRelations:ImageByAutomorphism}
\ff\nn(\dddd\indi\indii)&\equiv\dddd\indip\indiip&&\qquad\textrm{for $\indi<\indii\le\nno$},\\
\label{E:DeltaRelations:Commutation}
\dddd\indi\indii\ \dddd{\indiii}{\indiv}&\equiv\dddd{\indiii}{\indiv}\ \dddd\indi\indii&&\qquad\textrm{for $\indi<\indii<\indiii<\indiv$},\\
\label{E:DeltaRelations:AutomorphismQuotient}
\dddd\indiii\indiv\inv\ \aa\indi\indii\ \dddd\indiii\indiv&\equiv\ff\indiv\inv\big(\ff\indiii(\aa\indi\indii)\big)&& \qquad\textrm{for $\indi<\indii\le\indiii<\indiv$}.
\end{align}
\end{subequations}
\end{lemm}

\begin{proof}
Relation~\eqref{E:DeltaRelations:Decomposition} holds by definition of~$\dddd\indi\indii$.
Relation~\eqref{E:DeltaRelations:ImageByAutomorphism} is an immediate consequence of \eqref{E:Rotation}. 
For \eqref{E:DeltaRelations:Commutation}, we observe that the $\sig\ii$ of
greatest index occurring in~$\dddd\indi\indii$ is $\sig\indiio$, while the
$\sig\ii$ of lower index occurring in~$\dddd{\indiii}{\indiv}$ is
$\sig{\indiii}$. As $\indii<\indiii$ implies $\indiio\le\indiii\minus2$, we can
apply the Artin commutativity relation of~\eqref{E:BnPresentation} to obtain
the expected result.

It remains to prove~\eqref{E:DeltaRelations:AutomorphismQuotient}.
First, \eqref{E:DeltaRelations:Decomposition} implies
$\dddd1\indiv\equiv\dddd1\indiii\,\dddd\indiii\indiv$, hence
$\dddd\indiii\indiv\equiv\dddd1\indiii\inv\,\dddd1\indiv$.
We deduce
$\dddd\indiii\indiv\inv\ \aa\indi\indii\
\dddd\indiii\indiv \equiv \dddd1\indiv\inv\ \dddd1\indiii\,\aa\indi\indii\,\dddd1\indiii\inv\
\dddd1\indiv$ . As, by hypothesis, $\aa\indi\indii$ lies in
$\BKL\indiii$, the subword~$\dddd1\indiii\,\aa\indi\indii\,\dddd1\indiii\inv$ is equivalent
to~$\ff\indiii(\aa\indi\indii)$. Finally the conjunction of $\BKL\indiii\subseteq\BKL\indiv$  and
$\ff\indiii(\aa\indi\indii)\in\BKL\indiii$ implies
$\dddd1\indiv\inv\,\ff\indiii(\aa\indi\indii)\,\dddd1\indiv
\equiv \ff\indiv\inv\big(\ff\indiii(\aa\indi\indii)\big)$.
\end{proof}

\subsection{Sigma-positive words}

Our aim is to obtain $\sigg$-positive and $\sigg$-negative
representative words. We shall need slightly more precise versions of these
notions.

\begin{defi}
\label{D:SigmaPositiveSigmaWord}
\label{D:SigmaNonNegative}
\label{D:SigmaPositiveADWord}
$(i)$ A $\sigg$-word~$\ww$ is called \emph{$\sig\ii$-positive} (resp.
$\sig\ii$-negative) if $\ww$ contains at least one letter $\sig\ii$, no letter
$\siginv\ii$ (resp. at least one letter
$\siginv\ii$ and  no letter~$\sig\ii$) and no letter $\sigpm\jj$ for $\jj\ge\ii$.

$(ii)$ A $\sigg$-word~$\ww$ is said to be \emph{$\sig\ii$-nonnegative} if  it is
$\sig\ii$-positive, or it does not contain the letter~$\sigpm\jj$ with $\jj\ge\ii$.

$(iii)$ An $ad$-word~$\ww$ is called $\sig\ii$-positive (resp. $\sig\ii$-negative, resp. $\sig\ii$-nonnegative)
if the word~$\wwb$ is $\sig\ii$-positive (resp. $\sig\ii$-negative, resp.
$\sig\ii$-nonnegative).
\end{defi}

\begin{exam}
A $\sigg$-word cannot be simultaneously $\sig\ii$-positive and
$\sig\ii$-negative, but, on the other hand, a $\sigg$-word can be neither
$\sig\ii$-positive nor $\sig\ii$-negative for any~$\ii$. For instance,
$\sig2\sig1\siginv2$ is neither $\sig2$-positive (since it contains a
letter~$\siginv2$), nor $\sig2$-negative (since it contains a letter~$\sig2$), nor
$\sig1$-positive or $\sig1$-negative (since it contains a letter~$\sig2$). By
contrast, the equivalent word $\siginv1\sig2\sig1$ is $\sig2$-positive.
On the other hand, the empty word, $\siginv1$, and $\sig2\,\siginv1$ are
$\sig2$-nonnegative words, since the letter~$\siginv2$ does not occur in it.

As for $a$-words, $\aa23\inv \aa13$ is not $\sig2$-positive,
since its translation under~\eqref{E:Translation} is
the~$\sigg$-word~$\siginv2 \sig1
\sig2
\siginv1$, which is not $\sig2$-positive as it contains the letter~$\siginv2$.
However, the previous $a$-word is equivalent to the~$a$-word~$\aa13
\aa12\inv$, which translates into~$\sig1 \sig2 \siginv1 \siginv1$ and is
therefore $\sig2$-positive, since $\sig1 \sig2 \siginv1 \siginv1$ contains one
letter~$\sig2$ and no letter~$\siginv2$.
\end{exam}

An immediate consequence of Definition~\ref{D:SigmaPositiveADWord}$(iii)$ is

\begin{lemm}
An $ad$-word~$\ww$ is $\sig\ii$-positive if $\ww$ contains at least one letter 
$\aa{..}\iip$ or $\dddd{..}\iip$, and no letter $\aa{..}\iip\inv$,
$\dddd{..}\iip\inv$, $\aa{..}\jj^{\pm1}$, or ~$\dddd{..}\jj^{\pm1}$ with
$\jj>\iip$.
\end{lemm}


\subsection{Dangerous words}

We arrive at a key notion. The problem is
to identify the generic form of the $\sigg$-negative fragments we wish to control
and, possibly, get rid of. It turns out that the convenient notion is defined in terms of
the letters~$\dddd\indi\indii\inv$, and it is what we call a dangerous word.

\begin{defi} 
For $\nn\ge3$, a $d$-word is called \emph{$\aa\indi\nn$-dangerous of type
$\indii$} if it has the form 
\begin{equation}
\label{E:Dangerous}
\dddd{\findi\dd}\nno\inv\ \dddd{\findi\ddo}\nno\inv\ ...\ \dddd{\findi1}\nno\inv 
\end{equation} 
with $\indii=\findi\dd\ge\findi\ddo\ge...\ge\findi1=\indi$. 
\end{defi}

By convention the unique $\aa\nno\nn$-dangerous word is the empty word.

Note that a dangerous $d$-word~$\ww$ is completely determined  by
the~$\sigg$-word~$\wwb$. Indeed, we recover~$\ww$ from~$\wwb$
by gathering the $\siginv\ii$'s and cutting before each letter~$\siginv\nnt$. For
instance,
$\siginv3 \siginv2 \siginv3 \siginv2 \siginv1$ can only be the translation of the~$\aa15$-dangerous word~$\dddd24\inv \dddd14\inv$.

At this point, the definition of a dangerous word comes out of a hat.
For the moment, let us observe that the letter~$\aa\indi\nn$ is equivalent to~$\dddd\indi\nn\dddd\indi\nno\inv$.
In this expression, $\dddd\indi\nno\inv$, which is $\aa\indi\nn$-dangerous, corresponds to the negative fragment of~$\aa\indi\nn$.
This reflects the intuition that dangerous words are associated with the negative parts of~$a$-words---hence with their dangerous parts in view of our aim, which is to find $\sigg$-positive expressions.

%
 %

\subsection{The reversing algorithm}
\label{SS:ReversingDiagram}

The aim of this section is to describe an algorithm that, starting with an $\aa\indi\nn$-dangerous word~$\uu$ and an $\aa\indi\nn$-ladder~$\ww$, returns a $\sig\nnt$-positive word~$\www$ 
that is equivalent to~$\uu\,\ww$ and that is close to be an $\aa\indi\nn$-ladder
in a sense that will be defined below.

The basic ingredient is a process called \emph{reversing} that transforms (certain)
$ad$-words with letters $\dddd{..}\nno\inv$ on the left into equivalent words
with letters $\dddd{..}\nno\inv$ on the right (or with no letter $\dddd{..}\nno\inv$ at all). Thus
reversing is a process for pushing letters $\dddd{..}\nno\inv$ to the right.

\begin{defi}
\label{D:Reversing}
Let $\ww, \www$ be $ad$-words.
We declare that $\ww\rev1\www$ is true if $\www$ is obtained from $\ww$ by replacing  a
subword $\uu$ of $\ww$ by a word $\uuu$ such that $(\uu,\uuu)$ is one of 
the pairs
\begin{align}
\label{E:Reversing:i}&(\dddd\pp\nno\inv\aa\indiii\indiv, \quad\revi\indi(\aa\indiii\indiv)\,\dddd\indi\nno\inv) &&\text{with $\indiv\le\indi\le\nnt$ or $\indi\le\indiii\le\nnt$},\\
\label{E:Reversing:ii}&(\dddd\pp\nno\inv\aa\indiii\indiv, \quad\dddd\indiii\nno\revii\indi(\aa\indiii\indiv)\,\dddd\indiv\nno\inv) &&\text{with $\indiii<\indi<\indiv\le\nno$},\\
\label{E:Reversing:iii}&(\dddd\pp\nno\inv\dddd\indiii\nno, \quad\dddd\indiii\nno\,\reviii\indi) &&\text{with $\indiii<\indi\le\nnt$},
\end{align}
with
\begin{equation*}
\revi\indi(\aa\indiii\indiv)=\begin{cases}
\aa\indiii\nno & \text{for $\indiv=\indi$},\\
\aa\indiii\indiv & \text{for $\indiv<\indi$},\\
\aa\indivo\nno&\text{for $\indiii=\indi$},\\
\aa\indiiio\indivo&\text{for $\indiii>\indi$}.
\end{cases},\quad
\begin{array}{l}
\revii\indi(\aa\indiii\indiv)=\dddd\indio\nnt\inv\,\dddd\indiii\indivo\inv,\\
\\
\reviii\indi=\dddd\indio\nnt\inv.
\end{array}
\end{equation*}
We say that $\ww$ \emph{reverses} to~$\www$, denoted
$\ww\revv\www$,  if there exists a sequence of words
$\ww_0$,$\ww_1$,...,$\ww_\ll$ satisfying
$\ww_0=\ww$, $\ww_\ll=\www$, and $\ww_\kk\rev1\ww_\kkp$  for every
$\kk$.
\end{defi}

Before giving an example, we introduce the notion of a reversing diagram, which
enables one to conveniently illustrate the reversing process. Assume that
$\ww_0,\ww_1,\Ldots,\ww_\ll$ is a reversing sequence, \ie, is a sequence of $ad$-words such that $\ww_\kk\rev1\ww_\kkp$ holds for every~$\kk$. First, 
we associate with $\ww_0$ a path labeled with the successive letters
of~$\ww_0$: we associate to every letter~$\dddd\indi\nno\inv$ a vertical
down-oriented edge labeled $\dddd\indi\nno$, and to every other letter~$\xx$ a
horizontal right-oriented edge labeled $\xx$. Then, we successively represent the
words $\ww_1,\Ldots,\ww_\ll$ as follows : if $\ww_\kkp$ is obtained from
$\ww_\kk$ by replacing $\dddd\indi\nno\inv\,\xx$ by
$\uu\,\dddd\indii\nno\inv$ (where
$\dddd\indi\nno\inv\,\xx\revv\uu\,\dddd\indii\nno\inv$ holds), then we
complete the pattern associated with the subword~$\dddd\indi\nno\inv\,\xx$ using
right-oriented edges labeled $\uu$ and down-oriented edge labeled~$\dddd\indii\nno$, see Figure~\ref{F:Reversing:Definition}.
\begin{figure}[htb]
\begin{picture}(121,20)
\put(8,2){\includegraphics{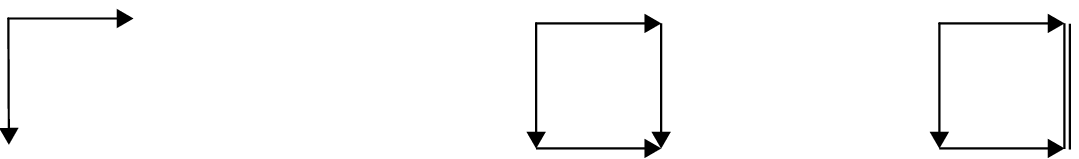}}
{\footnotesize
\put(0,10){$\dddd\indi{\nn-1}$}
\put(14,17){$\xx$}
\put(26,8){is completed into}
\put(53,10){$\dddd\indi{\nn-1}$}
\put(67,17){$\xx$}
\put(67,0.5){$\uu$}
\put(76,10){$\dddd\indii{\nn-1}$}
\put(88,8){or}
\put(94,10){$\dddd\indi{\nn-1}$}
\put(109,17){$\xx$}
\put(109,0.5){$\uu$}
\put(118,10){$\varepsilon$}
}
\end{picture}
\caption{{\sf \smaller Reversing of~$\dddd\indi\nno\inv\,\xx$
into~$\uu\,\dddd\indii\nno$. We replace the down-oriented edge
labeled~$\dddd\indii\nno$ by a vertical double line labeled $\varepsilon$
whenever the relation $\indii=\nno$ holds, \ie , $\dddd\indii\nno\equiv\varepsilon$ holds.}}
\label{F:Reversing:Definition}
\end{figure}

Assume that $\ww$ and $\www$ are $ad$-words and $\ww$ reverses
to $\www$. Then the reversing sequence going from $\ww$ to $\www$ is not
unique in general, but the resulting reversing diagram depends on $\ww$
and $\www$ only. Reversing can easily be turned into a (deterministic)
algorithm by choosing to always reverse the rightmost possible subword.
The algorithm terminates when a word with no subword $\dddd\indi\nno\inv\,\xx$ satisfying $\dddd\indi\nno\inv\,\xx\revv\uuu$ for some $\uuu$ has been obtained. This algorithm is called
\emph{reversing algorithm}. See Figure~\ref{F:X:ReversingDiagram} for an
example.

\begin{figure}[htb]
\begin{picture}(110,32)
\put(5,3){\includegraphics{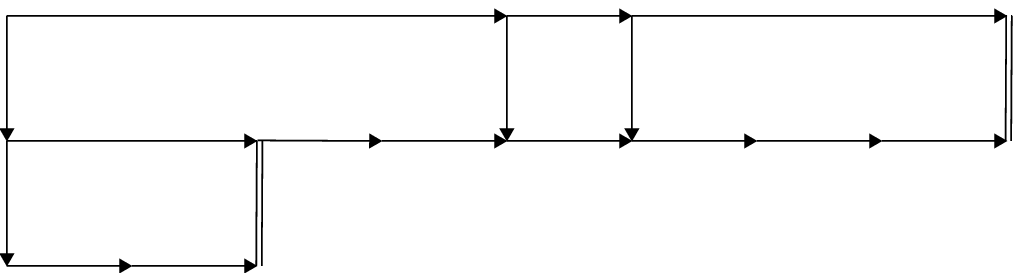}}
{\footnotesize
\put(0,22){$\dddd24$}
\put(0,10){$\dddd34$}
\put(10,0.7){$\dddd14$}
\put(16,13){$\dddd14$}
\put(22,0.8){$\dddd23\inv$}
\put(28,30.5){$\aa13$}
\put(32.5,10){$\varepsilon$}
\put(35,13.5){$\dddd23\inv$}
\put(47,13.5){$\dddd12\inv$}
\put(50.5,22){$\dddd34$}
\put(60.5,30.5){$\aa13$}
\put(60.5,13.8){$\aa13$}
\put(70,22){$\dddd34$}
\put(74,13.5){$\dddd24$}
\put(86,13.5){$\dddd23\inv$}
\put(99,13.5){$\dddd23\inv$}
\put(86,30.5){$\aa24$}
\put(108.5,23){$\varepsilon$}}
\end{picture}
\caption{{\sf \smaller Reversing diagram
of $\dddd34\inv\,\dddd24\inv\,\aa13\,\aa13\,\aa24$. We end with
$\dddd14\,\dddd23\inv\,\dddd23\inv\,\dddd12\inv\,\aa13\,\dddd24\,\dddd23\inv\,\dddd23\inv$.
Each rectangle in the diagram corresponds to one relation $\uu\rev1\uuu$,
hence the number of rectangles is the length of every reversing sequence
$(\ww_0, ..., \ww_\ll)$ from~$\ww_0$ to~$\ww_\ll$: the sequence is not
unique, but its length and the corresponding diagram are.}}
\label{F:X:ReversingDiagram}
\end{figure}

\begin{rema}
Formally, the above notion of reversing is similar to the transformation called
``word reversing'' in~\cite{Dehornoy:GP}. However, similarity is superficial
only: what is common is the idea of iteratively pushing some specific factors to
the right, but the considered factors and the basic switching rules are
completely different.
\end{rema}

The first, easy observation is that reversing transforms a braid word into an equivalent braid word.

\begin{lemm}
\label{L:Reversing:Equivalence}
For $\ww, \www$ $ad$-words, $\ww\revv\www$ implies  $\ww\equiv\www$.
\end{lemm}

\begin{proof}
A simple verification. It is sufficient to 
prove that
$\ww\rev1\www$ implies
$\ww\equiv\www$, hence to prove that $\uu \equiv \uu'$ holds for each pair
$(\uu, \uu')$ of Definition~\ref{D:Reversing}.  
We start with~\eqref{E:Reversing:i}.
Assume first $\indiv\le\indi$.
Relation~\eqref{E:DeltaRelations:AutomorphismQuotient} implies 
\begin{equation}
\label{L:Reversing:Equivalence:1}
\dddd\indi\nno\inv\,\aa\indiii\indiv\,\dddd\indi\nno\equiv\ff\nno\inv(\ff\indi(\aa\indiii\indiv)).
\end{equation}
For~$\indiv<\indi$, we have
$\ff\indi(\aa\indiii\indiv)=\aa\indiiip\indivp$, and 
\eqref{L:Reversing:Equivalence:1} implies
$\dddd\indi\nno\inv\,\aa\indiii\indiv\equiv\aa\indiii\indiv\,\dddd\indi\nno\inv$.
For~$\indiv=\indi$, we have $\ff\indi(\aa\indiii\indiv)=\aa1\indiiip$, and 
\eqref{L:Reversing:Equivalence:1} implies
$\dddd\indi\nno\inv\,\aa\indiii\indiv\equiv\aa\indiii\nno\,\dddd\indi\nno\inv$. 

Assume now  $\indiii\ge\indi$.
Relation~\eqref{E:DeltaRelations:Decomposition} implies $\dddd\indi\nno\equiv\dddd1\indi\inv\,\dddd1\nno$, hence 
\begin{equation}
\label{L:Reversing:Equivalence:2}
\dddd\indi\nno\inv\,\aa\indiii\indiv\,\dddd\indi\nno\equiv\dddd1\nno\inv\,\dddd1\indi\,\aa\indiii\indiv\,\dddd1\indi\inv\,\dddd1\nno.
\end{equation}
For $\indiii>\indi$, \eqref{E:DeltaRelations:DualGenerator} with \eqref{E:DeltaRelations:Commutation} imply that
 $\dddd1\indi$ and $\aa\indiii\indiv$ commute. \!Then, \!\eqref{L:Reversing:Equivalence:2} gives
$\dddd\indi\nno\inv\,\aa\indiii\indiv\,\dddd\indi\nno\equiv\ff\nno\inv(\aa\indiii\indiv)$,
\ie, $\dddd\indi\nno\inv\,\aa\indiii\indiv\equiv\aa\indiiio\indivo\,\dddd\indi\nno\inv$.
For $\indiii=\indi$, Relation \eqref{E:DeltaRelations:DualGenerator} gives
$\dddd1\indi\,\aa\indiii\indiv\,\dddd1\indi\inv=\aa1\indiv$. Then,
\eqref{L:Reversing:Equivalence:2} gives
$\dddd\indi\nno\inv\,\aa\indiii\indiv\,\dddd\indi\nno\equiv\ff\nno\inv(\aa1\indiv)$,
\ie,
$\dddd\indi\nno\inv\,\aa\indiii\indiv\equiv\aa\indivo\nno\,\dddd\indi\nno\inv$.

Next, we consider~\eqref{E:Reversing:iii}.
Relation~\eqref{E:DeltaRelations:Decomposition} implies $\dddd\indiii\nno\equiv\dddd1\indiii\inv\,\dddd1\nno$, hence 
\begin{equation}
\label{L:Reversing:Equivalence:3}
\dddd\indiii\nno\inv\, \dddd\indi\nno\inv\, \dddd\indiii\nno\, \equiv \dddd1\nno\inv\, \dddd1\indiii \,\dddd\indi\nno \, \dddd1\indiii\inv\, \dddd1\nno.
\end{equation}
By \eqref{E:DeltaRelations:Commutation}, $\dddd1\indiii$ and
$\dddd\indi\nno$ commute, hence
$\dddd1\indiii\,\dddd\indi\nno\,\dddd1\indiii\inv\equiv\dddd\indi\nno$
holds. Then, \eqref{L:Reversing:Equivalence:3} implies
$\dddd\indiii\nno\inv\,\dddd\indi\nno\inv\,\dddd\indiii\nno\equiv\ff\nno\inv(\dddd\indi\nno\inv)$.
From the relation~\eqref{E:DeltaRelations:ImageByAutomorphism}, we obtain
$\dddd\indi\nno\inv\,\dddd\indiii\nno\,\equiv\,\dddd\indiii\nno\,\dddd\indio\nnt\inv$.

Finally, we consider~\eqref{E:Reversing:ii}.
By \eqref{E:DeltaRelations:DualGenerator}, we have $\aa\indiii\indiv\equiv\dddd\indiii\indiv\,\dddd\indiii\indivo\inv$.
Relation \eqref{E:DeltaRelations:Decomposition} implies $\dddd\indiii\indiv\equiv\dddd\indiii\nno\,\dddd\indiv\nno\inv$, hence we find
\begin{equation}
\label{L:Reversing:Equivalence:4}
\dddd\indi\nno\inv\,\aa\indiii\indiv\equiv\dddd\indi\nno\inv\,\dddd\indiii\nno\,\dddd\indiv\nno\inv\,\dddd\indiii\indivo\inv.
\end{equation}
By \eqref{E:DeltaRelations:Commutation}, the letters $\dddd\indiv\nno$ and $\dddd\indiii\indivo\inv$ commute.
Moreover, \eqref{E:Reversing:iii} implies that the word $\dddd\indi\nno\inv\,\dddd\indiii\nno$ is equivalent to $\dddd\indiii\nno\,\dddd\indio\nnt\inv$.
Hence, from \eqref{L:Reversing:Equivalence:4}, we obtain the relation $\dddd\indi\nno\inv\,\aa\indiii\indiv\equiv\dddd\indiii\nno\,\dddd\indio\nnt\inv\,\dddd\indiii\indivo\inv\,\dddd\indiv\nno\inv$.
\end{proof}

%
%

\section{Walls}
\label{S:Walls}

%
%

We shall now apply the reversing algorithm of Section~\ref{SS:ReversingDiagram} to 
those words that consist of an $\aa\indi\nn$-dangerous word followed by an
$\aa\indi\nn$-ladder, with the aim of obtaining an equivalent $\sig\ii$-positive
word whenever this is possible.

Once again, the problem is to identify the generic
form of the final words we can obtain. A new type of braid words called \emph{walls}
occurs here, and the main result is that reversing a word consisting of a dangerous
word followed by a ladder always results in a $\sigg$-nonnegative word that is a
wall.

\subsection{Dangerous against ladders: case of length $1$}

We first concentrate on the case when the dangerous word has length $1$, \ie,
it consists of a single negative $d$-letter---the general case will be handled in
Section~\ref{SS:DangerousAgainstLadder:GeneralCase}. In view of Theorems~1
and~2, we shall not only describe the resulting word~$ad$-word, but
also compute both the time and space complexity of the algorithm involved in
the transformation. 

First we introduce now the notion of a \emph{wall}, a weak variant of a ladder.
It comes in two versions called \emph{high} and \emph{low}.

\begin{defi}
\label{D:Wall}
For $\nn\ge3$ and $\indi\le\nnt$, we say that an $ad$-word~$\ww$ is a \emph{high $\aa\indi\nn$-wall lent on~$\aa\indiio\nno$} if there exists a 
decomposition 
$$\ww=\uu\cdot\dddd\indiii\nno\cdot\www\cdot\dddd\indiio\nno\cdot
\vv$$ 
such that
\begin{conditions}
\label{D:Wall:i}
\cond{$\uu$ is a positive $a$-word,}\\
\label{D:Wall:ii}
\cond{$\indiii<\indi$ holds,}\\
\label{D:Wall:iii}
\cond{$\www$ is a $\sig\nnt$-nonnegative $ad$-word,}\\
\label{D:Wall:iv}
\cond{$\vv$ is $\aa\indiio\nno$-dangerous.}
\end{conditions}
We say that an $ad$-word~$\ww$ is a \emph{low $\aa\indi\nn$-wall lent on~$\aa\indiio\nno$} if there exists a 
decomposition 
$$\ww=\uu\cdot\dddd\indiio\nno\cdot\vv$$
such that
\begin{conditions}
\setcounter{equation}{4}
\label{D:Wall:i'}
\cond{$\uu$ is a positive~$a$-word,}\\
\label{D:Wall:ii'}
\cond{$\indiio<\indi$ holds,}\\
\label{D:Wall:iii'}
\cond{$\vv$ is an $\aa\indiio\nno$-dangerous of type $\indi'<\indi$.}
\end{conditions}
In both cases, we write \emph{$\floor\ww$} for the
word denoted $\uu$  above, and
\emph{$\dang\ww$} for the
word denoted $\vv$ above.
\end{defi}

We say that an $a$-word~$\ww$ is an \emph{$\aa\indi\nn$-wall} if it is
either a high or a low~$\aa\indi\nn$-wall. 

Note that the condition satisfied by the letter $\dddd\indiii\nno$  occurring in
the decomposition of a high wall is  the condition satisfied by the
$\aa\indi\nn$-barrier $\aa\indiii\nno$. The same property holds for the letter $\dddd\indiio\nno$ occurring in
the decomposition of a low wall. 

So far we have defined $\aa\pp\nn$-walls for $\pp \le \nnt$ only. We
now consider $\aa\nno\nn$-walls, which are special as are $\aa\nno\nn$-ladders.

\begin{defi}
For $\nn\ge3$, we say that an
$ad$-word~$\ww$ is an \emph{$\aa\nno\nn$-wall lent on~$\aa\indiio\nno$} if $\ww$
can be decomposed as
$\uu\cdot\dddd\indiio\nno\cdot\vv$ with
$\uu$ a positive~$a$-word and $\vv$ an $\aa\indiio\nno$-dangerous word. Then we define $\floor{\ww}=\uu$ and
$\dang{\ww}=\vv$.
\end{defi}

By definition, every $\aa\indi\nn$-wall lent on~$\aa\indiio\nno$ is also an
$\aa\indiii\nn$-wall lent on~$\aa\indiio\nno$ for $\indiii\ge\indi$.

Walls are introduced in order to describe the output of the reversing algorithm
running on those words that consist of an $\aa\indi\nn$-dangerous
word followed by an $\aa\indi\nn$-ladder.

\begin{lemm}
\label{L:DangerousAgainstLadder}
Let $\ww$ be an $\aa\indi\nn$-ladder lent on~$\aa\indiio\nno$ with $\indi\le\nnt$ and $\nn\ge\nobreak3$.
Let $\ww_0\,\xx_1\,\Ldots\,\xx_\hh\,\ww_\hh$ be the decomposition of $\ww$ as a ladder.
Then $\dddd\indi\nno\inv\,\ww$ is equivalent to an $\aa\indi\nn$-wall~$\www$ lent
on~$\aa\indiio\nno$.
The latter can be computed using at most~$\ll$ reversing steps plus one basic
operation, and it satisfies
\begin{conditions}
\label{L:DAL:i}
\cond{$\len{\floor\www}=\len{\ww_0}$,}\\
\label{L:DAL:ii}
\cond{$\len{\dang\www}\le2$,}\\
\label{L:DAL:iii}
\cond{$\len{\www}\le\len{\ww}\plus2(\hho)\plus2\len{\ww_\hh}\plus\len{\dang{\www}}$,}\\
\label{L:DAL:iv}
\cond{$\www$ is a high wall for $\ww_\hh\not=\varepsilon$.}
\end{conditions}
\end{lemm}

\begin{figure}[htb]
\begin{picture}(118,20)
\put(8,3){\includegraphics{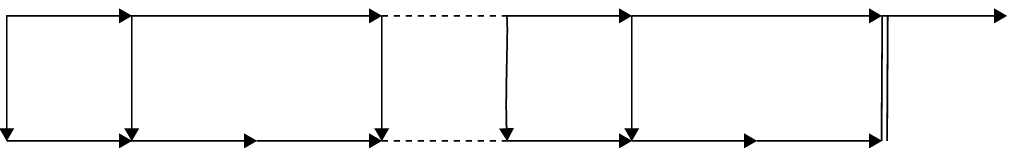}}
{\footnotesize
\put(0,10){$\dddd\indi{\nn-1}$}
\put(14,17.5){$\ww_0$}
\put(14,1.5){$\uu_0$}
\put(26.5,1.5){$\yy_1$}
\put(32,17.5){$\xx_1$}
\put(39,1.5){$\vv_1$}
\put(63,17.5){$\ww_{\hh-1}$}
\put(63,1.5){$\uu_{\hh-1}$}
\put(77,1.5){$\yy_\hh$}
\put(83,17.5){$\xx_\hh$}
\put(90,1.5){$\vv_\hh$}
\put(99,10){$\varepsilon$}
\put(102,17.5){$\ww_\hh$}}
\end{picture}
\caption{{\sf \smaller Reversing $\dddd\indi\nno\inv\,\ww$ into  a wall when
$\ww$ is a ladder (proof of Lemma~\ref{L:DangerousAgainstLadder}).}}
\label{F:DangerousAgainstLadder}
\end{figure}

\begin{proof}
The main idea is illustrated in Figure~\ref{F:DangerousAgainstLadder}: starting with
$\dddd\indi\nno\inv\,\ww$, \ie, with $\dddd\indi\nno\inv\,\ww_0\, \xx_1\,
\Ldots\, \xx_\hh \, \ww_\hh$, we reverse the diagram by pushing the vertical
(negative) $d$-arrows to the right until a wall is obtained. The success at each
elementary step is guaranteed by Lemma~\ref{L:Reversing:Equivalence}. In
general we obtain a high wall. A few particular cases have to be considered
separately, namely when $\ww_\hh$ is empty, in which case we obtain a low wall if
the height~$\hh$ is~$1$.

We start with a description of elementary blocks of the diagram of
Figure~\ref{F:DangerousAgainstLadder}. Write
$\xx_\kk=\aa{\find\kk}{\findi\kk}$ for $\kk=1,\Ldots,\hh$, and put
$\findi0=\nobreak\indi$. Fix $\kk$ with $0\le\kk\le\hho$. Put
$\uu_\kk=\revi{\findi\kk}(\ww_\kk)$, $\yy_\kkp=\dddd{\find\kkp}\nno$
and $\vv_\kkp=\revii{\findi\kk}(\xx_\kkp)$. Then, we have
$\dddd{\findi\kk}\nno\inv\,\ww_\kk\,\xx_\kkp\revv\uu_\kk\,\yy_\kkp\,\vv_\kkp\,\dddd{\findi\kkp}\nno\inv$,
corresponding to the diagram

\begin{figure}[htb]
\begin{picture}(68,20)
\put(14,3){\includegraphics{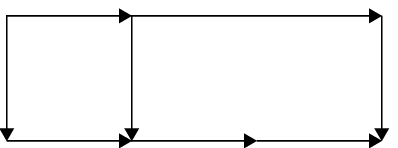}}
{\footnotesize
\put(2.5,10){$\dddd{\findi\kk}{\nn-1}$}
\put(19,17.5){$\ww_\kk$}
\put(19,1.5){$\uu_\kk$}
\put(28.5,10){$\dddd{\findi\kk}{\nn-1}$}
\put(32,1.5){$\yy_{\kk+1}$}
\put(37,17.5){$\xx_{\kk+1}$}
\put(44,1.5){$\vv_{\kk+1}$}
\put(54,10){$\dddd{\findi{\kk+1}}\nno$}}
\end{picture}
\end{figure}

Gathering the reversing diagrams corresponding to the successive values of the
parameter~$\kk$, we precisely obtain the diagram of 
Figure~\ref{F:DangerousAgainstLadder}. Put
$\www_\kk=\uu_\kk\,\yy_\kkp\,\vv_\kkp$ for $0\le\kk\le\hho$.

At this point, we have to consider three slightly different cases.
Assume first $\ww_\hh=\varepsilon$ and $\hh\ge2$, the easiest case,
of which the other two cases will be derived.

Put $\www = \www_0\,...\,\www_\hho$.
By construction, we have $\dddd\indi\nno\inv\,\ww\revv\www$.
Hence, by Lemma~\ref{L:Reversing:Equivalence},  $\dddd\indi\nno\inv\,\ww$ is equivalent to $\www$. We shall now prove that
$\www$ is a wall of the expected type, and that the complexity statements are
satisfied.

As $\ww_\hh$ is empty, the last letter of~$\ww$ is $\xx_\hh$. 
This implies $\xx_\hh=\aa\indiio\nno$, hence $\yy_\hh=\dddd\indiio\nno$.
Put $\wwww=\vv_1\,\www_2\,...\,\www_\hht\,\uu_\hho$.
By construction, we have 
\[
\www=\uu_0\cdot\dddd{\find1}\nno\cdot\wwww\cdot\dddd\indiio\nno\cdot\vv_\hh.
\]
We shall now check that $\www$ is a high $\aa\indi\nn$-wall lent on~$\aa\indiio\nno$.
As the image of an $a$-letter under~$\revi\indi$ is an $a$-letter, the word~$\uu_0$ is a positive~$a$-word whose length is~$\len{\ww_0}$.
Hence \eqref{D:Wall:i} and \eqref{L:DAL:i} are satisfied.

Next, by definition of a ladder, the letter~$\xx_1$ is an $\aa\indi\nn$-barrier, hence $\find1<\indi$ holds, \ie, \eqref{D:Wall:ii} is satisfied.

As the words $\uu_\kk$, $\yy_\kkp$, $\vv_\kkp$ are $\sig\nnt$-nonnegative, the word~$\wwww$ is also 
$\sig\nnt$-nonnegative. So \eqref{D:Wall:iii} holds.

Now, we recall that $\vv_\hh$ is equal to~$\dddd{\findi\hho-1}\nnt\inv\,\dddd{\find\hh}\nnt\inv$ with $\find\hh=\indiio$. 
By definition of a ladder, the letter~$\xx_\hh$ is an $\aa{\findi\hho}\nn$-barrier.
Therefore, we have $\indiio<\findi\hho$, which implies $\findi\hho\minus1\ge\indiio$.
Hence $\vv_\hh$ is $\aa\indiio\nno$-dangerous of length~$2$. So \eqref{D:Wall:iv} and \eqref{L:DAL:ii} are satisfied.

Finally, for \eqref{L:DAL:iii}, we compute 
\[\len{\www_\kk}=\len{\uu_\kko}\plus\len{\yy_\kk}\plus\len{\vv_\kk}=\len{\ww_\kko}\plus1\plus2=\len{\ww_\kko\,\xx_\kk}\plus2.\]
Then, as $\ww_\hh$ is empty, we obtain
\[\len{\www}=\sum_{\kk=0}^\hho\len{\www_\kk}=\sum_{\kk=0}^\hho\len{\ww_\kk\,\xx_\kkp}\plus2\hh=\len{\ww}\plus2\hh.\]
As in this case $\ww_\hh$ is empty and the length of~$\dang\www$ is~$2$, \ie, the length of~$\vv_\hh$ is~$2$, Condition \eqref{L:DAL:iii} holds.
So the case of~$\ww_\hh$ empty with $\hh \ge 2$ is completed, except for the time
complexity analysis.

Assume now $\ww_\hh=\varepsilon$ and $\hh=1$. 
Then $\www$ is equal to~$\uu_0\cdot\dddd\indiio\nno\cdot\vv_1$.
As in the previous case, the word~$\uu_0$ is a positive~$a$-word of length $\ww_0$ and we have $\len{\www}=\len{\ww}\plus2$.
The word~$\vv_1$ is equal to~$\dddd\indio\nnt\inv\,\dddd\indiio\nnt\inv$, which is $\aa\indiio\nno$-dangerous of type $\indio$ and has length $2$.
Therefore, $\www$ is a low $\aa\indi\nn$-wall lent on~$\aa\indiio\nno$ satisfying \eqref{L:DAL:i}, \eqref{L:DAL:ii} and~\eqref{L:DAL:iii}.

Assume finally $\ww_\hh\not=\varepsilon$. 
Then, we decompose $\ww_\hh$ as $\wwww_\hh\,\aa\indiio\nno$.
Put \[\www=\www_0\,...\,\www_\hho\,\wwww_\hh\,\dddd\indiio\nno\,\dddd\indiio\nnt\inv,\]
and $\wwww=\vv_1\,\www_2\,...\,\www_\hho\,\wwww_\hh$.
We have $\www=\uu_0\cdot\,\dddd{\findi1}\nno\cdot\,\wwww\cdot\,\dddd\indiio\nno\cdot\,\dddd\indiio\nnt\inv$.
Then \eqref{D:Wall:i}, \eqref{D:Wall:ii}, \eqref{D:Wall:iii} are checked as in the case $\ww_\hh = \varepsilon$.
By construction, the word~$\dddd\indiio\nnt\inv$ is $\aa\indiio\nno$-dangerous of length $1$.
So \eqref{D:Wall:iv} and \eqref{L:DAL:ii} are satisfied.
Then, by definition of $\www$,  \eqref{L:DAL:iv} holds.
We check now
\eqref{L:DAL:iii}. Starting form
$\len{\www_\kk}=\len{\ww_\kko\,\xx_\kk}\plus2$, we obtain
\[\len{\www}=\sum_{\kk=0}^\hho\len{\www_\kk}\plus\len{\wwww_\hh}\plus2=\sum_{\kk=0}^\hho\len{\ww_\kk\,\xx_\kkp}\plus\len{\ww_\hh}\plus2\hh\plus1=\len{\ww}\plus2\hh\plus1.\]
As $\ww_\hh$ is not empty, we have $\len{\ww_\hh}\ge1$, hence $2\len{\ww_\hh}\ge2$.
Moreover, in this case, the length of~$\dang\www$ is $1$.
Therefore, we get $3\le2\len{\ww_\hh}\plus\len{\dang\www}$, and eventually find
\[ \len{\www}\le\len{\ww}\plus2(\hho)\plus2\len{\ww_\hh}\plus\len{\dang\www}.\]

So, all cases have been considered. It only remains to consider the time complexity.
In the first and second cases, at most $\len{\ww}$ reversing operations are needed.
In the last case---$\ww_\hh \not= \varepsilon$---at most $\len{\ww}$ reversing operations are needed, plus the decomposition of~$\last{\ww}_\hh$ into two $d$-letters.
\end{proof}

\begin{exam}
We saw in Example~\ref{X:ExistLadder} that the word $\ww=\aa34\,\aa13\,\aa13\,\aa24\,\aa34$ is an $\aa25$-ladder lent on $\aa34$.
Let us compute the $\aa25$-wall lent on $\aa34$ that is equivalent to $\dddd25\inv\,\ww$.
Applying the reversing algorithm to $\dddd24\inv\,\ww$ gives
\[\wwww=\aa23\,\dddd14\,\dddd13\inv\,\dddd12\inv\,\aa14\,\dddd24\,\dddd23\inv\,\dddd23\inv\,\aa34\quad\text{(see Figure~\ref{F:X:L:DangerousAgainstLadder})}.\]
The word $\wwww$ is not a wall because its last letter does not have the correct form. However,
if we replace the last letter $\aa34$ of $\wwww$ by $\dddd34$ we obtain the high wall
\[\www=\aa23\,\cdot\,\dddd14\,\cdot\,\dddd13\inv\,\dddd12\inv\,\aa14\,\dddd24\,\dddd23\inv\,\dddd23\inv\,\cdot\,\dddd34\,\cdot\,\varepsilon.\]
The word~$\floor{\www}$ of $\www$ is $\aa23$, whereas $\dang{\www}$ is
empty.
\end{exam}

\begin{figure}[htb]
\begin{picture}(120,20)
\put(5,3){\includegraphics{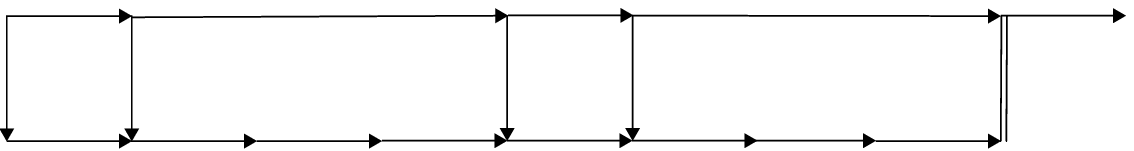}}
\footnotesize
\put(10,18){$\aa34$}
\put(35,18){$\aa13$}
\put(60,18){$\aa13$}
\put(85,18){$\aa24$}
\put(110,18){$\aa34$}

\put(0,10){$\dddd24$}
\put(13,10){$\dddd24$}
\put(51,10){$\dddd34$}
\put(64,10){$\dddd34$}
\put(104,10){$\varepsilon$}

\put(10,1){$\aa23$}
\put(23,1){$\dddd14$}
\put(35,1){$\dddd13\inv$}
\put(47,1){$\dddd12\inv$}
\put(60,1){$\aa14$}
\put(74,1){$\dddd24$}
\put(85,1){$\dddd23\inv$}
\put(97,1){$\dddd23\inv$}
\end{picture}
\caption{{\sf \smaller Reversing diagram of the
word~$\dddd24\inv\,\aa34\,\aa13\,\aa13\,\aa24\,\aa34$.}}
\label{F:X:L:DangerousAgainstLadder}
\end{figure}

%
%

\subsection{Dangerous against wall}

In the previous section, we studied the action of the reversing algorithm running on a word~$\uu\,\ww$ in the special case when $\uu$ is an~$\aa\indi\nn$-dangerous word of length $1$ and $\ww$ is an $\aa\indi\nn$-ladder.
We proved that the output word is an $\aa\indi\nn$-wall.
Before turning to the general case of an initial dangerous word with an arbitrary
length---that will be done in the next section---we consider here the case of an
$\aa\indi\nn$-dangerous word  of length $1$ followed by an $\aa\indi\nn$-wall.
The result is that the output word is again an $\aa\indi\nn$-wall. This shows
that, contrary to the family of ladders, the family of walls enjoys good closure
properties that will make inductive arguments possible.

We start with a technical result that will be used twice in
the proof of Lemma~\ref{L:DangerousAgainstWall}.

\begin{lemm}
\label{L:For:L:DangerousAgainstWall}
Assume $\nn\ge3$, that $\ww$ is a positive $a$-word containing an
$\aa\indi\nn$-barrier and that $\indiii<\indi$ holds. Then
the word $\dddd\indi\nno\inv\,\ww\,\dddd\indiii\nno$ reverses to
the
$ad$-word 
\[\uu\,\dddd\indv\nno\,\uuu\,\dddd\indiii\nno\,\dddd\indivo\nnt\inv,\]
which is obtained in at most $\len{\ww}+1$ steps and satisfies
\begin{conditions}
\label{L:FLDAW:i}
\cond{$\indv<\indi$ and $\indiii<\indiv$,}\\
\label{L:FLDAW:ii}
\cond{$\uu$ is a positive~$a$-word with $\len{\uu}<\len{\ww}$,}\\
\label{L:FLDAW:iii}
\cond{$\uuu$ is a $\sig\nnt$-nonnegative $ad$-word,}\\
\label{L:FLDAW:iv}
\cond{$\len{\uu\,\dddd\indv\nno\,\uuu\,\dddd\indiii\nno\,\dddd\indivo\nnt\inv}\le\len{\ww\,\dddd\indiii\nno}\plus2\len{\ww}\minus2\len{\uu}\plus1$.}
\end{conditions}
\end{lemm}

\begin{figure}[htb]
\begin{picture}(94,20)
\put(8,3){\includegraphics{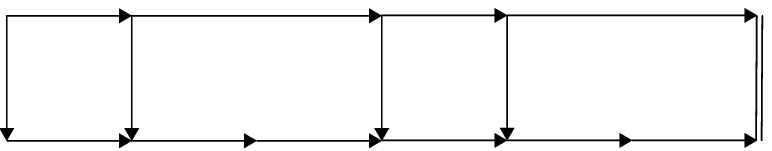}}
{\footnotesize
\put(0,10){$\dddd\indi{\nn-1}$}
\put(14,17.5){$\www$}
\put(11,1){$\revi\indi(\www)$}
\put(24,1){$\dddd\indv{\nn-1}$}
\put(32,18){$\aa\indv\indvi$}
\put(34.5,1){$\revii\indi(\aa\indv\indvi)$}
\put(37.5,10){$\dddd\indvi{\nn-1}$}
\put(52,18){$\vv$}
\put(52,1){$\vvv$}
\put(60.5,10){$\dddd\indiv{\nn-1}$}
\put(68,18){$\dddd\indiii{\nn-1}$}
\put(62,1){$\dddd\indiii{\nn-1}$}
\put(77,1){$\reviii\indiv$}
\put(86.5,10){$\varepsilon$}}
\end{picture}
\caption{{\sf \smaller Reversing $\dddd\indi\nno\inv\,\ww$ into a wall in the case when $\ww$ is a
wall (proof of Lemma~\ref{L:For:L:DangerousAgainstWall}).}}
\label{F:DangerousAgainstWall}
\end{figure}

\begin{proof}
Write $\ww$ as $\www\,\aa\indv\indvi\,\vv$ where $\www$ is the maximal
prefix of~$\ww$ that contains no $\aa\indi\nn$-barrier, and with
$\aa\indv\indvi$ an $\aa\indi\nn$-barrier.
The argument is illustrated in Figure~\ref{F:DangerousAgainstWall}: starting with
$\dddd\indi\nno\inv\,\ww\,\dddd\indiii\nno$, we reverse the diagram by pushing
the vertical (negative) $d$-arrows to the right until a wall is obtained. The
success at each elementary step is guaranteed by
Lemma~\ref{L:Reversing:Equivalence}.

As $\www$ does not contains any $\aa\indi\nn$-barrier, we have $\dddd\indi\nno\inv\,\www\revv\revi\indi(\www)\,\dddd\indi\nno\inv$.
By construction $\aa\indv\indvi$ is an $\aa\indi\nn$-barrier, \ie, $\indv<\indi<\indvi$ holds.
We deduce
$$\dddd\indi\nno\inv\,\aa\indv\indvi\revv\dddd\indv\nno\revii\indi
(\aa\indv\indvi)\,\dddd\indvi\nno\inv.$$
By definition of elementary
reversing steps, we obtain
$\dddd\indvi\nno\inv\,\vv\revv\vvv\,\dddd\indiv\nno\inv$ for some
$ad$-word~$\vvv$ with $\len{\vvv}\le3\len{\vv}$ and some $\indiv\ge\indv'$. The
hypothesis $\indiii<\indi$ together with $\indi<\indvi$ and $\indvi\le\indiv$
implies $\indiii<\indiv$. Hence
$\dddd\indiv\nno\inv\,\dddd\indiii\nno\revv\dddd\indiii\nno\,\reviii\indiv$
holds.

Write $\uu=\revi\indi(\www)$ and $\uuu=\revii\indi(\aa\indv\indvi)\,\vvv$.
By construction, we have 
$$\dddd\indi\nno\inv\,\ww\,\dddd\indiii\nn\revv
\uu\,\dddd\indv\nno\,\uuu\,\dddd\indiii\nno\,\dddd\indivo\nno\inv,$$
and we claim that the latter word has the expected properties.

Condition \eqref{L:FLDAW:i} is an immediate consequence of the above results.

As the image of an $a$-letter under $\revi\indi$ is an $a$-letter, the word~$\uu$ is a positive~$a$-word of length $\len{\www}$.
By definition, the word~$\www$ is a proper prefix of~$\ww$.  
Then $\len{\www}<\len{\ww}$ holds, \ie, \eqref{L:FLDAW:ii} is satisfied.

By definition of elementary reversing steps, the image of a positive~$a$-word under~$\revi{}$ and $\revii{}$ is $\sig\nnt$-nonnegative, hence the word~$\vvv$ is $\sig\nnt$-nonnegative.
As $\revii\indi(\aa\indv\indvi)$ is $\sig\nnt$-nonnegative, the word $\uuu$ is
$\sig\nnt$-nonnegative, \ie, \eqref{L:FLDAW:iii} holds.

For \eqref{L:FLDAW:iv}, we compute 
\[\len{\uu\,\dddd\indv\nno\,\uuu\,\dddd\indiii\nno\,\dddd\indivo\nnt\inv}=\len{\www}\plus\len{\uuu}\plus3=\len{\www}\plus\len{\vvv}\plus5.\]
By construction of~$\vvv$, we have $\len{\vvv}\le3\len{\vv}$.
Using $\len{\ww}=\len{\www}\plus1\plus\len{\vv}$, we deduce
\begin{align*}
\len{\uu\,\dddd\indv\nno\,\uuu\,\dddd\indiii\nno\,\dddd\indivo\nnt\inv}
&\le3\len{\ww}\minus2\len{\www}\plus2\\
&=\len{\ww\,\dddd\indi\nno}\plus2(\len{\ww}\minus\len{\www})\plus1,
\end{align*}
which is the expected inequality since $\len{\www}$ is equal
to~$\len{\uu}$ by~\eqref{L:FLDAW:ii}.

An easy bookkeeping argument gives the bound on the number of steps in the
revering process.
\end{proof}

We are now able to establish the main result of this section.

\begin{lemm}
\label{L:DangerousAgainstWall}
Assume that $\ww$ is an $\aa\indi\nn$-wall lent on~$\aa\indiio\nno$.
Then $\dddd\indi\nno\inv\ww$ reverses in at most $\len{\floor{\ww}}+1$ steps to an $\aa\indi\nn$-wall
$\www$ satisfying
\begin{conditions}
\label{L:DAW:i}
\cond{$\len{\floor\www}\le\len{\floor\ww}$,}\\
\label{L:DAW:ii}
\cond{$\len{\dang\www}\le\len{\dang\ww}\plus1$,}\\
\label{L:DAW:iii}
\cond{$\len{\www}\le\len{\ww}\plus2\len{\floor\ww}\minus2\len{\floor\www}+1$,}\\
\label{L:DAW:iv}
\cond{$\www$ is a high wall whenever $\ww$ is a high wall.}
\end{conditions}
\end{lemm}

\begin{proof}
Assume that $\ww$ is a low wall.
Then $\ww$ admits the decomposition~$\ww=\floor{\ww}\,\dddd\indiio\nno\,\dang{\ww}$.
By definition of a wall, we have $\indiio<\indi$.
First, assume in addition that $\floor\ww$ contains no $\aa\indi\nn$-barrier.
Then, the reversing process gives
\[
 \dddd\indi\nno\inv\,\ww\,\revv\,\revi\indi{(\floor{\ww})}\,\dddd\indi\nno\inv\,\dddd\indiio\nno\,\dang{\ww}\,\revv\,\revi\indi{(\floor{\ww})}\,\dddd\indiio\nno\,\dddd\indio\nnt\inv\,\dang{\ww}.
\]
Write $\www=\uu\cdot\dddd\indiio\nno\cdot\vv$ with $\uu=\revi\indi{(\floor{\ww})}$ and $\vv=\dddd\indio\nnt\inv\,\dang{\ww}$.
As the image of a positive $a$-letter under $\revi{}$ is a positive $a$-letter, the word $\uu$ is a positive $a$-word of length~$\len{\floor{\ww}}$.
Then, $\indiio<\indi$ implies that $\www$ is a low $\aa\indi\nn$-wall lent on $\aa\indiio\nno$ satisfying \eqref{L:DAW:i} and  \eqref{L:DAW:ii} hold---$\dang{\www}=\vv$.
Condition \eqref{L:DAW:iii} is a direct consequence of the construction of
$\www$ together with $\len{\vv}=\len{\dang{\ww}}+1$.

Next, assume in addition that $\floor\ww$ contains an $\aa\indi\nn$-barrier.
By Lemma~\ref{L:For:L:DangerousAgainstWall} applied to $\floor{\ww}\,\dddd\indiio\nno$, there exist two words $\uu$ and $\uuu$, and two integers $\indiv$ and $\indv$ satisfying
\[
\dddd\indi\nno\inv\,\ww\,\revv\,\uu\,\dddd\indv\nno\,\uuu\,
\dddd\indiio\nno\,\dddd\indivo\nnt\inv\,\dang\ww.
\]
Write $\www=\uu\cdot\dddd\indv\nno\cdot\uuu\cdot\dddd\indiio\nno\cdot\vv$, with $\vv=\dddd\indivo\nnt\inv\,\dang\ww$.
Condition \eqref{L:FLDAW:i} implies that $\vv$ is an
$\aa\indiio\nno$-dangerous word of length at most~$\len{\dang{\ww}}+1$,
and that $\indv<\indi$ holds. Then, \eqref{L:FLDAW:ii} and
\eqref{L:FLDAW:iii} imply that $\ww$ is a high $\aa\indi\nn$-wall lent on
$\aa\indiio\nno$ and it satisfies \eqref{L:DAW:i} and \eqref{L:DAW:ii}.
Using
\eqref{L:FLDAW:iv}, we compute 
\[
 \len{\www}=\len{\floor{\ww}\,\dddd\indiio\nno}+2\len{\floor{\ww}}-2\len{\uu}+\len{\vv},
\]
which implies \eqref{L:DAW:iii} since we have $\floor{\www}=\uu$ and
$\dang{\www}=\dddd\indivo\nnt\inv\,\dang\ww=\vv$.

Assume now that $\ww$ is a high wall.
Then $\ww$ admits the decomposition 
\[\ww=\floor\ww\,\dddd\indiii\nno\,\wwww\,\dddd\indiio\nno\,\dang{\ww},\]
 with $\indiii<\indi$.
First, assume that $\floor\ww$ contain no $\aa\indi\nn$-barrier.
Then, reversing process gives
\[
 \dddd\indi\nno\,\ww\revv\revi\indi(\floor{\ww})\,\dddd\indiii\nno\,\dddd\indio\nnt\inv\,\wwww\,\dddd\indiio\nno\,\dang{\ww}.
\]
Write $\www=\revi\indi(\floor{\ww})\cdot\dddd\indiii\nno\cdot\dddd\indio\nnt\inv\,\wwww\cdot\dddd\indiio\nno\cdot\dang{\ww}$.
A direct verification, based on the fact that $\ww$ is an high $\aa\indi\nn$-wall lent on $\aa\indiio\nno$, gives that $\www$ is an high $\aa\indi\nn$-wall lent on $\aa\indiio\nno$ satisfying \eqref{L:DAW:i}, \eqref{L:DAW:ii} and~\eqref{L:DAW:iv}.
For \eqref{L:DAW:iii}, we compute $\len{\www}=\len{\ww}+1$.

Assume now that $\floor\ww$ contains an $\aa\indi\nn$-barrier.
Then, by Lemma~\ref{L:For:L:DangerousAgainstWall} applied to $\floor\ww\,\dddd\indiii\nno$, there exists two words
$\uu$, $\uuu$ and two integers $\indiv$, $\indv$ satisfying
\[
 \dddd\indi\nno\,\ww\revv\,\uu\,\dddd\indv\nno\,\uuu\,\dddd\indiiio\nno\,\dddd\indivo\nnt\inv\,\wwww\,\dddd\indiio\nno\,\dang{\ww}.
\]
Write $\www=\uu\cdot\dddd\indv\nno\cdot\uuu\,\dddd\indiiio\nno\,\dddd\indivo\nnt\inv\,\wwww\cdot\dddd\indiio\nno\cdot\dang{\ww}$.
Condition~\eqref{L:FLDAW:iii} implies that the word $\uuu\,\dddd\indiiio\nno\,\dddd\indivo\nnt\inv\,\wwww$ is $\sig\nnt$-nonnegative, 
and even $\sig\nnt$-positive. Hence, a direct verification, based on the fact
that $\ww$ is an high $\aa\indi\nn$-wall lent on $\aa\indiio\nno$, shows that
$\www$ is a high $\aa\indi\nn$-wall lent on $\aa\indiio\nno$ and it satisfies
\eqref{L:DAW:i}, \eqref{L:DAW:ii} and~\eqref{L:DAW:iv}. Using
\eqref{L:FLDAW:iv}, we compute
\[
 \len{\www}\le\len{\floor{\ww}\,\dddd\indiii\nno}+2\len{\floor{\ww}}-2\len{\uu}+1+\len{\wwww}+1+\len{\dang{\ww}},
\]
which implies \eqref{L:DAW:iii} since $\floor{\www}=\uu$.

As for the number of reversing steps, it follows from an easy bookkeeping argument using Lemma~\ref{L:For:L:DangerousAgainstWall}.
\end{proof}

\subsection{Dangerous against ladders: the general case}

In the previous section, we studied the action of the reversing algorithm running on a word~$\uu\,\ww$ in the special case when 
$\uu$ is $\aa\indi\nn$-dangerous of length $1$ and $\ww$ is an
$\aa\indi\nn$-ladder. We proved that the output word is an $\aa\indi\nn$-wall. The aim of this section is to describe the reversing
algorithm in the general  case, \ie, for a  dangerous word of arbitrary length.

\label{SS:DangerousAgainstLadder:GeneralCase}

\begin{prop}
\label{P:DangerousAgainstLadder}
Assume that $\ww$ is an $\aa\indi\nn$-ladder lent on~$\aa\indiio\nno$ and $\uu$ be an $\aa\indi\nn$-dangerous word, with $\nn \ge 3$.
Then $\uu\,\ww$ is equivalent to
an $\aa\indi\nn$-wall~$\www$ lent on~$\aa\indiio\nno$. It can be
computed using at most $\len{\uu}\,\len{\ww}$ reversing steps, plus one basic
operation, hence in time $O(\len{\uu}\len{\ww}+1)$, and it satisfies
\begin{conditions}
\label{P:DAL:i}
\cond{$\len{\dang\www}\le\len{\uu}\plus1$,}\\
\label{P:DAL:ii}
\cond{$\len{\www}\le3\len\ww\plus\len\uu\minus1$.}
\end{conditions}
Moreover, if $\ww$ is an $\aa\indi\nn$-ladder lent on~$\aa\nnt\nno$ but different from $\aa\nnt\nno$, then $\www$ admits the decomposition $\www=\wwww\,\dddd\nnt\nno$, where $\wwww$ is a $\sig\nnt$-positive word.
\end{prop}

\begin{proof}
All ladders and walls in this proof are supposed to be lent
on~$\aa\indiio\nno$. We shall construct an $\aa\indi\nn$-wall~$\www$ that
is equivalent to~$\uu\,\ww$ by induction on the length of~$\uu$.

Assume first $\indi\le\nnt$.
Then $\uu$ is not empty.
Write $\uu$ as $\dddd{\findi\dd}\nno\inv\cdot...\cdot\dddd{\findi1}\nno\inv$.
Define $\wwl1$ to be the $\aa{\findi1}\nn$-wall of Lemma~\ref{L:DangerousAgainstLadder} that is equivalent to~$\dddd{\findi1}\nno\inv\,\ww$.
Starting from $\wwl1$, we inductively define $\wwl\kkp$ to be the
$ad$-word obtained by reversing~$\dddd{\findi\kkp}\nno\inv\,\wwl\kk$.

We claim that $\wwl\kk$ is an $\aa{\findi\kk}\nn$-wall.
Indeed, by definition of a wall, the relation
$\findi\kk\ge\findi\kko$ implies that $\wwl\kko$ is also an
$\aa{\findi\kk}\nn$-wall. Then Lemma~\ref{L:DangerousAgainstWall}
guarantees that  $\wwl\kk$ is an $\aa{\findi\kk}\nn$-wall.

By construction, we have $\uu\,\ww\equiv\wwl\dd$. We shall now prove that $\wwl\dd$ satisfies the
complexity statements.

Lemma~\ref{L:DangerousAgainstLadder} gives $\len{\dang{\wwl1}}\le2$.
For $1\le\kk\le\ddo$, $\eqref{L:DAW:ii}$ implies $ \len{\dang{\wwl\kkp}}\le\len{\dang{\wwl\kk}}\plus1$.
Hence, $\len{\dang{\wwl\dd}}\le\len{\uu}\plus1$ holds, \ie, \eqref{P:DAL:i} is satisfied.

Let $\ww_0\,\xx_1\,\Ldots\,\xx_\hh\,\ww_\hh$ be the decomposition of
the $\aa\indi\nn$-ladder~$\ww$. Then, by \eqref{L:DAW:iii},  we have for each $\kk\ge1$
\begin{equation}
\label{P:DAL:1}
\len{\wwl\kkp}\le\len{\wwl\kk}\plus2\len{\floor{\wwl\kk}}\minus2\len{\floor{\wwl\kkp}}\plus1.
\end{equation}
Gathering the various relations~\eqref{P:DAL:1} for $\kk=1,\Ldots,\ddo$, we obtain
\begin{align*}
\len{\wwl\dd}=\len{\wwl\dd}&\le\len{\wwl1}\plus2\len{\floor{\wwl1}}\minus2\len{\floor{\wwl\dd}}\plus\ddo\\
&\le\len{\wwl1}\plus2\len{\floor{\wwl1}}\plus\ddo
\end{align*}
By \eqref{L:DAL:i}, we have $\len{\floor{\wwl1}}=\len{\ww_0}$, hence
\[\len{\wwl\dd}\le\len{\wwl1}\plus2\len{\ww_0}\plus\ddo.\]
Condition~\eqref{L:DAL:iii} implies $\len{\wwl1}\le\len{\ww}\plus2(\hho)\plus2\len{\ww_\hh}\plus\len{\dang{\wwl1}}$.
Using the relation
$\len{\ww}\le\len{\ww_0}\plus\hh\plus\len{\ww_\hh}$, we obtain
\[\len{\wwl\dd}\le3\len{\ww}\plus\dd\plus\len{\dang{\wwl1}}\minus3.\]
By construction, $\dd$ is the length of $\uu$.
As \eqref{L:DAL:ii} implies $\len{\dang{\wwl1}}\le\nobreak2$, we find
\[\len{\wwl\dd}\le3\len\ww\plus\len\uu\minus1,\]
which completes the case $\pp \le \nnt$ writing $\www=\wwl\dd$.

Assume now $\indi=\nno$.
Then the word~$\uu$ is empty.
Put $\ww=\wwww\,\aa\indiio\nno$, and write $\www=\wwww\,\dddd\indiio\nno\,\dddd\indiio\nnt\inv$.
The word $\www$ is clearly an $\aa\nno\nn$-wall lent on~$\aa\indiio\nno$
and all complexities statements are satisfied. Moreover, for $\indii=\nno$ and
$\ww\not=\aa\nnt\nno$, Lemma~\ref{L:LastLetter}$(iii)$ implies that 
$\wwww$ ends with $\aa\indv\nno$ for some $\tt$, hence it is
$\sig\nnt$-positive. Then $\www$ has the expected properties.

Finally, assume $\indi\not=\nno$, $\indii=\nno$ and
$\ww\not=\aa\nnt\nno$. Then $\uu$ is not empty.
By hypothesis, the last letter of§~$\ww$ is~$\aa\nnt\nno$, which is not a
barrier. Hence the word $\ww_\hh$ is not empty and its last
letter~$\aa\nnt\nno$. Then \eqref{L:DAL:iv} implies that
the wall
$\wwl1$ is high. Hence,
\eqref{L:DAW:iv} implies that the wall $\wwl\kk$ is high for every
$\dd\ge\kk\ge1$, and, therefore, $\www$ is a high wall. By definition of a high
wall, $\www$ can be expressed as
$\uu\,\dddd\indiii\nno\,\wwh\,\dddd\nnt\nno$. By construction,
$\uu\,\dddd\indiii\nno\,\wwh$ is a $\sig\nnt$-positive word, so $\www$ has
all expected properties.

As for the time complexity upper bound, it follows from an easy
bookkeeping argument using Lemmas~\ref{L:DangerousAgainstLadder}
and~\ref{L:DangerousAgainstWall}, and the fact that the cost of one 
reversing step is $O(1)$.
\end{proof}

\begin{exam}
Let $\ww$ to be the $\aa37$-ladder $\aa46\,\aa14\,\aa26$ and $\uu$ to be 
the $\aa37$-dangerous word $\dddd56\inv\,\dddd36\inv\,\dddd36\inv$.
The reversing diagram of~$\uu\,\ww$ is displayed in Figure~\ref{F:X:DangerousAgainsLadder}.

\begin{figure}[htb]
\begin{picture}(144,41)
\put(0,2){\includegraphics{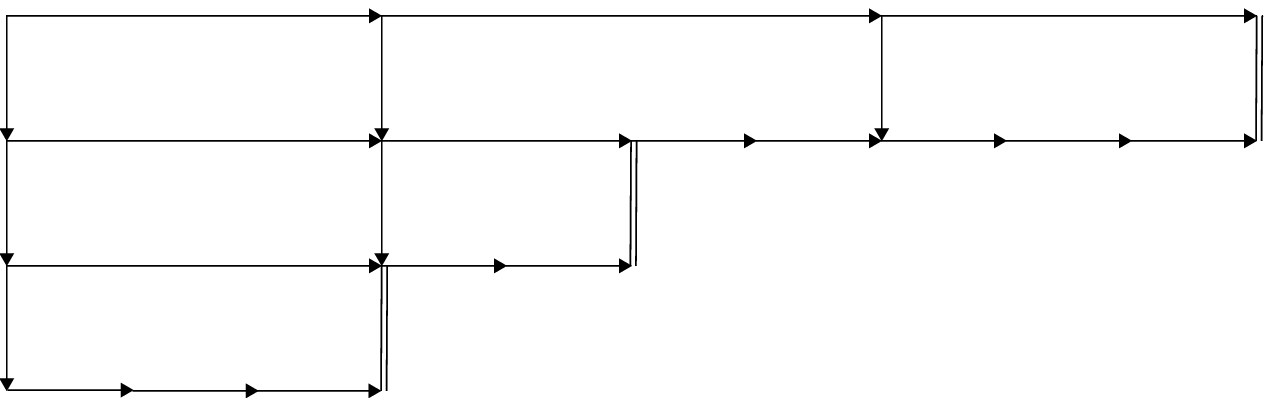}}
{\footnotesize

\put(17,42){$\aa46$}
\put(61,42){$\aa14$}
\put(105,42){$\aa26$}

\put(1,34){$\dddd36$}
\put(33.5,34){$\dddd36$}
\put(84,34){$\dddd46$}
\put(125,34){$\varepsilon$}

\put(17,29.5){$\aa35$}
\put(48,29.5){$\dddd16$}
\put(68,29.5){$\dddd25\inv$}
\put(80,29.5){$\dddd13\inv$}
\put(94,29.5){$\dddd26$}
\put(106,29.5){$\dddd35\inv$}
\put(118,29.5){$\dddd25\inv$}

\put(1,21){$\dddd36$}
\put(33.5,21){$\dddd36$}
\put(62,21){$\varepsilon$}

\put(17,17){$\aa46$}
\put(43,17){$\dddd16$}
\put(55,17){$\dddd25\inv$}

\put(1,8){$\dddd56$}
\put(36,8){$\varepsilon$}

\put(5,4){$\dddd46$}
\put(17,4){$\dddd45\inv$}
\put(30,4){$\dddd45\inv$}
}
\end{picture}
\caption{{\sf \smaller Reversing $\uu\,\ww$ into a wall. Here $\uu$ is the
$\aa37$-dangerous  word~$\dddd56\inv\,\dddd36\inv\,\dddd36\inv$ and
$\ww$ is the
$\aa37$-ladder~$\aa46\,\aa14\,\aa26$, which is lent on $\aa26$.
With the notation of Proposition~\ref{P:DangerousAgainstLadder}, 
$\wwl1$ is the word 
$\aa35\,\dddd16\,\dddd25\inv\,\dddd13\inv\,\dddd26\,\dddd35\inv\,\dddd25\inv$:
it can be read (from left to right) on the third row from the bottom.
Then $\wwl2$ is the word
$\aa46\,\dddd16\,\dddd25\inv\,\dddd25\inv\,\dddd13\inv\,\dddd26\,\dddd35\inv\,\dddd25\inv$:
it can be read on the second row, continuing on the third row when 
the vertical $\varepsilon$-labeled edge is met. Finally $\wwl3$ is the word
$\dddd46\,\dddd45\inv\,\dddd45\inv\,\dddd16\,\dddd25\inv\,\dddd25
\inv\,\dddd13\inv\,\dddd26\,\dddd35\inv\,\dddd25\inv$: it can be read on
the bottom row, continued on the second row, and finally on the
third row. The point is that we had three negative $d_{..., 6}$-letters at first and
that, at each step, we get rid of one of them, ending with a word that
contains negative $\dddd{...}{\qq}$-letters for
$\qq \le 5$ only.}}
\label{F:X:DangerousAgainsLadder}
\end{figure}
\end{exam}

We conclude the section with one more technical result, which provides the precise basic step needed
in the inductive definition of our final normal form~$\NF\nn$.

\begin{lemm}
\label{L:KeyLemma}
Assume $\nn \ge 3$, and  that 

- $(\ww_\brdi,\Ldots,\ww_1)$ is the $\ff\nn$-splitting of a normal word~$\ww$, with~$\brdi \ge 3$, 

- $\uu_\brdi$ is a $\last{\ww}_\brdi$-dangerous word,

- $\brdii$ is a number in $\{\brdi,\Ldots,3\}$.

\noindent
Then, there exists 

- a $\sig\nno$-nonnegative word $\www$,

- an $\last{\ww}_\brdii$-dangerous word~$\uu_\brdii$, 

\noindent both computable in time
$O(\len{\uu_\brdi}\len{\ww}^2)$, that satisfy
\begin{equation}
\label{E:L:KeyLemma}
 \dddd1\nn^{-\brdi+3}\ff\nn^\brdio(\uu_\brdi)\,\ff\nn^\brdit(\ww_\brdio)\Ldots\ww_1\equiv\www\cdot \dddd1\nn^{-\brdii+3}\ff\nn^\brdiio(\uu_\brdii)\,\ff\nn^\brdit(\ww_\brdiio)\Ldots\ww_1,
\end{equation}
with
$\len{\www}\le3\len{\ww_\brdio}+\Ldots+3\len{\ww_\brdii}+\len{\uu_\brdi}-\len{\uu_\brdii}-\brdi+\brdii$
and $\len{\uu_\brdii}\le\len{\uu_\brdi}+\brdi$.
\end{lemm}

\begin{proof}
The idea  is as follows:
using induction for $\kk$ going from $\brdi$ to $\brdiip$, we compute a
$\sig\nno$-nonnegative word~$\www_\kko$ and a
$\last{\ww}_\kk$-dangerous word~$\uu_\kko$ satisfying
\begin{equation}
\label{E:L:Upper:1}
\dddd1\nn^{-\kk+1}\,\ff\nn^\kko(\uu_\kk)\,\ff\nn^\kkt(\ww_\kko)\equiv\,\www_\kko\,\dddd1\nn^{-\kk+2}\,\ff\nn^\kkt(\uu_\kko).
\end{equation}
Then we define $\www$ to be $\www_\brdio\Ldots\www_\brdii$.

Let us go into details.
First we construct the words $\www_\kk$ and $\uu_\kk$.
Fix $\kk$ in $\{\brdi, \Ldots, \brdiip\}$ and assume that $\uu_\kk$ is an $\last{\ww}_\kk$-dangerous word.
Corollary~\ref{C:SplittingAndLadders} guarantees that $\ww_\kko$ is an
$\ff\nn(\last\ww_\kk)$-ladder lent on $\last{\ww}_\kko$. Then, by
Proposition~\ref{P:DangerousAgainstLadder}, the word
$\ff\nn(\uu_\kk)\,\ww_\kko$ is equivalent to a
$\ff\nn(\last{\ww}_\kk)$-wall~$\vv_\kko$ lent on $\last{\ww}_\kko$.
 By definition of a wall, we have 
\[
\vv_\kko=\vvv_\kko\,\dddd\indio\nno\,\uu_\kko, 
\]
where $\vvv_\kko$ is a $\sig\nnt$-nonnegative word, $\uu_\kko$ is an $\last{\ww}_\kko$-dangerous
word and $\aa\indio\nno$ being the
last letter of $\ww_\kko$. Then, we obtain
\[
 \dddd1\nn^{-\kk+1}\,\ff\nn^{\kko}(\uu_\kk)\,\ff\nn^\kkt(\ww_\kko)\equiv\dddd1\nn^{-\kk+1}\,\ff\nn^\kkt(\vvv_\kko\,\dddd\indio\nno\,\uu_\kko).
 \]
We push the power of $\dddd1\nn\inv$ to the last word between $\dddd\indio\nno$ and
$\uu_\kko$:
\[
\dddd1\nn^{-\kk+1}\,\ff\nn^\kkt(\vvv_\kko\,\dddd\indio\nno\,\uu_\kko)\equiv\ff\nn(\vvv_\kko\,\dddd\indio\nno)\,\dddd1\nn^{-\kk+1}\,\ff\nn^\kkt(\uu_\kko)\]
By Relation~\eqref{E:DeltaRelations:Decomposition}, we have $\ff\nn(\dddd\indio\nno)\,\dddd1\nn\inv\equiv\dddd1\indi\inv$.
Eventually, we obtain
\begin{equation}
\label{E:L:Upper:2}
  \dddd1\nn^{-\kk+1}\,\ff\nn^\kko(\uu)\,\ff\nn^\kkt(\ww_\kko)\equiv\ff\nn(\vvv_\kko)\,\dddd1\indi\inv\,\dddd1\nn^{-\kk+2}\,\ff\nn^\kkt(\uuu).
\end{equation}
Writing  $\www_\kko=\ff\nn(\vvv_\kko)\,\dddd1\indi\inv$, Relation \eqref{E:L:Upper:2} implies \eqref{E:L:Upper:1}.
By construction, $\www_\kko$ is $\sig\nno$-nonnegative and $\uu_\kk$ is an
$\last{\ww}_\kk$-dangerous word.

Gathering the relations \eqref{E:L:Upper:1} for $\kk$ form $\brdi$ to $\brdiip$, we obtain the relation \eqref{E:L:KeyLemma} for $\www=\www_\brdio\,\Ldots\www_\brdii.$

By construction, $\www_\kk$ is $\sig\nno$-nonnegative for each~$\kk$,
hence $\www$ is $\sig\nno$-nonnegative as well.

It remains to establish the complexity statements.
For every $\kk$ in $\{\brdi,\Ldots,\brdiip\}$, \eqref{P:DAL:ii} implies
$\len{\vv_\kko}\le3\len{\ww_\kko}+\len{\uu_\kk}-1$. Then, by construction
of $\www_\kko$, we have
$\len{\www_\kko}\le3\len{\ww_\kko}+\len{\uu_\kk}-\len{\uu_\kko}-1$.
We deduce
\[
\len{\www}=\len{\www_\brdio\Ldots\www_\brdii}\le3\len{\ww_\brdio
\Ldots\ww_\brdii}+\len{\uu_\brdi}-\len{\uu_\brdii}-\brdi+\brdii.
\]

For each $\kk$, \eqref{P:DAL:i} implies
$\len{\uu_\kk}\le\len{\uu_\kkp}\plus1$. Then, we find
$\len{\uu_\kk}\le\len{\uu_\brdi}+\brdi-\kk$, hence $\len{\uu_\kk}\le\len{\uu_\brdi}+\brdi$.
By Proposition~\ref{P:DangerousAgainstLadder}, computing $\uu_\kk$ and
$\vv_\kk$ from $\ww_\kk$ and $\uu_\kkp$ requires at most
$O(\len{\uu_\kkp}\len{\ww}+1)$ steps. Therefore, computing $\www$
and $\uu_\brdii$ from~$\uu_\brdi$ and the $\ff\nn$-splitting of $\ww$
requires at most $O((\len{\uu_\brdi}+\brdi)\len{\ww}+\brdi-\brdii)$ steps.
As $\brdi\le\len\ww+2$ holds, we deduce that computing $\www$ and
$\uu_\brdii$ from $\uu_\brdi$ and the $\ff\nn$-splitting of~$\ww$
requires at most $O(\len{\uu_\brdi}\len{\ww}^2)$ steps.
\end{proof}

%
%

\section{The main result}
\label{S:TheMainResult}

We are now ready to establish Theorems~$1$ and~$2$ of the
introduction. What we shall do is to construct, for each $\nn$-strand
braid~$\br$, a certain $ad$-word~$\nf\nn(\br)$ that represents~$\br$ and that is
$\sigg$-definite, \ie, is a word in the letters $\aa\indi\indii$ and
$\dddd\indi\indii$ which, translated to the alphabet of~$\sig\ii$, becomes
either $\sigg$-positive or $\sigg$-negative.

The construction of the word~$\nf\nn(\br)$ involves two steps. The first
(easy) step, described in Section~\ref{SS:RotatingNormalForm},  consists in
extending the rotating normal form of Section~\ref{S:RotatingNF} to all
of~$\BB\nn$ by appending convenient denominators. The process is based on
the Garside structure of the monoid~$\BKL\nn$.

The second step starts from the rotating normal form, and it
is described in Section~\ref{SS:TheSigmaDefiniteNormalForm}. The process
splits into three cases according to the relative position of two parameters
associated with~$\br$, namely the breadth of the numerator and the
exponent of the denominator in the rotating normal form of~$\br$. The 
reversing machinery developed in Sections~\ref{S:Reversing} and \ref{S:Walls}
is needed to treat the difficult case, which is the case when the above two
parameters are close one to the other.

%
%

\subsection{The rotating normal form of an arbitrary braid}
\label{SS:RotatingNormalForm}

As mentioned above, we first extend the rotating normal form, so far
defined only for those braids that belong to the monoid~$\BKL\nn$, to all
braids.

\begin{prop}
\label{P:GarsideQuotient}
Each braid $\br$ admits a unique expression
$\dddd1\nn^{-\tt} w$ where $\tt$ is a nonnegative integer, 
$\ww$ is a (rotating) normal word, and the braid $\wwt$ is not left-divisible by
$\dddd1\nn$ unless $\tt$ is zero.
\end{prop}

\begin{proof}
By Proposition~\ref{P:Garside}, the monoid~$\BKL\nn$ is a Garside monoid with
Garside element $\ddd\nn$, and the group~$\BB\nn$ is a group of fractions for
the monoid $\BKL\nn$. Hence, there exists a smallest integer $\tt$ such
that
$\ddd\nn^{\tt}\,\br$ belongs to the monoid $\BKL\nn$. If $\tt$ is positive,
the minimality hypothesis implies that $\ddd\nn$ is not a  left-divisor of
$\ddd\nn^{\tt}\,\br$. Taking for $\ww$ the rotating normal form of
$\ddd\nn^{-\tt}\,\br$ gives a pair $(\tt, \ww)$ of the expected form---we recall
that $\dddd1\nn\equiv\ddd\nn$ holds.

Assume that $(\tt', \www)$ is another pair with the above properties. 
Then $\ddd\nn^{\tt'}\,\br$ belongs to $\BKL\nn$, hence we
have $\tt' \ge \tt$. If we had $\tt' > \tt$, the hypothesis
$\ddd\nn^{-\tt}\,\wwt = \ddd\nn^{-\tt'} \wwwt$
would imply $\ddd\nn^{\tt'-\tt} \wwt = \wwwt$,
implying that $\wwwt$ is left-divisible by $\ddd\nn$,
which contradicts $\tt' > 0$. Hence we have $\tt' = \tt$, whence
$\www = \ww$ by uniqueness of the rotating normal form.
\end{proof}

\begin{defi}
The $ad$-word $\dddd1\nn^{-\tt} \ww$ involved in
Proposition \ref{P:GarsideQuotient} is called the $\nn$-rotating normal form of 
the braid $\br$.
The number $\tt$ is called the \emph{$\nn$-depth} of $\br$,
denoted~$\dpt\nn(\br)$; the number $\tt+\len{\ww}$, \ie, the length of the
$ad$-word $\dddd1\nn^{-\tt} \ww$, is called the
\emph{$\nn$-length} of~$\br$, denoted~$\lenr\nn\br$; finally, for
$\nn\ge3$, the $\nn$-breadth of~$\ww$ is called the \emph{$\nn$-breadth}
of $\br$, denoted~$\brd\nn(\br)$.
\end{defi}

By definition, the rotating normal form of a braid is
an $ad$-word, \ie, a word involving the letters $\aa\indi\indii$
and the letters $\dddd\indi\indii$ (actually the letter $\dddd1\nn\inv$
only). The terminology is coherent since, for $\br$ in~$\BKL\nn$, the
rotating normal form as defined above coincides with the rotating normal form
of Definition~\ref{D:RotatingNormalForm}: indeed, $\br$ belongs to
$\BKL\nn$ if and only if its $\nn$-depth is $0$.

Building on Proposition~\ref{P:AlgoC} and on the Garside structure of $\BKL\nn$, we easily see that
the rotating normal form of an arbitrary braid can be
computed in quadratic time.

\begin{lemm}
\label{L:DivisionGarsideElementByAtom}
For $\nn\ge3$ and $1 \le \ii \le \nno$, let $\theta_{\ii, \nn}$ be the $a$-word $\ff\nn^\iip(\ddd\nno)$.
Then  
$\theta_{\ii,\nn}$ is equivalent to~$\ddd\nn\,\siginv\ii$, and it has length~$\nnt$.
\end{lemm}

\begin{proof}
By Lemma~\ref{L:Rotation}, we have $\ff\nn^{\iip}(\aa\nno\nn)=\aa\ii\iip$. 
We deduce $\ff\nn^\iip(\ddd\nn)\equiv\ff\nn^\iip(\ddd\nno)\,\sig\ii
= \theta_{\ii,\nn}\,\sig\ii$. As $\ddd\nn$ is invariant under
$\ff\nn$, we have $\ddd\nn=\theta_{\ii,\nn}\,\sig\ii$. The length of the
$a$-word $\ddd\nno$ is $\nnt$. As $\ff\nn$ preserves the length of $a$-word,
the length of $\theta_{\ii,\nn}$ is $\nnt$.
\end{proof}

\begin{prop}
\label{P:RotatingNF}
For each $\nn$-strand braid $\br$, we have $\lenr\nn{\br}\le(\nno)\,\lens{\br}$ .
Moreover, if $\br$ is specified by a word of length $\ll$, the rotating normal
form of $\br$ can be computed in time $O(\ll^2)$.
\end{prop}

\begin{proof}
The case $\nn = 2$ is trivial.
Starting with a word on the alphabet~$\{\sig1,
\siginv1\}$, we freely reduce it to~$\sig1^\kk$ by deleting
the factors~$\sig1\siginv1$ and~$\siginv1\sig1$. The rotating normal
form is $\aa12^\kk$ in the case~$\kk \ge 0$, and $\dddd12^\kk$
in the case $\kk < 0$, and it is geodesic.

Assume now $\nn\ge3$. Let $\ww$ be an $\nn$-strand braid word representing $\br$.
Then the rotating normal form of $\br$ is obtained as follows:

- Replace each positive letter~$\sig\ii$ in~$\ww$ with $\aa\ii\iip$, so as to obtain
\[\uu = \ww_0 \siginv{\ii_1} \ww_1\Ldots\ww_\brdiio \siginv{\ii_\brdiio} \ww_\brdii;\]

- Put $\vv=\ff\nn^\brdii(\ww_0)\,
\ff\nn^\brdiio(\theta_{\ii_1,\nn} \,\ww_1)\,\Ldots\,
\ff\nn(\theta_{\ii_\brdiio ,\nn}\,\ww_\brdiio)\,\theta_{\ii_\brdii ,\nn}\,\ww_\brdii$;

- Let $\ss$ be the maximal integer such that $\ddd\nn^\ss$ left-divides~$\overline{\vv}$
in~$\BKL\nn$, and let $\vvv$ be a positive $a$-word satisfying $\vv \equiv \ddd\nn^\ss \vvv$;

- If $\ss \ge \brdii$ holds, put $\tt=0$ and
$\wwww=\ddd\nn^{\ss-\brdii}\,\vvv$; otherwise put 
$\tt=\brdii-\ss$ and $\wwww=\vvv$.

- Let $\www$ be the normal form of $\wwwwt$.
Then the rotating normal form of $\br$ is~$\dddd1\nn^{-\tt}\,\www$.

\noindent
Indeed, Lemma~\ref{L:DivisionGarsideElementByAtom} and $\ddd\nn\equiv\dddd1\nn$ imply 
\[
\ww\equiv\ww_0\,\dddd1\nn\inv\,\theta_{\ii_1,\nn}\,\ww_1\,\Ldots\,\dddd1\nn\inv\,\theta_{\ii_\brdiio,\nn}\,\ww_\brdiio\,\dddd1\nn\inv\,\theta_{\ii_\brdii,\nn}\,\ww_\brdii
\]
Pushing the letters $\dddd1\nn\inv$ to the left, we obtain
\[
 \ww\equiv\dddd1\nn^{-\brdii}\,\ff\nn^\brdii(\ww_0)\,
\ff\nn^\brdiio(\theta_{\ii_1,\nn} \,\ww_1)\,\Ldots\,
\ff\nn(\theta_{\ii_\brdiio ,\nn}\,\ww_\brdiio)\,\theta_{\ii_\brdii ,\nn}\,\ww_\brdii=\dddd1\nn^{-\brdii}\,\vv.
\]
Using the relation $\dddd1\nn\equiv\ddd\nn$ and the construction of $\www$, we obtain $\ww\equiv\dddd1\nn^{-\tt}\,\www$, where $\wwwt$ is not left-divisible by $\dddd1\nn$ unless $\tt$ is zero.

As for the length, replacing $\sig{\ii_\kk}$ by $\dddd1\nn\inv\,\theta_{\ii_\kk}$ multiplies it by at most~$\nno$.
Applying the construction in the case when $\ww$ is a shortest representative of~$\br$ gives $\lenr\nn\br\le(\nno)\lens\br$.

As for the  time complexity, $\vv$ is
obtained in time
$O(\ll)$, the integer $\ss$ is obtained in time
$O(\ll^2)$---see for instance \cite{EpsteinCanonHoltLevyPatersonThurston}---and $\www$ is obtained
in time
$O(\len{\wwww}^2)$ by Proposition~\ref{P:AlgoC}. Hence, as
$\len{\wwww}\le\ll$ holds, the rotating normal form of $\br$ is obtained from
the word $\ww$ in time $O(\ll^2)$.
\end{proof}

\begin{exam}
\label{X:RotatingNormalForm}
Consider $\br = \sig1\,\sigma_3^{-2}\,\sig2\,\sig3$. We use the notation of
Proposition~\ref{P:RotatingNF}. First, we write
$\uu=\ww_0\,\siginv3\,\ww_1\,\siginv3\,\ww_2$ with $\ww_0=\aa12$,
$\ww_1=\varepsilon$ and $\ww_2=\aa23\,\aa34$. Then we have
$\theta_{3,4} = \ff4^4(\ddd3) = \aa12\,\aa23$, and we find
\[
\vv=\ff4^2(\ww_0)\,\ff4(\theta_{3,4}\,\ww_1)\,\theta_{3,4}\,\ww_2=\aa34\ \aa23\,\aa34\ \aa12\,\aa23\,\aa23\,\aa34.
\]
The maximal power of $\ddd4$ that left-divides $\overline{\vv}$ is $1$ and we
have
$\vv \equiv \ddd4\,\aa23\,\aa12\,\aa23\,\aa24$. So we find $\ss=1$
and
$\vvv=\aa23\,\aa12\,\aa23\,\aa24$. Here we have $\brdii=2$ and $\ss=1$
hold, hence we put
$\tt=1$ and $\wwww=\aa23\,\aa12\,\aa23\,\aa24$. The rotating
normal form $\www$ of~$\overline{\wwww}$ turns out to be $\aa12\,\aa14\,\aa23\,\aa12$.
So, finally, the rotating normal form of~$\br$ is 
\[
 \dddd14\inv\,\aa12\,\aa14\,\aa23\,\aa12.
\]
Hence the $4$-depth of $\br$ is $1$, its length is $5$, and its $4$-breadth is
$4$, since we saw in Example~\ref{X:RotatingToSplitting} that the
$4$-breadth of $\aa12\,\aa14\,\aa23\,\aa12$ is~$4$: its $\ff4$-splitting is
$(\aa23, \aa23, 1, \aa23\,\aa12)$, a sequence of length~$4$.
\end{exam}

\subsection{The word $\nf\nn(\br)$: the easy cases}
\label{SS:TheSigmaDefiniteNormalForm}

Starting from the rotating normal form, we shall now
define for each braid $\br$ a new distinguished
representative $\nf\nn(\br)$ that is a $\sigg$-definite
word. The word~$\nf\nn(\br)$ will be constructed as a word on
the letters $\aa\indi\indii$ and $\dddd\indi\indii$. At the end, it will be
obvious to translate it into an ordinary braid word, \ie, a word on the
letters~$\sig\ii$.

The construction of $\nf\nn(\br)$ depends on the relative 
values of $\dpt\nn(\br)$ and $\brd\nn(\br)$.
The first case, which is easy, is when the $\nn$-depth of $\br$ is $0$, \ie,
when $\br$ belongs to~$\BKL\nn$, or it is $\lenr\nn\br$, \ie, when $\br$ is a
negative power of~$\dddd1\nn$. Note that this case is the only possible one in
the case of~$\BB2$.

\begin{defi}
\label{D:NF:VeryEasy}
Assume that $\br$ is a braid of $\BB\nn$ satisfying $\dpt\nn(\br) = 0$ or
$\dpt\nn(\br) = \lenr\nn\br$. Then we define $\nf\nn(\br)$ to be the
$\nn$-rotating normal form of $\br$.
\end{defi}


In this case, everything is clear.

\begin{prop}
\label{P:NF:VeryEasy}
Under the hypotheses of Definition~\ref{D:NF:VeryEasy}, the
word $\nf\nn(\br)$ is a 
$\sigg$-definite expression of $\br$, and its length
is at most $\lenr\nn\br$. Moreover, if $\br$
is specified by a $\sigg$-word of length $\ll$, the
word $\nf\nn(\br)$ can be computed in time $O(\ll^2)$.
\end{prop}

\begin{proof}
If $\lenr\nn\br$ is equal to $0$, then $\br$ is the trivial braid $1$ and its rotating normal
form is the empty word. If $\br$ is nontrivial with $\dpt\nn(\br) =0$, then the
rotating normal form is a nonempty positive $a$-word, \ie, a $\sigg$-positive
word. If $\br$ is nontrivial with $\dpt\nn(\br)=\lenr\nn\br$, then the rotating
normal form of $\br$ is $\dddd1\nn^{-\dpt\nn(\br)}$, which is
$\sig\nno$-negative. The complexity statements are clear from
Proposition~\ref{P:RotatingNF}.
\end{proof}

The second case, which is easy as well, is when the depth is 
large. We recall that, if $\ww$ is a normal word, then the $\ff\nn$-splitting of~$\ww$ is the sequence of normal words that represent the entries in the $\ff\nn$-splitting of the braid represented by~$\ww$.

\begin{defi}
\label{D:NF:Easy}
Assume that $\br$ is a nontrivial braid of $\BB\nn$ with $\nn \ge 3$ satisfying
$\dpt\nn(\br)\not=0$ and $\dpt\nn(\br) > \brd\nn(\br)-2$. Let
$\dddd1\nn^{-\tt}\,\ww$ be the rotating normal form of $\br$ and
$(\ww_\brdi,\Ldots,\ww_1)$ be the splitting of $\ww$. Then we put
$$\nf\nn(\br) = \dddd1\nn^{-\tt\plus\brdio}\cdot\ww_\brdi\,\dddd1\nn\inv\cdot
...\cdot\ww_2\,\dddd1\nn\inv\cdot\ww_1.$$
\end{defi}

\begin{prop}
\label{P:NF:Easy}
Under the hypotheses of Definition~\ref{D:NF:Easy}, the word $\nf\nn(\br)$ is
a  $\sigg$-negative expression of~$\br$, and its length
is at most $\lenr\nn\br$. Moreover, if $\br$
is specified by a $\sigg$-word of length~$\ll$, the
word $\nf\nn(\br)$ can be computed in time $O(\ll^2)$.
\end{prop}

\begin{proof}
First, we claim that $\nf\nn(\br)$ is an expression of $\br$.
Let $\dddd1\nn^{-\tt}\,\ww$ be the rotating normal form of $\br$ and $(\ww_\brdi,\Ldots,\ww_1)$ be the $\ff\nn$-splitting of $\ww$.
We have 
\begin{equation}
\label{E:P:NF:VeryNegative}
\dddd1\nn^{-\tt}\,\ww=\dddd1\nn^{-\tt}\cdot\ff\nn^\brdio(\ww_\brdi)\cdot ...\cdot\ff\nn(\ww_2)\cdot\ww_1.
\end{equation}
Pushing $\brdio$ powers of~$\dddd1\nn$ to the right in~\eqref{E:P:NF:VeryNegative} and dispatching them between the factors~$\ww_\kk$, we find
\begin{align*}
\dddd1\nn^{-\tt}\,\ww&=\dddd1\nn^{-\tt}\cdot\ff\nn^\brdio(\ww_\brdi)\cdot ...\cdot\ff\nn(\ww_2)\cdot\ww_1\\
&=\dddd1\nn^{-\tt\plus\brdio}\cdot\dddd1\nn^{\minus\brdip}\cdot\ff\nn^\brdio(\ww_\brdi)\cdot ...\cdot\ff\nn(\ww_2)\cdot\ww_1\\
&\equiv\dddd1\nn^{-\tt\plus\brdio}\cdot\ww_\brdi\cdot\dddd1\nn\inv\cdot\dddd1\nn^{\minus\brdi\plus2}\cdot ...\cdot\ff\nn(\ww_2)\cdot\ww_1\\
&\equiv \ ... \ 
\equiv\dddd1\nn^{-\tt\plus\brdio}\cdot\ww_\brdi\cdot\dddd1\nn\inv\cdot
...\cdot\ww_2\cdot
\dddd1\nn\inv\cdot\ww_1
= \nf\nn(\br).
\end{align*}

Next, exactly $\dpt\nn(\br)$ powers of $\dddd1\nn\inv$ occur in $\nf\nn(\br)$. 
Hence, as $\dpt\nn(\br)\not=0$,  at least one $\dddd1\nn\inv$ appears in $\nf\nn(\br)$. By construction, the
intermediate words~$\ww_\kk$ contain no letter~$\aa\indi\nn$. Therefore,
the word $\nf\nn(\br)$ is $\sig\nno$-negative.

As for the length, we find
\begin{align*}
\len{\nf\nn(\br)}&=\tt\minus\brdi\plus1\plus\len{\ww_\brdi}\plus1\plus...\plus\len{\ww_2}\plus1\plus\len{\ww_1}\\
&=\tt\minus\brdi\plus1\plus\len{\ww}\plus\brdio=\len{\dddd1\nn^{-\tt}\,\www}=\lenr\nn{\br}.
\end{align*}

Finally, assume that $\br$ is specified by a word of length $\ll$.
Then, by Proposition~\ref{P:RotatingNF}, we can compute the rotating normal
form of $\br$ in at most $O(\ll^2)$ steps. By Lemma~\ref{L:AlgoS}, computing
the $\ff\nn$-splitting of $\ww$ can be done in $O(\len{\ww})$ steps.
Hence, $\nf\nn(\br)$ can be computed in time~$O(\ll^2)$.
\end{proof}

\subsection{The word $\nf\nn(\br)$: the difficult case}
\label{SS:Difficult}

There remains the case of a braid~$\br$ satisfying  $\dpt\nn(\br)\not=0$ and $ \dpt\nn(\br) \le \brd\nn(\br) - 2$: this is the difficult
case. In this case, it is impossible to directly predict whether $\br$ has a $\sig\nno$-positive or
a $\sig\nno$-neutral expression, and this is the point
where we shall use the ladder and reversing machinery
developed in Sections~\ref{S:Ladders}, \ref{S:Reversing} and \ref{S:Walls}.

\begin{defi}
\label{D:NF:Hard}
Assume that $\br$ is a nontrivial braid of $\BB\nn$ with $\nn \ge 3$ satisfying
$\dpt\nn(\br)\not=0$ and $ \dpt\nn(\br) \le \brd\nn(\br) - 2$.
Let $\dddd1\nn^{-\tt}\,\ww$ be the rotating normal form
of~$\br$, and $(\ww_\brdi,\Ldots,\ww_1)$ be the
$\ff\nn$-splitting of $\ww$. Write $\ww_{\ttpp} = \www_{\ttpp}\,
\aa\indio\nno$. Put 
\[\vv = \ff\nn^{\brdi-1-\tt}(\ww_\brdi) \; ... \; \ff\nn^2(\ww_{\tt+3})\;\ff\nn(\www_{\ttpp})\; \dddd1\indi\inv,
\quad
\uu_\ttpp = \dddd\indio\nnt\inv.
\]

\noindent{\bf Case 1:} $\ww_2 \not= \varepsilon$. Then we put
\[
\nf\nn(\br) = \vv\, \wwww\, \ff\nn(\www_2)\, \ww_1,
\]
where $\wwww$ and $\uu_3$ are the words produced by Lemma~\ref{L:KeyLemma} 
applied to the sequence $(\ww_{\tt+2}, ..., \ww_1)$, the word~$\uu_{\tt+2}$ and the integer $3$, and where $\www_2$ is the word given by Proposition~\ref{P:DangerousAgainstLadder} applied to the words~$\ww_2$ and~$\ff\nn(\uu_3)$;

\noindent{\bf Case 2:} $\ww_2\,{=}\, \varepsilon$, $\ww_3 \,{=}\, ... \,{=}\,  \ww_{\kko} = \aa\nnt\nno$ and
$\ww_\kk \not= \aa\nnt\nno$ for some $\kk \le \tt+1$. Then we put
\[
\nf\nn(\br)=\vv\,\wwww\,\ff\nn(\www_\kk)\,\dddd1\nno^{-\brdii+2}\,\ww_1,
\]
where $\wwww$ and $\uu_\kkp$ are the words given by
Lemma~\ref{L:KeyLemma} applied to the sequence $(\ww_{\tt+2}, ..., \ww_1)$,
the word~$\uu_{\tt+2}$ and the integer $\kkp$, and where
$\www_\kk\,\aa\nnt\nno$ is the word produced by
Proposition~\ref{P:DangerousAgainstLadder} applied to the
words~$\ww_\kk$ and~$\ff\nn(\uu_\kkp)$;

\noindent{\bf Case 3:} $\ww_2 \,{=}\, \varepsilon$, $\ww_3 \,{=}\, ... \,{=}\, \ww_{\tt+1} \,{=}\,
\aa\nnt\nno$ and $\vv\not=\dddd1\nno\inv$. Then we put
\[
 \nf\nn(\br)=\vv\,\dddd1\nno^{-\tt+1}\,\ww_1;
\]

\noindent{\bf Case 4:} $\ww_2 \,{=}\, \varepsilon$, $\ww_3 \,{=}\, ... \,{=}\, \ww_{\tt+1} \,{=}\,
\aa\nnt\nno$ and $\vv=\dddd1\nno\inv$. Then we put
\[
 \nf\nn(\br)=\nf\nno(\ddd\nno^{-\tt}\,\wwt_1).
\]
\end{defi}

\begin{prop}
\label{P:NF:Hard}
Under the hypotheses of Definition~\ref{D:NF:Easy}, the word $\nf\nn(\br)$ is
a  $\sigg$-definite expression of~$\br$, and its length is at most
$3\,\lenr\nn\br$. Moreover, if $\br$ is specified by a $\sigg$-word of
length~$\ll$, the word~$\nf\nn(\br)$ can be computed in time $O(\ll^2)$. 
\end{prop}

\begin{proof}
We use the notation of Definition~\ref{D:NF:Hard}.
First, we claim that the following equivalence holds:
\begin{equation}
\label{E:Claim1}
\dddd1\nn^{-\tt}\,\ww\equiv\vv\ \dddd1\nn^{-\tt+1}\,\ff\nn^\ttp(\uu_\ttpp)\,
\ff\nn^\tt(\ww_\ttp)\,\Ldots\ff\nn(\ww_2)\,\ww_1.
\end{equation}
Indeed, as the sequence $(\ww_\brdi,\Ldots,\ww_1)$ is the $\ff\nn$-splitting of $\ww$, we
have
\begin{equation}
\label{E:AlgoA2p:1}
\dddd1\nn^{-\tt}\,\ww=\dddd1\nn^{-\tt}\,\ff\nn^\brdio(\ww_\brdi)\Ldots\ff\nn^\ttp(\ww_\ttpp)\Ldots\ff\nn(\ww_2)\,\ww_1.
\end{equation}
By construction, $\ww_\ttpp$ is $\www_\ttpp\,\aa\indio\nno$.
By~\eqref{E:DeltaRelations:DualGenerator}, we have $\aa\indio\nno \equiv
\dddd\indio\nno\,\uu_\ttpp$, hence $\ww_\ttpp
\equiv \www_\ttpp\,\dddd\indio\nno\,\uu_\ttpp$. Then, the word $\dddd1\nn^{-\tt}\,\ww$ is
equivalent to
\begin{equation}
\label{E:AlgoA2p:2}
\dddd1\nn^{-\tt}\,\ff\nn^\brdio(\ww_\brdi)\Ldots\ff\nn^\ttp(\www_\ttpp\,\dddd\indio\nno)\,\ff\nn^\ttp(\uu_\ttpp)\,\ff\nn^\tt(\ww_\ttp)\,\Ldots\ff\nn(\ww_2)\,\ww_1.
\end{equation}
We push the factor $\dddd1\nn^{-\tt}$ appearing in~\eqref{E:AlgoA2p:2} to the right, until
it arrives at the left of the factor~$\ff\nn^{\ttp}(\uu_\ttpp)$. In this way, we obtain
\[
\dddd1\nn^{-\tt}\,\ww\equiv\ff\nn^{\brdi-\tt-1}(\ww_\brdi)\Ldots\ff\nn(\www_\ttpp\,\dddd\indio\nno)\ \dddd1\nn^{-\tt}\,\ff\nn^\ttp(\uu_\ttpp)\,\ff\nn^\tt(\ww_\ttp)\,\Ldots\ff\nn(\ww_2)\,\ww_1.
\]
Relation~\eqref{E:DeltaRelations:Decomposition} and \eqref{E:DeltaRelations:ImageByAutomorphism}  imply 
$\ff\nn(\dddd\indio\nno)\,\dddd1\nn^{-\tt}\equiv\dddd1\indi\inv$. Inserting the latter value
in the relation above, we obtain~\eqref{E:Claim1}, as expected.

Next, by construction, the word~$\vv$ is $\sig\nno$-nonnegative, and its
length satisfies
\begin{equation}
\label{E:AlgoA2P:v:length}
\len{\vv}=\len{\ww_\brdi}+\Ldots+\len{\ww_\ttpp}.
\end{equation}

To go further, we consider the four cases of
Definition~\ref{D:NF:Hard} separately. In the first three cases, we shall show
that $\nf\nn(\br)$ is
$\sig\nno$-positive; in the fourth case, we shall show that $\nf\nn(\br)$ is $\sigg$-definite using an induction on $\nn$ and 
possibly Propositions~\ref{P:NF:VeryEasy} and~\ref{P:NF:Easy}.

\smallskip
\noindent{\bf Case 1.} First, $\nf\nn(\br)$ is equivalent to~$\dddd1\nn^{-\tt} \ww$.
Indeed, Lemma~\ref{L:KeyLemma} implies
\[
\dddd1\nn^{-\tt}\,\ww\,\equiv\vv\,\wwww\,\ff\nn^2(\uu_3)\,
\ff\nn(\ww_2)\,\ww_1,
\]
while Proposition~\ref{P:DangerousAgainstLadder} implies
$\ff\nn(\uu_3)\,\ww_2\equiv\www_2$. We deduce
\[
\dddd1\nn^{-\tt}\,\ww\,\equiv\vv\,\wwww\,\ff\nn(\www_2)\,\ww_1=
\nf\nn(\br).
\]

Next, by construction, $\www_2$ is a wall lent on~$\last{\ww}_2$, hence, by
definition,  it is $\sig\nnt$-positive. So $\ff\nn(\www_2)$ is $\sig\nno$-positive.
As $\vv$, $\wwww$ and $\ww_1$ are $\sig\nno$-nonnegative, $\nf\nn(\br)$ is
$\sig\nno$-positive.

As for the length, Lemma~\ref{L:KeyLemma} and Proposition~\ref{P:DangerousAgainstLadder} imply
\[
\label{E:AlgoA2p:C1:length}
\len{\wwww}\le3\len{\ww_\ttp}+\Ldots+3\len{\ww_3}-\len{\uu_3}-\tt+2,\quad
\len{\www_2}\le3\len{\ww_2}+\len{\uu_3}-1.
\]
Merging this values with \eqref{E:AlgoA2P:v:length}, and $\tt>0$, we deduce
$\len{\nf\nn(\br)}\le3\len{\ww}$.

\smallskip
\noindent{\bf Case 2.} First, we observe that the last letter
of~$\ww_\kk$ must be~$\aa\nnt\nno$: this  follows from 
Corollary~\ref{C:SplittingAndLadders2} since, by construction of $\kk$, the
word $\ww_\kko$ is either $\varepsilon$ or $\aa\nnt\nno$.

Now, we check that $\nf\nn(\br)$ is equivalent
to~$\dddd1\nn^{-\tt} \ww$. By Lemma~\ref{L:KeyLemma}, we have
\[
\dddd1\nn^{-\tt}\,\ww\,\equiv\vv\,\wwww\,\dddd1\nn^{-\kk+2}\,\ff\nn^{\kk}(\uu_\kkp)\,\ff\nn^\kko(\ww_\kk)\,\ff\nn^\kkt(\aa\nnt\nno)\Ldots\ff\nn^2(\aa\nnt\nno)\,\ww_1.
\]
By Proposition~\ref{P:DangerousAgainstLadder}, $\www_\kk$ is a $\ff\nn(\last{\ww}_\kkp)$-wall and it satisfies
$\ff\nn(\uu_\kkp)\,\ww_\kk\equiv\www_\kk\,\aa\nnt\nno$. Then, we have 
\begin{equation}
\label{E:AlgoA2p:C2:P}
\dddd1\nn^{-\tt}\,\ww\,\equiv\vv\,\wwww\,\dddd1\nn^{-\kk+2}\,\ff\nn^\kko(\www_\kk)\,\ff\nn^\kko(\aa\nnt\nno)\,\Ldots\ff\nn^2(\aa\nnt\nno)\,\ww_1.
\end{equation}
Pushing the negative powers of $\dddd1\nn$ appearing
in~\eqref{E:AlgoA2p:C2:P} to the right and dispatching them between the
$\ff\nn^{..}(\aa\nnt\nno)$, we find
\begin{align*}
\dddd1\nn^{-\tt}\,\ww\,&\equiv\vv\,\wwww\,\ff\nn(\www_\kk)\,\dddd1\nn^{-\kk+2}\,\ff\nn^\kko(\aa\nnt\nno)\,\Ldots\ff\nn^2(\aa\nnt\nno)\,\ww_1\\
&\equiv\vv\,\wwww\,\ff\nn(\www_\kk)\,\ff\nn(\aa\nnt\nno)\,\dddd1\nn\inv\,\dddd1\nn^{-\kk+3}\,\Ldots\ff\nn^2(\aa\nnt\nno)\,\ww_1\\
&\equiv\  \Ldots\ \equiv
\vv\,\wwww\,\ff\nn(\www_\kk)\,\ff\nn(\aa\nnt\nno)\,\dddd1\nn\inv\,\Ldots\ff\nn(\aa\nnt\nno)\,\dddd1\nn\inv\,\ww_1.
\end{align*}
Then, $\ff\nn(\aa\nnt\nno)\,\dddd1\nn\inv\equiv\dddd1\nno\inv$ implies
\[
 \dddd1\nn^{-\tt}\,\ww\equiv\vv\,\wwww\,\ff\nn(\www_\kk)\,\dddd1\nno^{-\kk+2}\,\ww_1=\nf\nn(\br).
\]

Next, by construction, $\www_\kk$ is a $\ff\nn(\last{\ww}_\kkp)$-wall, hence, by definition, it is $\sig\nnt$-positive. 
So $\ff\nn(\www_\kk)$ is $\sig\nno$-positive. As $\vv$, $\wwww$, and~$\dddd1\nno^{-\kk+2}\,\ww_1$, are
$\sig\nno$-nonnegative, the word~$\nf\nn(\br)$ is $\sig\nno$-positive.

As for the length, Lemma~\ref{L:KeyLemma} and Proposition~\ref{P:DangerousAgainstLadder} imply
\[
\len{\wwww}\le3\len{\ww_\ttp}+\Ldots+3\len{\ww_3}-\len{\uu_\kkp}-\tt+2,\quad
\len{\www_\kk\,\aa\nnt\nno}\le3\len{\ww_\kk}+\len{\uu_\kkp}-1.
\]
Merging these values with~\eqref{E:AlgoA2P:v:length} and the
hypothesis $\tt>0$, we find
$\len{\nf\nn(\br)}\le3\len{\ww}$.

\smallskip
\noindent {\bf Case 3.}
As above, we observe that the last letter of $\ww_\ttpp$ is $\aa\nnt\nno$,
which follows from Corollary~\ref{C:SplittingAndLadders2}, since $\ww_\ttp$ is either $1$ or $\aa\nnt\nno$.

Then, we check that $\nf\nn(\br)$ is equivalent to $\dddd1\nn^{-\tt}\,\ww$.
As the last letter of
$\ww_\ttpp$ is $\aa\nnt\nno$, the word $\uu_\ttpp$ is empty.
Then, we find
\begin{equation}
\label{E:AlgoA2p:C3:P} 
\dddd1\nn^{-\tt}\,\ww\,\equiv\vv\,\dddd1\nn^{-\tt+1}\,\ff\nn^\tt(\aa\nnt\nno)\,\Ldots\ff\nn^2(\aa\nnt\nno)\,\ff\nn(\varepsilon)\,\ww_1.
\end{equation}
Pushing again the negative powers of $\dddd1\nn$ of~\eqref{E:AlgoA2p:C3:P}
to the right and dispatching them between the $\ff\nn^{..}(\aa\nnt\nno)$, we find
\begin{align*}
\dddd1\nn^{-\tt}\,\ww\,&\equiv\vv\,\ff\nn(\aa\nnt\nno)\,\dddd1\nn\inv\,\dddd1\nn^{-\tt+2}\,\ff\nn^\tto(\aa\nnt\nno)\,\Ldots\ff\nn^2(\aa\nnt\nno)\,\ww_1\\
&\equiv \ \Ldots \ \equiv
\vv\,\ff\nn(\aa\nnt\nno)\,\dddd1\nn\inv\,\Ldots\ff\nn(\aa\nnt\nno)\,\dddd1\nn\inv\,\ww_1.
\end{align*}
Then, $\ff\nn(\aa\nnt\nno)\,\dddd1\nn\inv\equiv\dddd1\nno\inv$ implies
\[
 \dddd1\nn^{-\tt}\,\ww\equiv\vv\,\dddd1\nno^{-\tt+1}\,\ww_1=\nf\nn(\br)
\]

Next, we check that $\nf\nn(\br)$ is $\sig\nno$-positive.
As $\last{\ww}_\ttpp=\aa\nnt\nno$ holds, we have 
\[\vv=\ff\nn^{\brdi-1-\tt}(\ww_\brdi) \; ... \; \ff\nn^2(\ww_{\tt+3})\;\ff\nn(\wwww_{\ttpp})\; \dddd1\nno\inv\]
By Lemma~\ref{L:LastLetter}$(iii)$, if the word $\www_\ttpp$ is not empty, it ends with a letter of the form~$\aa{..}\nno$, hence the word $\vv$ is $\sig\nno$-positive.
Assume that $\www_\ttpp$ is empty and $\tt\le\brdi-3$ holds.
As the word $\ww_\ttpp$ is $\aa\nnt\nno$, Corollary~\ref{C:SplittingAndLadders2} implies that 
$\ww_{\tt+3}$ ends with $\aa\nnt\nno$. Then, $\vv$ ends
with $\ff\nn^2(\aa\nnt\nno)\,\dddd1\nno\inv$, which is
$\aa1\nn\,\dddd1\nno\inv$, hence $\vv$ is $\sig\nno$-positive. 

Relation~\eqref{E:AlgoA2P:v:length} directly implies $\len{\nf\nn(\br)}=\len{\ww}$.

\smallskip
\noindent {\bf Case 4.}
By construction, we have $\vv = \dddd1\nno\inv$.
The same analysis as in Case~3 gives $\tt=\brdit$ and
\[
 \dddd1\nn^{-\tt}\,\ww\equiv\dddd1\nno^{-\tt}\,\ww_1,
\]
The induction hypothesis together with Propositions~\ref{P:NF:VeryEasy} and~\ref{P:NF:Easy} gives
$\dddd1\nno^{-\tt}\,\ww_1\equiv\nf\nno(\ddd\nno^{-\tt}\,\overline{\ww_1})$, hence
$\dddd1\nn^{-\tt}\,\ww\equiv\nf\nn(\br)$ by definition.

Always by induction hypothesis and Propositions~\ref{P:NF:VeryEasy} and~\ref{P:NF:Easy}, we have 
\[
\len{\nf\nn(\br)}=\len{\nf\nno(\ddd\nno^{-\tt}\,\overline{\ww_1})}\le3\lenr\nno{\ddd\nno^{-\tt}\,\overline{\ww_1}}.
\]
By definition, we have $\lenr\nn\br=\tt+\len{\ww_\brdi}+\Ldots+\len{\ww_1}$ and $\lenr\nno{\ddd\nno^{-\tt}\,\overline{\ww_1}}\le\tt+\len{\ww_1}$, hence $\lenr\nno{\ddd\nno^{-\tt}\,\overline{\ww_1}}\le\lenr\nn\br$. 
Then, as we obtain $\len{\nf\nn(\br)}\le3\lenr\nn{\br}$.

\smallskip
So all cases have been considered, and it only remains to analyze the time
complexity. By Proposition~\ref{P:RotatingNF} and Lemma~\ref{L:AlgoS}, the
rotating normal form of $\br$ and the $\ff\nn$-splitting of
$\ww$ can be computed in time $O(\ll^2)$. Then, in Cases~$1$ and~$2$,
Lemma~\ref{L:KeyLemma} is used once for  $(\ww_\ttpp,\Ldots,\ww_1)$ and
$\uu_\ttpp$, with a cost~$O(\ll^2)$. In addition,
Proposition~\ref{P:DangerousAgainstLadder} is used at most once with
$\ff\nn(\uu_\kkp)$  and $\ww_\kk$ ($\kk=2$ for Case 1), with a cost at
most $O(\max(1,\len{\uu_\kkp}\ll))$. Lemma~\ref{L:KeyLemma} guarantees
$\len{\uu_\kkp}\le\len{\uu_\kkp}+\tt+1-\brdii$, \ie, $\len{\uu_\ttpp}\le\tt$. So 
the total cost entailed by Proposition~\ref{P:DangerousAgainstLadder} is at
most $O(\ll^2)$. The other computations in Cases~$1$, $2$, and $3$ require at
most $O(\ll)$ steps and, therefore, the total cost of the computation
of~$\nf\nn(\br)$ is~$O(\ll^2)$ in Cases~$1$, $2$ and and~$3$. The result is
similar for Case~$4$, using the induction hypothesis, and possibly
Propositions~\ref{P:NF:VeryEasy} and~\ref{P:NF:Easy}.
\end{proof}

\subsection{Putting things together}

Using the $\sigg$-definite words~$\nf\nn(\br)$ constructed in Sections~\ref{SS:TheSigmaDefiniteNormalForm} and~\ref{SS:Difficult}, we are
now ready to establish Theorems~1 and~2 of the introduction. As a
preliminary remark, we observe that the words~$\nf\nn(\br)$ do not really
depend on the index~$\nn$.

\begin{lemm}
If $\br$ belongs to~$\BB\nno$, the words $\nf\nn(\br)$ and
$\nf\nno(\br)$ coincide.
\end{lemm}

\begin{proof}
An easy verification shows that, if $\br$ belongs to~$\BB\nno$, then either
we have $\dpt\nn(\br) = 0$ (if $\br$ belongs to~$\BKL\nno$), or we are in 
Case~4 of Definition~\ref{D:NF:Hard}. In both cases, the definition
of~$\nf\nn(\br)$ implies $\nf\nn(\br)=\nf\nno(\br)$.
\end{proof}

So, from now on, we can skip the subscript~$\nn$ and write $\nf{}(\br)$
without ambiguity. The main result, of which Theorems~1 and~2 are easy
consequences, is as follows. We recall that, for $\br$ a braid, $\lens\br$
denotes the length of the shortest expression of~$\br$ in terms of the Artin
generators~$\sig\ii$.

\begin{thrm}
\label{T:Main}
For each $\nn$-strand braid~$\br$, the $ad$-word~$\nf{}(\br)$ is a
$\sigg$-definite representative of~$\br$, and its length is at most
$3\,(\nno)\,\lens\br$. Moreover, if $\br$
is specified by a $\sigg$-word of length $\ll$, the
word $\nf{}(\br)$ can be computed in time $O(\ll^2)$.
\end{thrm}

\begin{proof}
Everything is obvious in the case $\nn = 2$, so we assume $\nn \ge 3$.
According to Proposition~\ref{P:GarsideQuotient}, and, according to the
case, Proposition~\ref{P:NF:VeryEasy}, \ref{P:NF:Easy}, or \ref{P:NF:Hard}, the
word~$\nf{}(\br)$ is, in any case, a $\sigg$-definite representative of~$\br$,
and its length is at most $3\lenr\nn\br$. On the other hand,
Proposition~\ref{P:RotatingNF} implies
$\lenr\nn\br\le(\nno)\lens{\br}$, so we deduce the expected upper bound
\begin{equation}
\label{E:Length}
\len{\nf{}(\br)}\le3(\nno)\lens\br.
\end{equation}
Finally, gathering the complexity analysis of Propositions~\ref{P:RotatingNF}, 
\ref{P:NF:VeryEasy}, \ref{P:NF:Easy}, and \ref{P:NF:Hard} shows that, in all
cases, $\nf{}(\br)$ can be computed in $O(\ll^2)$ steps when
$\br$ is specified by an initial word of length~$\ll$.
\end{proof}

As promised, we can now deduce Theorems~1 and~2 in a few words.

\begin{proof}[Proof of Theorem~$1$]
Let $\NF{}(\br)$ be the translation of the $ad$-word~$\nf{}(\br)$
into a $\sigg$-word. The formulas of~\eqref{E:Translation} show that the
translation of a letter~$\aa\indi\indii$ or~$\dddd\indi\indii$ with $\indii \le
\nn$ has length at most~$2\nn-3$. So 
\eqref{E:Length} implies
$\len{\NF{}(\br)}\le6(\nno)^2\lens\br$.
\end{proof}

\begin{proof}[Proof of Theorem~$2$]
Translating $\nf{}(\br)$ into~$\NF{}(\br)$ has a linear time cost, so
the quadratic upper bound for the computation of~$\nf{}(\br)$
established in Theorem~\ref{T:Main} immediately gives a quadratic upper
bound for the computation of~$\NF{}(\br)$.

A non-empty $\sigg$-definite braid word is never trivial, so computing the
word~$\nf{}(\br)$ solves in particular the word problem of~$\BB\nn$, which
is known to have a quadratic complexity exactly for $\nn \ge 3$. Hence the
above quadratic upper bound is sharp.
\end{proof}

Let us now give a concrete example of the previous constructions.

\begin{exam}
\label{X:Final}
We consider the braid $\br=\sig1\,\sigma_3^{-2}\,\sig2\,\sig3$ of
Example~\ref{X:RotatingNormalForm} again. We saw above that its rotating
normal form is the $ad$-word
 \[\dddd14\inv\,\aa12\,\aa14\,\aa23\,\aa12.\]
We saw in Example~\ref{X:RotatingToSplitting} that the
$\ff4$-splitting of $\aa12\aa14\aa23\aa12$ is $(\ww_4, ..., \ww_1)$, with
\[\ww_4 = \aa23,
\quad \ww_3 = \aa23,
\quad \ww_2 = \varepsilon,
\quad \mbox{and} \quad
\ww_1 = \aa23 \aa12.\]
So we have $\dpt4(\br) = 1$
and $\brd4(\br) = 4$, hence $\dpt4(\br)\le\brd4(\br)-2$, and we are in the
difficult case.
With the notation of Definition~\ref{D:NF:Hard}, we have $\tt = 1$ and $\ww_3=\varepsilon\cdot\aa23$, so we first put $\www_3=\varepsilon$, $\indi=3$, $\vv=\ff4^2(\ww_4)\,\ff4(\www_3)\,\dddd1\indi\inv$, and $\uu_3=\dddd\indio2\inv$,
\ie, in the current case, $\vv = \aa14\,\dddd13\inv$ and $\uu_3 = \varepsilon$.
Then, as we have $\ww_2=\varepsilon$, $\ww_3=\aa23$ and $\vv\not=\dddd13\inv$, we are in Case~$3$ of Definition~\ref{D:NF:Hard}. According to the latter, we define
$\nf{}(\br) = \vv\,\dddd13^{0}\,\ww_1$, \ie, 
$\nf{}(\br) =  \aa14\,\dddd13\inv\,\aa23\,\aa12$.
This $ad$-word is $\sig3$-positive: indeed, its $\sigg$-translation is the
$\sigg$-word
\[\NF{}(\br) = \sig1\,\sig2\,\sig3\,\siginv2\,\siginv1\,\siginv2\,\siginv1\,\sig2\,\sig1\] 
which contains one~$\sig3$, but no~$\siginv3$, and no $\sig\ii^{\pm1}$ with $\ii \ge 4$.
\end{exam}

In the very simple case of Example~\ref{X:Final}, the reversing machinery is not used (and directly guessing a
$\sigg$-definite word equivalent to the initial word  would have be
easy). However, much more complicated phenomena may occur in general,
in particular when the braid index reaches~$5$, which is the smallest value for
which there exist ladders with more than one bar. All situations
considered in Definition~\ref{D:NF:Hard} may occur when the
length and the braid index increase, and explicit examples can easily be found
using a computer. The examples witnessing really complicated
behaviors, typically requiring more than one reversing step, involve words that are too long to be given here. However their existence
confirms the really amazing intricacy of the braid relations.

%
%

\section*{Acknowledgment}
The author wishes to thank Patrick Dehornoy for his help in writing the paper.

%
%

\bibliographystyle{ams-pln} \bibliography{Bibliography}

\end{document}